\documentclass[a4paper,reqno,12pt]{amsart}

\usepackage{amsmath, amsfonts, amssymb, amsthm, amscd}
\usepackage{perpage}
\usepackage{url}
\usepackage{color}
\usepackage{framed} % For framed and shaded environments (see below)
\usepackage{mathrsfs}
\usepackage{tikz}
\usepackage{mathabx}
\usetikzlibrary{arrows,decorations.pathmorphing,backgrounds,positioning,fit,petri} % for TIKZ
\usepackage{tikz,tikz-cd}\usepackage{caption,subcaption}

\usepackage{dsfont} % nice indicator function symbol (blackboard 1)

\usepackage[utf8]{inputenc}
\usepackage[T1]{fontenc}
\usepackage{microtype}

\usepackage[a4paper,scale={0.72,0.74},marginratio={1:1},footskip=7mm,headsep=10mm]{geometry}

\usepackage{hyperref}
\usepackage{cleveref}

\makeatletter
\def\@secnumfont{\bfseries\scshape}

\def\section{\@startsection{section}{1}%
  \z@{.8\linespacing\@plus\linespacing}{.5\linespacing}%
  {\normalfont\large\bfseries\scshape\centering}}

\def\subsection{\@startsection{subsection}{2}%
  \z@{.5\linespacing\@plus.7\linespacing}{-.5em}%
  {\normalfont\bfseries\scshape}}

\def\subsubsection{\@startsection{subsubsection}{3}%
  \z@{.5\linespacing\@plus.7\linespacing}{-.5em}%
  {\normalfont\scshape}}

\def\specialsection{\@startsection{section}{1}%
  \z@{\linespacing\@plus\linespacing}{.5\linespacing}%
  {\normalfont\centering\large\bfseries\scshape}}
\makeatother

\numberwithin{equation}{section}

\definecolor{shadecolor}{gray}{.94}

\makeatletter
\newenvironment{myshade}{%
  \topsep4\p@\@plus4\p@\relax%
  \MakeFramed{\advance\hsize-\width \FrameRestore}}%
 {\par\unskip\endMakeFramed}%
\makeatother

\newtheoremstyle{mytheorem}{0}{0}%
     {\itshape}%         Body font
     {}%         Indent amount (empty = no indent, \parindent = para indent)
     {\bfseries}% Thm head font (e.g. \bfseries, \scshape, \sffamily)
     {. }%        Punctuation after thm head
     {0.3ex}%     Space after thm head (\newline = linebreak)
     {\thmname{{\bfseries #1}}\thmnumber{ {\bfseries #2}}\thmnote{ (#3)}}  % Thm head spec

\theoremstyle{mytheorem}

\newtheorem{theo}{Theorem}[section]
\newenvironment{theorem}{\begin{myshade}\begin{theo}}{\end{theo}\end{myshade}}

\newtheorem{prop}[theo]{Proposition}
\newenvironment{proposition}{\begin{myshade}\begin{prop}}{\end{prop}\end{myshade}}

\newtheorem{lem}[theo]{Lemma}
\newenvironment{lemma}{\begin{myshade}\begin{lem}}{\end{lem}\end{myshade}}

\newtheorem{defin}[theo]{Definition}
\newenvironment{definition}{\begin{myshade}\begin{defin}}{\end{defin}\end{myshade}}

\newtheorem{cor}[theo]{Corollary}

\newtheoremstyle{mydefinition}{.7\linespacing\@plus.3\linespacing}{.7\linespacing\@plus.3\linespacing}%
     {\rmfamily}%         Body font
     {}%         Indent amount (empty = no indent, \parindent = para indent)
     {\bfseries}% Thm head font (e.g. \bfseries, \scshape, \sffamily)
     {. }%        Punctuation after thm head
     {0.3ex}%     Space after thm head (\newline = linebreak)
     {\thmname{{\bfseries #1}}\thmnumber{ {\bfseries #2}}\thmnote{ (#3)}}  % Thm head spec

\theoremstyle{mydefinition}

\newtheorem{example}[theo]{Example}
\newtheorem{remark}[theo]{Remark}

\newenvironment{myenumerate}{%
\renewcommand{\theenumi}{\arabic{enumi}}%
\renewcommand{\labelenumi}{{\rm(\theenumi)}}%
\begin{list}{\labelenumi}
	{%
	\setlength{\itemsep}{0.4em}%
	\setlength{\topsep}{0.5em}%
	\setlength\leftmargin{2.45em}%
	\setlength\labelwidth{2.05em}%
	\setlength{\labelsep}{0.4em}%
	\usecounter{enumi}%
	}%
	}%
{\end{list}
}

{\end{list}
}

{\end{list}
}

\renewenvironment{enumerate}{
\begin{myenumerate}}%
{\end{myenumerate}}

\newenvironment{myitemize}{%
\begin{list}{$\bullet$}%
 	{%
	\setlength{\itemsep}{0.4em}%
	\setlength{\topsep}{0.5em}%
	\setlength\leftmargin{2.65em}%
	\setlength\labelwidth{2.65em}%
	\setlength{\labelsep}{0.4em}%
	}%
	}%
{\end{list}}

\renewenvironment{itemize}{
\begin{myitemize}}%
{\end{myitemize}}

\MakePerPage[2]{footnote} % restarts footnote counter at every new page

\newcommand{\R}{\mathbb{R}}

\newcommand{\N}{\mathbb{N}}
\newcommand{\Z}{\mathbb{Z}}

\newcommand{\cB}{{\mathcal B}}
\newcommand{\cC}{{\mathcal C}}
\newcommand{\cD}{{\mathcal D}}

\newcommand{\cM}{{\mathcal M}}

\newcommand{\cR}{{\mathcal R}}

\renewcommand{\epsilon}{\varepsilon}
\newcommand{\co}{\mathrm{coh}}
\newcommand{\ho}{\mathrm{hom}}

\newcommand{\balpha}{\boldsymbol{\alpha}}
\newcommand{\bbeta}{\boldsymbol{\beta}}

\newcommand{\vertiii}[1]{{\left\vert\kern-0.25ex\left\vert\kern-0.25ex\left\vert #1 
    \right\vert\kern-0.25ex\right\vert\kern-0.25ex\right\vert}}

\def\d{\, \mathrm{d}}
\def\and{\ and }

\date{\today}

\title[Reconstruction Theorem without Regularity Structures]{Hairer's Reconstruction Theorem\\
without Regularity Structures}

\author[F. Caravenna]{Francesco Caravenna}
\address{Dipartimento di Matematica e Applicazioni\\
 Universit\`a degli Studi di Milano-Bicocca\\
 via Cozzi 55, 20125 Milano, Italy}
\email{francesco.caravenna@unimib.it}

\author[L. Zambotti]{Lorenzo Zambotti}
\address{Laboratoire de Probabilit\'es, Statistique et Mod\'elisation, 
Sorbonne Universit\'e,  Universit\'e de Paris, CNRS,
4 Place Jussieu, 75005 Paris, France}
\email{zambotti@lpsm.paris}


\begin{document}

\begin{abstract}
This survey is devoted to Martin Hairer's Reconstruction Theorem, 
which is one of the cornerstones of his theory
of Regularity Structures~\cite{Hairer2014d}. 
Our aim is to give a new self-contained and elementary proof of this Theorem, together
with some applications, including a characterization,
based on a single arbitrary test function, of negative H\"older spaces.
We present the Reconstruction Theorem as a general result in the theory 
of distributions that can be understood without
any knowledge of Regularity Structures themselves, which we do not even need to define.
\end{abstract}

\maketitle

%\begin{classification}
%46F10; 60L30
%\end{classification}

%\begin{keywords}
%Distributions, Reconstruction Theorem, Regularity Structures
%\end{keywords}

\section{Introduction}

Consider the following problem: 
if at each point $x\in\R^d$ we are given a distribution
(generalized function) $F_x$ on $\R^d$, is there a distribution $f$ on $\R^d$ which
is well approximated by $F_x$ around
each point $x\in\R^d$? 

A classical example is when $f:\R^d\to\R$ is a smooth function
and $F_x$ is the Taylor polynomial of $f$ based at $x$, of some
fixed order $r\in\N$; 
then we know that $f(y)-F_x(y)$ is of order $|y-x|^{r+1}$ for $y\in\R^d$ close to $x$. 
Of course, in this example $F_x$ is built from $f$, which is known in advance.
We are rather interested in the reverse problem
of finding $f$ given a (suitable) family of $F_x$'s,
as in Whitney's Extension Theorem \cite{Whi34}.
However if we allow the local descriptions $F_x$ to be non-smooth and even distributions, 
then existence and uniqueness of such $f$ become non-trivial.

Martin Hairer's Reconstruction Theorem~\cite{Hairer2014d}
provides a complete and elegant solution to this problem.
We present here an enhanced version of this result
which allows to prove existence and uniqueness of $f$ under an \emph{optimal} assumption
on the family of distributions $(F_x)_{x\in\R^d}$, that we call \emph{coherence}.
We also present some applications of independent interest, including a characterization
of negative H\"older spaces based on a single \emph{arbitrary} test function.

The Reconstruction Theorem
was originally formulated in the framework
of Hairer's theory of \emph{regularity structures}~\cite{Hairer2014d}. In this survey
we state and prove this result without any reference to regularity structures,
which we do not even define. 
The original motivation for this theory was stochastic analysis,
but here we present the Reconstruction Theorem
in a completely analytical and deterministic framework.
Our approach contains novel ideas and techniques which 
may be generalized to other settings, e.g.\ to distributions on manifolds.

Although regularity structures have already attracted 
a lot of attention, we hope that this survey will give the 
opportunity to an even larger audience to become familiar with some of the ideas of this theory, 
which may still find applications outside the area which
motivated it first.

\subsection*{A look at the literature}

With his theory of Rough Paths \cite{Lyons1998}, Terry Lyons
introduced the idea of a local description
of the solution to a stochastic differential equation as a generalized Taylor expansion,
where classical monomials are replaced
by iterated integrals of the driving Brownian motion. 
This idea led Massimiliano Gubinelli to introduce his Sewing Lemma \cite{Gubinelli2004}, which is 
a version of the Reconstruction Theorem in~$\R^1$ (the name ``Sewing Lemma'' is 
actually
due to Feyel and de La Pradelle \cite{FeDe06}, who gave the proof which is now commonly used).
With his theory of regularity structures~\cite{Hairer2014d}, Martin Hairer translated 
these techniques in the context of  \emph{stochastic partial differential equations} (SPDEs),
whose solutions are  defined on $\R^d$ with $d>1$
(see \cite{Zambotti20} for a history of SPDEs).

The first proof of the Reconstruction Theorem was based on wavelets \cite{Hairer2014d}. 
Later Otto-Weber \cite{ow19}
proposed a self-contained approach based on semigroup methods. 
The core of our proof is based on
elementary multiscale arguments, which allow to characterize the regularity
of a distribution via scaling of a \emph{single arbitrary test function}.
The second edition of Friz-Hairer's book \cite{Friz2020}
contains a proof close in spirit to the one presented here.
For other proofs of versions of the Reconstruction Theorem, see \cite{GIP15,HL17,mw18,ST18}.

\subsection*{Outline of the paper}

In \Cref{sec:notation} we set the notation used throughout this survey
and in \Cref{sec:bg-distributions} we recall basic facts on test functions and distributions.

In \Cref{sec:germs} we define the key notion of \emph{germs of distributions} and the 
property of \emph{coherence}. This leads directly to
the Reconstruction Theorem in \Cref{sec:reco}, see \Cref{th:reco+}.
We then show in \Cref{sec:necessity} that the coherence condition is \emph{optimal}.

The core of the paper,
from \Cref{sec:conv} to \Cref{sec:reco+neg},
is devoted to the proof of the Reconstruction Theorem
(see \Cref{sec:guide} for a guide).

The last sections
are devoted to applications of the Reconstruction Theorem.
In \Cref{sec:negHolder} we study negative H\"older spaces, providing criteria
based on a single arbitrary test function, see \Cref{th:charHolder}.
In \Cref{sec:coherent+} we investigate more closely the coherence condition.
In \Cref{sec:Young} we construct a suitable product between distributions and non smooth
functions, see \Cref{th:Young},
which is a multi-dimensional analogue of Young integration.

\subsection*{Acknowledgements}

We are very grateful to Massimiliano Gubinelli 
for many inspiring discussions and for suggesting the name \emph{coherence}.
We also thank Malek Abdesselam, Florian Bechtold, Carlo Bellingeri, Lucas Broux,
Tommaso Cornelis Rosati, Henri Elad-Altman, Peter Friz, Martin Hairer, Cyril Labb\'e, 
David Lee, Sylvie Paycha, Nicolas Perkowski, Hendrik Weber
for precious feedback on earlier versions of this manuscript.

\section{Notation}
\label{sec:notation}

We work on the domain $\R^d$, equipped with the Euclidean norm $|\cdot|$.
We denote by $B(x,r) = \{z \in \R^d: \ |z-x| \le r\}$ 
the closed ball centered at $x$ of radius $r$.
The $R$-enlargement of a set $K \subseteq \R^d$ is denoted by
\begin{equation} \label{eq:enlargement}
	\bar{K}_R := K + B(0,R) = \{z\in\R^d: \ |z-x| \le R \text{ for some } x \in K
	\} \,.
\end{equation}
Partial derivatives of a differentiable function $\varphi: \R^d \to \R$ are denoted by
\begin{equation*}
	\partial^k \varphi = \partial_{x_1}^{k_1} \cdots \partial_{x_d}^{k_d} \varphi
	\qquad \text{for a multi-index } k = (k_1, \ldots, k_d) \in \N_0^{d} \,,
\end{equation*}
where $\N_0 = \{0,1,2,\ldots\}$ and we set $|k| := k_1 + \ldots + k_d$.
If $k_i = 0$ then $\partial^{k_i}_{x_i}\varphi := \varphi$.

For functions $\varphi: \R^d \to \R$ we use the standard notation
\begin{equation*}
	\|\varphi\|_\infty := \sup_{x\in {\R^d}} \, |\varphi(x)| \,.
\end{equation*}

We denote by $C^r$, for $r \in \N_0\cup\{ \infty\}$, the space of
functions $\varphi: \R^d \to \R$ which admit
continuous derivatives $\partial^k \varphi$ for every multi-index $k$ with $|k| \le r$.
We set
\begin{equation} \label{eq:normCr}
	\|\varphi\|_{C^r} := \max_{|k| \le r} \| \partial^k \varphi \|_\infty \,.
\end{equation}

We denote by $\cC^\alpha$, for $\alpha>0$, the space of
\emph{locally $\alpha$-H\"older functions} $\varphi :\R^d\to\R$.
More explicitly, $\varphi \in \cC^\alpha$ means that:
\begin{enumerate}
\item $\varphi$ is of class $C^r$, where $r=\underline{r}(\alpha):=\max\{n\in\N_0:n<\alpha\}$;
\item uniformly for $x,y$ in compact sets we have
\begin{equation}\label{eq:holdergamma}
	\left| \varphi(y)-F_x(y)\right|\lesssim |y-x|^\alpha
\end{equation}
where $F_x(\cdot)$ is the Taylor polynomial of $\varphi$ of order $r$ based at $x$, namely
\begin{equation}\label{eq:polyF}
	F_x(y):=\sum_{|k|< \alpha} \partial^k \varphi(x) \, \frac{(y-x)^k}{k!}, \qquad y\in\R^d\,.
\end{equation}
\end{enumerate}

\begin{remark}
The meaning of $\lesssim$ in \eqref{eq:holdergamma} is that for any compact set $K \subseteq \R^d$
there is a constant $C = C_K < \infty$ such that 
$| \varphi(y)-F_x(y)| \le C |y-x|^\alpha$ for all $x,y \in K$.
\emph{This notation will be used extensively throughout the paper.}
\end{remark}

\begin{remark}
For $r\in\N$ and $\alpha < r \le \alpha'$ we have the (strict) inclusions 
$\cC^{\alpha'} \subset C^r \subset \cC^{\alpha}$.
We stress that for $r\in\N$ the space $\cC^r$ is strictly larger than $C^r$
(for instance, $\cC^1$ is the space of locally Lipschitz functions,
and similarly $\cC^r$ is the space of functions in $C^{r-1}$
whose derivatives of order $r-1$ are locally Lipschitz). 
Incidentally, we note that other definitions of the space $\cC^r$
for $r\in\N$ are possible, see e.g.\ \cite{HL17}.
The one that we give here is convenient for our goals.
\end{remark}

\begin{remark}
We will later extend the definition of $\cC^\alpha$ to negative exponents $\alpha \le 0$:
this will no longer be a space of functions, but rather of \emph{distributions}.
\end{remark}

\section{Test functions, distributions, and scaling}
\label{sec:bg-distributions}

We introduce the fundamental notions of test functions and distributions on $\R^d$.

\begin{definition}[Test functions]\label{def:test-functions}
We denote by $\cD := \cD(\R^d)$ the space of $C^\infty$ functions
$\varphi: \R^d \to \R$ with compact support, called \emph{test functions}.
We denote by $\cD(K)$ the 
subspace of functions in $\cD$ supported on a set $K \subseteq \R^d$.
\end{definition}

\begin{definition}[Distributions]
A linear functional $T:\cD(\R^d)\to\R$ is called a
\emph{distribution on $\R^d$} (or simply a distribution, or generalized function),
if for every compact set $K \subseteq \R^d$
there exist $r = r_K\in\N_0$ and $C = C_K < \infty$ such that
\begin{equation}\label{eq:order}
	|T(\varphi)|\leq C\|\varphi\|_{C^r}, \qquad \forall\, \varphi\in\cD(K).
\end{equation}
The space of distributions on $\R^d$ is denoted by $\cD' := \cD'(\R^d)$.

Given $K \subseteq \R^d$, any linear functional $T : \cD(K) \to \R$ 
which satisfies \eqref{eq:order}
for some $r\in\N$, $C < \infty$ is called a \emph{distribution on~$K$}.
Their space is denoted by $\cD'(K)$.
\end{definition}

\begin{remark}\label{rem:order}
When relation \eqref{eq:order} holds, we say that 
\emph{$T$ is a distribution of order $r$ on the set $K$}.
If one can choose $r$ independently of $K$,
we say that \emph{$T$ is a distribution of finite order $r$ on $\R^d$}
(the constant $C$ in \eqref{eq:order} is allowed to depend on $K$). 
Note that a distribution
of order $r$ on the set $K$ is also of order $r'\geq r$ on $K$.
\end{remark}

\begin{remark}\label{rem:dist-funct-meas}
Here are some basic examples of distributions.
\begin{itemize}
\item Any locally integrable function $f \in L^1_{\text{loc}}$
(hence \emph{any continuous function}) can be
canonically identified
with the distribution
$f(\varphi) := \int f(z) \, \varphi(z) \d z$.
\item More generally, 
any Borel measure $\mu$ on $\R^d$ which is finite on compact sets can be
identified with the distribution $\mu(\varphi) := \int \varphi \d \mu$.
\end{itemize}
Both $f(\varphi)$ and $\mu(\varphi)$ are distributions
of finite order $r=0$ on $\R^d$.
\end{remark}

\subsection*{Scaling}
We next introduce the key notion of \emph{scaling}.
Given a function $\varphi:\R^d \to \R$, we denote by
\emph{$\varphi_x^\lambda:\R^d \to \R$ the scaled version of $\varphi$
that is centered at $x$
and localised at scale $\lambda>0$}, defined as follows:
\begin{equation} \label{eq:scale}
	\varphi_x^\lambda(z) := \lambda^{-d} \varphi(\lambda^{-1} ( z - x)) \,.
\end{equation}
When $x = 0$ we write $\varphi^\lambda = \varphi^\lambda_0$,
when $\lambda = 1$ we write $\varphi_x = \varphi^1_x$.

Note that if $\varphi$ is supported in $B(0,1)$, 
then $\varphi_x^\lambda$ is supported in $B(x,\lambda)$.
The scaling factor $\lambda^{-d}$ in \eqref{eq:scale}
is chosen to preserve the integral:
\begin{equation*}
	\int \varphi_x^\lambda(z) \d z = \int \varphi(z) \d z \,, \qquad
	\| \varphi_x^\lambda \|_{L^1} = \| \varphi \|_{L^1} \,.
\end{equation*}

\textbf{We will use scaled functions $\varphi_x^\lambda$ extensively}.
The basic intuition is that
given a distribution $T \in \cD'$ and a test function $\varphi \in \cD$,
the map $\lambda\mapsto T(\varphi_x^\lambda)$ for small $\lambda > 0$
tells us something useful about the \emph{behavior of $T$ close to $x\in\R^d$}.

\begin{remark} \label{rem:order-r-d}
We can bound the $C^r$ norm of a scaled test function $\varphi_x^\lambda$
as follows:
\begin{equation} \label{eq:norm-scaling}
	\|\varphi_x^\lambda \|_{C^r} \le \lambda^{-d-r} \,
	\|\varphi \|_{C^r} \,,
\end{equation}
simply because $\|\partial^k \varphi_x^\lambda \|_\infty
= \lambda^{-|k|-d} \|\partial^k \varphi \|_\infty$,
see \eqref{eq:normCr} and \eqref{eq:scale}.

As a consequence, 
given a distribution $T \in \cD'$, a compact set $K\subseteq \R^d$
and a test function $\varphi \in \cD$, we have the following bound,
for a suitable $r\in\N$:
\begin{equation} \label{eq:bound-div}
	|T(\varphi_x^\lambda)| \lesssim \lambda^{-r-d}  \,,
\end{equation}
uniformly for $x \in K$ and $\lambda \in (0,1]$.
Indeed, it suffices to take $r = r_{\bar K_1}$ in \eqref{eq:order} for 
the compact set $\bar K_1$
(the $1$-enlargement of $K$, see \eqref{eq:enlargement})
and to apply \eqref{eq:norm-scaling}.
\end{remark}

{In some cases it can be useful to consider \emph{non-Euclidean} scalings
(like in the theory of regularity structures for applications to \emph{parabolic} SPDEs,
see \cite[Section~2]{Hairer2014d}).} Our approach could be easily adapted to
such scalings, but for simplicity of presentation we refrain from doing so in this survey.

\section{Germs of distributions and coherence}
\label{sec:germs}

The following definition is crucial to our approach.

\begin{definition}[Germs]\label{def:germs}
We call \emph{germ} a family $F = (F_x)_{x\in\R^d}$ of distributions
$F_x \in \cD'(\R^d)$ indexed by $x \in \R^d$,
or equivalently a map $F:\R^d\to\cD'(\R^d)$, such that for all $\psi\in\cD$ the map
$x\mapsto F_x(\psi)$ is measurable.
\end{definition}

We think of a germ $F = (F_x)_{x\in\R^d}$ as a collection of 
candidate local approximations for an unknown distribution.
More precisely, the problem is to find a distribution $f\in\cD'(\R^d)$ 
which in the proximity of any point $x\in\R^d$ is well-approximated by $F_x$,
in the sense that ``$f - F_x$ is small close to $x$''.
This can be made precise by requiring that for some given test function
$\varphi \in \cD$ with $\int \varphi \ne 0$ we have
\begin{equation} \label{eq:pre-coherent}
	\lim_{\lambda\downarrow 0} \, |(f-F_x)(\varphi_x^\lambda)| = 0
	\quad \text{uniformly for $x$ in compact sets} \,.
\end{equation}
Remarkably, this property is enough to guarantee \emph{uniqueness}.
The simple proof of the next result is given in \Cref{sec:conv} below.

\begin{lemma}[Uniqueness]\label{th:identi}
Given any germ $F = (F_x)_{x\in\R^d}$
and any test function $\varphi \in \cD$
with $\int \varphi \ne 0$,
there is at most one distribution $f \in \cD'$ which satisfies~\eqref{eq:pre-coherent}.

More precisely, given a compact set $K \subseteq \R^d$ and
two distributions $f_1, f_2 \in \cD'$ such that
$\lim_{\lambda\downarrow 0} \, |(f_i-F_x)(\varphi_x^\lambda)| = 0$
uniformly for $x \in K$, then $f_1$ and $f_2$
must ``coincide on $K$'', in the sense that
$f_1(\psi) = f_2(\psi)$ for any $\psi \in \cD(K)$.
\end{lemma}

\subsection*{Coherence}

Given a germ $F = (F_x)_{x\in\R^d}$, 
we now investigate the
\emph{existence} of a distribution $f\in\cD'$
which satisfies \eqref{eq:pre-coherent}.
The key to solving this problem
is the following condition, that we call \emph{coherence}.

\begin{definition}[Coherent germ]\label{def:coherent-germ}
Fix $\gamma \in \R$.
A germ $F = (F_x)_{x\in\R^d}$ is called
\emph{$\gamma$-coherent} if 
there is a test function $\varphi \in \cD$ with $\int \varphi \ne 0$ 
with the following property:
for any compact set $K \subseteq \R^d$
there is a real number $\alpha_K \le \min\{0, \gamma\}$ such that
\begin{equation}\label{eq:coherent}
\begin{gathered}
	| (F_z - F_y)(\varphi^\epsilon_y) | \lesssim
	\epsilon^{\alpha_K} \, (|z-y| + \epsilon)^{\gamma - {\alpha_K}} \\
	\text{uniformly for {$z,y \in K$}} \text{ and for } \epsilon \in (0,1] \,.
\end{gathered}
\end{equation}
If $\balpha = (\alpha_K)$ is the family of exponents in \eqref{eq:coherent},
we say that $F$ is \emph{$(\balpha,\gamma)$-coherent}.
If $\alpha_K = \alpha$ for every~$K$, we say that $F$ is \emph{$(\alpha,\gamma)$-coherent}.
\end{definition}

We can already state a preliminary version of the Reconstruction Theorem.

\begin{theorem}[Reconstruction Theorem, preliminary version]\label{th:reco}
Let $\gamma\in\R$ and $F = (F_x)_{x\in\R^d}$ a $\gamma$-coherent germ
as in \Cref{def:coherent-germ}. 
Then there exists a distribution $f=\cR F \in \cD'(\R^d)$ such that, for any given
test function $\xi \in\cD$, we have
\begin{equation}\label{eq:reco+0}
\begin{gathered}
	| ( f - F_x)(\xi^\lambda_x) |
	\lesssim \, \begin{cases}
	\lambda^\gamma & \text{if } \gamma \ne 0 \\
	1+|\log \lambda| & \text{if } \gamma = 0
	\end{cases} \\
	\text{uniformly for $x$ in compact sets and $\lambda \in (0,1]$} \,.
\end{gathered}
\end{equation}
If $\gamma > 0$, the distribution
$f$ is unique and we call it the \emph{reconstruction of $F$}.
\end{theorem}

The Reconstruction Theorem will be stated in full in \Cref{sec:reco} below,
with a strengthened version of relation \eqref{eq:reco+0}
which holds \emph{uniformly} over
a suitable class of test functions~$\xi$.
We first need to investigate the notion of coherence.

\begin{remark}The coherence condition \eqref{eq:coherent} is a
strong constraint on the germ. Indeed, we can equivalently rewrite this condition as follows
\begin{equation*}
\begin{split}
	| (F_z - F_y)(\varphi^\epsilon_y) | \lesssim
	\begin{cases}
	\epsilon^{\gamma} & \text{if }\ 0 \le |z-y| \le \epsilon \\
	\epsilon^{\alpha_K} \, |z-y|^{\gamma - \alpha_K} 
	& \text{if }\ |z-y| > \epsilon \,.
	\end{cases} 
\end{split}
\end{equation*}
In particular, as $|z-y|$
decreases from $1$ to $\epsilon$, the right hand side
\emph{improves} from $\epsilon^{\alpha_K}$ to $\epsilon^\gamma$, since $\alpha_K\leq\gamma$.
In the case $\alpha_K<0<\gamma$ this improvement is particularly dramatic, since, 
as $\epsilon \downarrow 0$, 
$\epsilon^{\alpha_K}$ \emph{diverges}  while
$\epsilon^\gamma$ \emph{vanishes}.
\end{remark}

\begin{remark}[Monotonicity of $\alpha_K$]\label{rem:monotalpha}
Without any real loss of generality,
\emph{we will always assume that the family of exponents $\balpha = (\alpha_K)$
in  \eqref{eq:coherent} is monotone}:
\begin{equation}\label{eq:monotone-alpha}
	\forall K \subseteq K': \qquad \alpha_K \ge \alpha_{K'} \,.
\end{equation}
This is natural, because the right hand side of \eqref{eq:coherent} 
is non-increasing in $\alpha_K$.
Indeed, starting from an arbitrary family $\balpha = (\alpha_K)$
for which \eqref{eq:coherent} holds, we can easily
build a \emph{monotone} family $\tilde\balpha = (\tilde\alpha_K)$
for which \eqref{eq:coherent} still holds, e.g.\ as follows:
\begin{itemize}
\item for balls $B(0,n)$ of radius $n\in\N$
we define $\tilde\alpha_{B(0,n)} := \min\{\alpha_{B(0,i)}: \ i=1,\ldots, n\}$;

\item for general compact sets $K$ we first define
$n_K := \min\{n\in\N: \ K \subseteq B(0,n)\}$ and then
$\tilde\alpha_K := \tilde\alpha_{B(0,n_K)}$.
\end{itemize}
\end{remark}

\begin{remark}[Vector space]\label{rem:germs-vector-space}
We stress that the coherence condition \eqref{eq:coherent}
is required to hold
\emph{for a single arbitrary test function $\varphi \in \cD$ with $\int \varphi \ne 0$}.
We will show in \Cref{th:coh} the non obvious fact that $\varphi$ in \eqref{eq:coherent} can be replaced
by \emph{any test function $\xi \in \cD$}, 
provided we also replace $\alpha_K$ by $\alpha'_K := \alpha_{\bar{K}_2}$,
where $\bar K_R$ denotes the $R$-enlargement of the set $K$,
see \eqref{eq:enlargement}.
It follows that, for any given $\gamma \in \R$,
\emph{the family of $\gamma$-coherent germs is a vector space}.
\end{remark}

\begin{remark}[Cutoffs]\label{rem:xyK}
In the coherence condition \eqref{eq:coherent}
we could require that the base points $z,y$ are at \emph{bounded distance}. Indeed, if 
\eqref{eq:coherent} holds when $|z-y| \le R$ for some fixed $R > 0$,
then the constraint $|z-y| \le R$ can be dropped and 
\eqref{eq:coherent} still holds (possibly with a different multiplicative constant).
Similarly, the constraint $\epsilon \in (0,1]$ 
can be replaced by $\epsilon \in (0, \eta]$, for any fixed $\eta > 0$.
The proof is left as an exercise.
\end{remark}

It is convenient to introduce a semi-norm
which quantifies the coherence of a germ. 
Fix a compact set $K \subseteq \R^d$, a test function $\varphi \in \cD$
and two real numbers $\alpha_K \le 0$, $\gamma \ge \alpha_K$.
Given an arbitrary germ $F = (F_x)_{x\in\R^d}$,
we denote by $\vertiii{F}^\co_{K, \varphi, \alpha_K,\gamma}$
the best (possibly infinite) constant 
for which the inequality \eqref{eq:coherent} holds
for $y,z\in K$ with $|z-y| \le 2$ (this
last restriction is immaterial, by \Cref{rem:xyK}):
\begin{equation} \label{eq:triple1}
	\vertiii{F}^\co_{K, \varphi, \alpha_K,\gamma} 
	:= \sup_{\substack{y,z\in K, \ |z-y|\le 2, \ \epsilon \in (0,1]}}
	\ \frac{| (F_z - F_y)(\varphi^\epsilon_y) |}
	{\epsilon^{\alpha_K} \, (|z-y| + \epsilon)^{\gamma - {\alpha_K}}} \,.
\end{equation}
Then, given $\gamma \in \R$ and $\balpha = (\alpha_K)$,
a germ $F$ is $(\balpha,\gamma)$-coherent if and only if 
for some $\varphi \in \cD$
with $\int \varphi \ne 0$ we have
$\vertiii{F}^\co_{K, \varphi, \alpha_K,\gamma} < \infty$
 for every compact set $K$.

\subsection*{Examples}

We now present a few concrete examples of germs.

\begin{example}[Constant germ]\label{ex:const}
Let us fix any distribution $T\in\cD'$ and set $F_x:=T$ for all $x\in\R^d$. Then 
$F = (F_x)_{x\in\R^d}$ is a $(\balpha,\gamma)$-coherent germ for any $(\balpha,\gamma)$, 
since $F_z-F_y=0$ for all $z,y\in\R^d$. Although this example may look trivial, 
it does occur in regularity structures, in particular for some notable elements
of negative homogeneity (see \Cref{th:boundor} below).
\end{example}

\begin{example}[A link with Regularity Structures]\label{ex:regstruc}
Let $\varphi \in \cD$ be a fixed test function with $\int\varphi \ne 0$.
Let $A\subset \R$ be a finite set and set $\alpha := \min A$.
Let $F = (F_x)_{x\in\R^d}$ be a germ such that,
for some $\gamma > \alpha$, we have
\begin{equation} \label{eq:cohRS}
\begin{gathered}
	| (F_z - F_y)(\varphi^\epsilon_y) | \lesssim \sum_{a \in A: \ a < \gamma}
	\epsilon^a \, |z-y|^{\gamma-a}  \\
	\text{uniformly for $z,y$ in compact sets and for $\epsilon \in (0,1]$} \,.
\end{gathered}
\end{equation}
Then the germ $F$ is $(\alpha,\gamma)$-coherent.
Indeed, it suffices to note that for $\alpha \le a < \gamma$
\begin{equation*}
	\epsilon^a \, |z-y|^{\gamma-a}  =
	\epsilon^\alpha \, ( \epsilon^{a-\alpha} \, |z-y|^{\gamma-a} )
	\le  \epsilon^\alpha \, (\epsilon + |z-y|)^{\gamma-\alpha}  \,,
\end{equation*}
simply because $v^\beta w^\delta \le (v+w)^{\beta+\delta}$ for any $v,w,\beta,\delta \ge 0$.

\emph{All germs which appear in Regularity Structures satisfy \eqref{eq:cohRS}.}
For readers who are familiar with this theory,
the precise link is the following: given a Regularity Structure $(A, \mathcal{T}, G)$, if
$(\Pi_x,\Gamma_{xy})_{x,y\in\R^d}$ is a model and $f \in \cD^\gamma$ 
is a modelled distribution, then
the germ $(F_x := \Pi_x f(x))_{x\in\R^d}$ satisfies~\eqref{eq:cohRS} since one can write
\[
\begin{split} 
(\Pi_z f(z)-\Pi_yf(y))(\varphi^\epsilon_y)&=-\Pi_y(f(y)-\Gamma_{yz}f(z))(\varphi^\epsilon_y)
 = \sum_{|\tau|<\gamma} g^\tau_{zy}\, \Pi_y\tau(\varphi^\epsilon_y)
\end{split}
\]
with $|\Pi_y\tau(\varphi^\epsilon_y)|\lesssim \epsilon^{|\tau|}$, which holds for all models, and
$|g^\tau_{zy}|\lesssim |z-y|^{\gamma-|\tau|}$, which holds for all modelled distributions.
\end{example}

\begin{example}[Taylor polynomials]\label{ex:taylor}
Let $\gamma>0$ and fix a function $f\in\cC^\gamma(\R^d)$. 
We recall that by \eqref{eq:holdergamma}
we have $\left| f(w)-F_y(w)\right|\lesssim |w-y|^\gamma$ for $w,y$ in compact sets, where
for all $y\in\R^d$ the function $F_y\in C^\infty(\R^d)$ given by
\[
F_y(w):=\sum_{|k|< \gamma} \partial^k f(y) \, \frac{(w-y)^k}{k!}, \qquad w\in\R^d\,,
\]
is the Taylor polynomial of $f$ centered at $y$ 
of order $\underline{r}(\gamma):=\max\{n\in\N_0:n<\gamma\}$ defined in \eqref{eq:polyF}. 
Let us now show that $F = (F_x)_{x\in\R^d}$ 
is a $(0,\gamma)$-coherent germ. 

Fix a compact set $K\subset\R^d$.
Note that for every $k\in\N_0^d$ such that $|k|< \gamma$ we have $\partial^k f\in\cC^{\gamma-|k|}$.
By Taylor expanding $\partial^k f(y)$ around $z$, we obtain
\[
\begin{split}
F_y(w)& = \sum_{|k|< \gamma} \left(\sum_{|\ell|< \gamma-|k|} \partial^{k+\ell} f(z) \, 
\frac{({y-z})^\ell}{\ell!}+R^k(y,z)\right)\frac{(w-y)^k}{k!}
\end{split}
\]
with $|R^k(y,z)|\lesssim |y-z|^{\gamma-|k|}$ uniformly for $y,z\in K$.
We change variable in the inner sum from $\ell$ to $k' := k+\ell$
and note that the constraint $|\ell|< \gamma-|k|$ becomes
$\{|k'| < \gamma\} \cap \{k' \ge k\}$, where $k'\ge k$ means
$k'_i \ge k_i \ \forall i=1,\ldots, d$.
If we interchange the two sums we then get, by the binomial theorem,
\begin{equation*}
\begin{split}
	F_y(w) & = \sum_{|k'| < \gamma}
	\partial^{k'} f(z) \left( \sum_{k \le k'} \frac{(y-z)^{k'-k}}{(k'-k)!}
	\, \frac{(w-y)^k}{k!} \right)
	+\sum_{|k|< \gamma}R^k(y,z)\,\frac{(w-y)^k}{k!}
	\\ & = F_z(w)+\sum_{|k|< \gamma}R^k(y,z)\,\frac{(w-y)^k}{k!}
	\,.
\end{split}
\end{equation*}
Therefore
\begin{equation}\label{eq:F-F}
F_z(w)-F_y(w)=-\sum_{|k|< \gamma}R^k(y,z)\,\frac{(w-y)^k}{k!}
\end{equation}
and since $|R^k(z,y)|\lesssim |y-z|^{\gamma-|k|}$ we get
\[
\begin{split}
|F_z(w)- F_y(w)| &\lesssim \sum_{|k|< \gamma} |w-y|^{|k|}\,|y-z|^{\gamma-|k|}\,.
\end{split}
\]
Therefore, for any $\varphi\in\cD$ we have, uniformly for $y,z\in K$,
\[
\left|\int_{\R^d} (F_z(w)-F_y(w))\, \varphi^\epsilon_y(w)\d w\right| \lesssim 
\sum_{n< \gamma} |z-y|^{\gamma-n}\,\epsilon^{n}\, .
\]
This is a particular case of the class studied in \Cref{ex:regstruc}, 
with $\alpha=0$ and $A=\{n\in\N_0: n<\gamma\}$, therefore the germ
$F$ is $(0,\gamma)$-coherent.
General germs are meant to be a generalisation of local Taylor expansions.
\end{example}

\subsection*{Homogeneity}
\label{sec:hom}

For a coherent germ $F = (F_x)_{x\in\R^d}$,
we can bound $|F_x(\varphi_x^\epsilon)|$ as $\epsilon \downarrow 0$.

\begin{lemma}[Homogeneity] \label{th:boundor}
Let $F = (F_x)_{x\in\R^d}$ be a $\gamma$-coherent germ.
For any compact set $K \subseteq \R^d$, there is a real number
$\beta_K < \gamma$ such that 
\begin{gather}
	\label{eq:bounded-order-germ}
	|F_x(\varphi^\epsilon_x)| \lesssim \, \epsilon^{\beta_K} 
	\quad \text{uniformly for $x \in K$ and $\epsilon \in (0,1]$} \,,
\end{gather}
with $\varphi$ as in \Cref{def:coherent-germ}. 
We say that $F$ has \emph{local homogeneity bounds
$\bbeta = (\beta_K)$}.
If $\beta_K = \beta$ for all $K$, we say that $F$ has 
\emph{global homogeneity bound $\beta$}.
\end{lemma}

\noindent
The request $\beta_K < \gamma$ is to rule out trivialities.

\begin{proof}
Let $(F_x)_{x\in\R^d}$ be $\gamma$-coherent.
Given a compact set $K\subseteq \R^d$,
fix a point $z \in K$.
By \Cref{rem:order-r-d} applied to $T=F_{z}$, see \eqref{eq:bound-div},
there is $r \in \N_0$ such that
\begin{equation*} 
	|F_{z}(\varphi_x^\epsilon)| \lesssim \, \epsilon^{-r-d}
	\qquad \text{uniformly for $x \in K$ and $\epsilon \in (0,1]$} \,.
\end{equation*}
If we denote by $\mathrm{diam}(K) := \sup\{|x-z|: \ x,z \in K\}$,
by \eqref{eq:coherent} we can bound
\begin{equation*} 
	|(F_{x} - F_{z})(\varphi_x^\epsilon)| 
	\lesssim \epsilon^{\alpha_K} \, (|x-z|+\epsilon)^{\gamma-{\alpha_K}} 
	\le \epsilon^{\alpha_K} \, (\mathrm{diam}(K)+1)^{\gamma-{\alpha_K}} 
	\lesssim \epsilon^{\alpha_K} \,,
\end{equation*}
always uniformly for $x\in K$ and $\epsilon \in (0,1]$. This yields
\begin{equation*} 
	|F_{x}(\varphi_x^\epsilon)| \le
	|(F_{x} - F_{z})(\varphi_x^\epsilon)| 
	+ |F_{z}(\varphi_x^\epsilon)|
	\lesssim \epsilon^{\alpha_K} + \epsilon^{-r-d} \,,
\end{equation*}
hence \eqref{eq:bounded-order-germ} holds
with $\beta_K = \min\{{\alpha_K}, -r - d\}$ (which, of course, might not be the  best value of $\beta_K$).
By further decreasing $\beta_K$, if needed, we may ensure that $\beta_K < \gamma$.
\end{proof}

\begin{remark}[Monotonicity of $\beta_K$]\label{rem:monotbeta}
In analogy with \Cref{rem:monotalpha},
we will always assume that the homogeneity bounds $\bbeta = (\beta_K)$
in  \eqref{eq:bounded-order-germ} are \emph{monotone}:
\begin{equation}\label{eq:monotone-beta}
	\forall K \subseteq K': \qquad \beta_K \ge \beta_{K'} \,.
\end{equation}
Note that the right hand side of \eqref{eq:bounded-order-germ} is non-increasing in $\beta_K$.
\end{remark}

\begin{remark}[Vector space]\label{rem:possfin}
We will show in \Cref{pr:enhom} that in \eqref{eq:bounded-order-germ} we can replace
$\varphi$ by \emph{any test function $\xi\in\cD$},
provided we also replace $\beta_K$ by $\beta'_K := \beta_{\bar{K}_2}$.
As a consequence (recall also \Cref{rem:germs-vector-space}), for any given
$\alpha \le 0$ and $\gamma \ge \alpha$, \emph{the family of $(\alpha,\gamma)$-coherent germs
with global homogeneity bound $\beta$ is a vector space}.
\end{remark}

\begin{remark}[Positive homogeneity bounds]\label{rem:poshom}
In concrete applications, we typically have $\beta_K\le 0$ in \eqref{eq:bounded-order-germ},
because the case $\beta_K > 0$ is somewhat trivial.
Indeed, we recall that given a $\gamma$-coherent germ $F$,
our problem is to find a distribution $f\in\cD'$ that satisfies \eqref{eq:pre-coherent}.
If $\beta_K > 0$ for some compact set $K \subseteq \R^d$, then
$f = 0$ satisfies \eqref{eq:pre-coherent} on $K$ and, by \Cref{th:identi}, any solution $f$
of \eqref{eq:pre-coherent} must therefore vanish on $K$.
In particular, \emph{if $\beta_K > 0$ for all $K$, the only solution to \eqref{eq:pre-coherent}
is $f=0$.}
Using the notation of the Reconstruction Theorem, we can write $\cR F = 0$.
\end{remark}

\begin{example}\label{ex:homconst}
For a coherent germ there is in general
no fixed order between the lower bound $\beta_K$ of the homogeneity in \eqref{eq:bounded-order-germ} and the exponent $\alpha_K$ 
appearing in the coherence definition \eqref{eq:coherent}.
\begin{itemize}
\item In Regularity Structures, see \Cref{ex:regstruc},
it is usually assumed that $\beta_K=\alpha_K=\alpha$ for all $K$.

\item A constant germ $F_x=T$ with $T\in\cD'$, 
see \Cref{ex:const},
is $(\balpha,\gamma)$-coherent for any $\balpha$ and
$\gamma$. 
It is possible that $\beta_K < 0$, e.g.\ for the function $T(y) := |y|^{-1/2}$
we have $\beta_K = -\frac{1}{2}$ for $K=B(0,1)$.
Since we can choose $\alpha_K = 0$ here, we might have $\beta_K < \alpha_K$.
\item If $F$ is a $(\balpha,\gamma)$-coherent germ, then 
for any fixed distribution $f \in \cD'$ the germ
$G = (G_x:=f - F_x)_{x\in\R^d}$ is still $(\balpha,\gamma)$-coherent.
By the Reconstruction Theorem that we are about to state,
it is possible to choose $f = \cR F$ such that
for the germ $G$ we have that $\beta_K\geq\gamma$ (see \eqref{eq:reco+0} below),
hence $\beta_K \ge \alpha_K$.
\end{itemize}
\end{example}

\section{The Reconstruction Theorem}
\label{sec:reco}

We are ready to state the full version
of Hairer's \emph{Reconstruction Theorem}~\cite[Th.~3.10]{Hairer2014d} in our context
(see also \cite[Prop.~3.25]{Hairer2014d}).
Recalling the definition \eqref{eq:normCr} of $\|\cdot\|_{C^r}$, for
$r\in\N_0$ we define the following family of test functions:
\begin{equation}\label{eq:cBr}
	\cB_r := \{\psi \in \cD(B(0,1)): \ \|\psi\|_{C^r} \le 1\} \,.
\end{equation}
We also recall that $\bar K_R$ denotes the $R$-enlargement of the set $K$,
see \eqref{eq:enlargement}.

\begin{theorem}[Reconstruction Theorem]\label{th:reco+}
Let $\gamma\in\R$ and $F = (F_x)_{x\in\R^d}$ be a $(\balpha,\gamma)$-coherent germ
as in \Cref{def:coherent-germ} with local homogeneity bounds $\bbeta$, see \Cref{th:boundor}.
Then there exists a distribution $f \in \cD'(\R^d)$ such that
for any compact set $K\subset\R^d$ and any
integer $r >\max\{-\alpha_{\bar K_2}, -\beta_{\bar K_2}\}$ we have, for 
$\alpha := \alpha_{\bar{K}_2}$,
\begin{equation}
	\label{eq:reco+1}
\begin{gathered}
	| ( f - F_x)(\psi^\lambda_x) |
	\le \, 
	\mathfrak{c}_{\alpha, \gamma, r, d, \varphi} \,
	\vertiii{F}^\co_{\bar K_2, \varphi, \alpha ,\gamma} \cdot 
	\begin{cases}
	\lambda^\gamma & \text{if } \gamma \ne 0 \\
	\big(1+|\log \lambda|\big) & \text{if } \gamma = 0
	\end{cases} \\
	\text{uniformly } \text{for } \psi \in \cB_r, \ x \in K, \ \lambda \in (0,1] \,,
\end{gathered}
\end{equation}
where
the semi-norm $\vertiii{F}^\co_{\bar K_2, \varphi, \alpha ,\gamma}$ is
defined in \eqref{eq:triple1}, $\varphi$ is as in \Cref{def:coherent-germ}, and
$\mathfrak{c}_{\alpha, \gamma, r, d, \varphi}$ is
an explicit constant, see \eqref{eq:frakc}-\eqref{eq:frakc2}-\eqref{eq:frakc3}.

If $\gamma > 0$, such a distribution
$f = \cR F$ is unique and we call it the \emph{reconstruction of $F$}.
Moreover the map $F \mapsto \cR F$ is linear.

If $\gamma\leq 0$ the distribution $f$ is not unique
but, for any fixed $\alpha \le 0$ and $\gamma \ge \alpha$,
one can choose $f$ in such a way that the 
map $F\mapsto f= \cR F$ is linear on the vector space of
$(\alpha,\gamma)$-coherent germs
with global homogeneity bound $\beta$.
\end{theorem}

The strategy of our proof of the Reconstruction Theorem is 
close in spirit to the original proof by Hairer:
given a germ $F$, we ``paste together''
the distributions $F_x$
on smaller and smaller scales, in order to build $\cR F$.
The existing proofs 
exploit
test functions possessing 
special multi-scale properties,
such as wavelets (by Hairer \cite{Hairer2014d}) or the heat kernel (by Otto-Weber \cite{ow19}).
Our proof is based on the 
\emph{single arbitrary test function $\varphi \in \cD$ with $\int \varphi \ne 0$}
which appears in the coherence condition \eqref{eq:coherent},
that we will suitably tweak in order to perform multi-scale arguments.

\begin{remark}\label{rem:RTsingle}
\Cref{th:reco} is a special case of \Cref{th:reco+}, because
equation \eqref{eq:reco+0} is a consequence of \eqref{eq:reco+1}.
This is obvious if $\xi \in \cB_r$, while for generic $\xi \in\cD$ it suffices
to note that $\psi := c \, \xi^{\eta} \in \cB_r$ for suitable $c, \eta > 0$, recall the notation \eqref{eq:scale}.
As a consequence, we can write $\xi_x^\lambda = c^{-1} \, \psi_x^{\eta^{-1} \lambda}$
with $\psi\in\cB_r$, hence \eqref{eq:reco+1} 
yields \eqref{eq:reco+0} for $\epsilon > 0$ small enough, which is enough (exercise).
\end{remark}

\begin{example}[Constant germ, reprise]
If we consider the constant germ $F_x=T$ of \Cref{ex:const} then for $f=T$ we have $f-F_x=0$ 
and therefore we can set $\cR F=T$. 
\end{example}

If we view a germ as a generalised local Taylor expansion, 
the Reconstruction Theorem associates to a \emph{coherent} germ
$F = (F_x)$ a \emph{global} distribution $f$ which is 
approximated by the germ $F_x$ locally around every $x\in\R^d$. If the
germ is a classical Taylor expansion of a function in $\cC^\gamma$,
as discussed in \Cref{ex:taylor}, then the Reconstruction
Theorem yields the function itself, as shown in the next example.

\begin{example}[Taylor polynomial, reprise]
Consider the germ given by the Taylor expansion of 
a function $f\in\cC^\gamma$, see \Cref{ex:taylor}. Then
by the Taylor theorem
\[
|f(y)-F_x(y)|\lesssim |y-x|^\gamma
\]
uniformly for $x,y$ in compact sets.
If $\psi$ is supported in $B(0,1)$, then $\psi^\lambda_x$ 
is supported in $B(x,\lambda)$, therefore
uniformly for $\lambda \in (0,1]$ we can bound
\begin{equation}\label{eq:TaylorRT}
	\left|\int_{\R^d} (f(y)-F_x(y))\, \psi^\lambda_x(y)\d y\right| \lesssim 
	\lambda^\gamma
	\, \int |\psi^\lambda_x(y)| \d y =  \lambda^\gamma \, \int |\psi(y)| \d y \,.
\end{equation}
This shows that $f$ satisfies \eqref{eq:pre-coherent}, therefore by uniqueness
we have $\cR F=f$. As a matter of fact, relation \eqref{eq:TaylorRT} holds uniformly for
$\psi \in \cB_0$ because 
$\int |\psi| \lesssim \|\psi\|_\infty \le 1$ (recall that $\psi \in \cB_0$ are supported in $B(0,1)$).
\end{example}

\begin{example}[On the case $\gamma = 0$]\label{ex:gamma=0}
If $F = (F_x)_{x\in\R^d}$ is a $(\alpha, 0)$-coherent germ, i.e.\ $\gamma = 0$,
the estimate \eqref{eq:reco+1} in the Reconstruction Theorem reads as follows:
\begin{equation}
	\label{eq:reco+1gamma0}
	| ( f - F_x)(\psi^\lambda_x) |
	\,\lesssim \, \log(1+\tfrac{1}{\lambda}) 
\end{equation}
uniformly for $x$ in compact sets, $\psi \in \cB_r$ and $\lambda \in (0,1]$.
We now show by an example
that \emph{the logarithmic rate in the right hand side of \eqref{eq:reco+1gamma0} is optimal}.

Consider the germ of functions $F = (F_x(y) := \log (1 + \frac{1}{|y-x|}))_{x\in\R^d}$.
If $\varphi \in \cD$ is a non-negative test function supported in $B(0,1)$ with
$\int \varphi > 0$, we can bound
\begin{equation*}
\begin{split}
	|(F_z  - F_y)(\varphi_y^\epsilon)| & \le
	|F_z(\varphi_y^\epsilon)| + |F_y(\varphi_y^\epsilon)| 
	\lesssim \log(1+\tfrac{1}{\epsilon})
	\lesssim \epsilon^{\alpha} \qquad \text{for any given $\alpha < 0$} \,.
\end{split}
\end{equation*}
This shows that the germ $F$ is $(\alpha,0)$-coherent,
hence by the Reconstruction Theorem there is $f\in\cD'$
such that \eqref{eq:reco+1gamma0} holds (e.g.\ $f \equiv 0$).
\emph{We claim that this bound cannot be improved, i.e.\ there is no $f \in \cD'$ such that
$| ( f - F_x)(\psi^\lambda_x) | \,\ll \, \log(1+\tfrac{1}{\lambda}) $.}

By contradiction, assume that such $f \in \cD'$ exists. Given a test function
$\psi \ge 0$ with $\psi(0) > 0$ and $\int\psi = 1$, we can bound 
$F_x(\psi^\lambda_x) \gtrsim \log(1+\tfrac{1}{\lambda})$ and by triangle inequality
\begin{equation*}
	 f(\psi^\lambda_x) \ge
	F_x(\psi^\lambda_x) - | ( F_x-f)(\psi^\lambda_x) |
	\gtrsim \log(1+\tfrac{1}{\lambda}) 
\end{equation*}
uniformly for $x$ in compact sets. In particular, there is a constant $c > 0$
such that
\begin{equation*}
	f(\psi^\lambda_x) \ge c \, \log(1+\tfrac{1}{\lambda})  \qquad
	\forall x \in B(0,2) \,.
\end{equation*}
This is impossible, for the following reason. Since 
$(\psi^\lambda)$ are mollifiers as $\lambda \downarrow 0$
(recall that we have fixed  $\int\psi = 1$), for any given test function $\xi \in \cD$ we can write
\begin{equation*}
	f(\xi) = \lim_{\lambda\downarrow 0} f(\xi * \psi^\lambda)
	= \lim_{\lambda\downarrow 0}  \int_{\R^d} f(\psi^\lambda_x)
	\, \xi(x) \d x \,.
\end{equation*}
If we fix $\xi \ge 0$ supported in $B(0,1)$ with $\int \xi = 1$, we finally get
\begin{equation*}
	f(\xi) \ge 
	\lim_{\lambda\downarrow 0}  \int_{\R^d} c \, \log(1+\tfrac{1}{\lambda}) \, \xi(x) \d x 
	= \lim_{\lambda\downarrow 0}  c \, \log(1+\tfrac{1}{\lambda}) = \infty
\end{equation*}
which is clearly a contradiction.
\end{example}

\begin{remark}
In the original formulation of the Reconstruction Theorem \cite[Thm. 3.10]{Hairer2014d}, 
the estimate in the right-hand side of \eqref{eq:reco+1} for $\gamma=0$ 
contains a factor $\lambda^\gamma$
instead of $(1+|\log\lambda|)$. This is not correct, as we showed in \Cref{ex:gamma=0}. The mistake 
in \cite{Hairer2014d} is in the very last display of 
the proof on page 324: in this formula we have $\|x-y\|_{\mathfrak s}\lesssim\delta+2^{-n}$ and $2^{-n}>\delta$,
so that the factor $\delta^{\gamma-\beta}$ in the left-hand side
must be replaced by $2^{-(\gamma-\beta)n}$. For
$\gamma<0$ the result does not change, but for $\gamma=0$ one obtains $1+|\log\delta|$ instead of $\delta^0$.
\end{remark}

\subsection{Guide to the proof of the Reconstruction Theorem}
\label{sec:guide}

The next sections are devoted to the proof of \Cref{th:reco+}. 

\begin{itemize}
\item In \Cref{sec:necessity} we show the necessity of coherence for the
Reconstruction Theorem.

\item In \Cref{sec:conv} we recall basic results
on test functions (such as convergence, convolutions and mollifiers) and
we prove \Cref{th:identi}.

\item In \Cref{sec:tweaking} we
show how to ``tweak'' an arbitrary test function, in order
to ensure that it \emph{annihilates all monomials up to a given degree}.
This is a key ingredient in the proof of the Reconstruction Theorem
because it will allow us to perform efficiently multi-scale arguments.

\item In \Cref{sec:basic-estimates-convolutions} we present
some elementary but crucial estimates on convolutions.

\item Finally, in \Cref{sec:reco+,sec:reco+neg} we give the proof of the Reconstruction
Theorem, first when $\gamma > 0$ and then when $\gamma \le 0$.
\end{itemize}

\section{Necessity of coherence}
\label{sec:necessity}

If a germ $F = (F_x)_{x\in\R^d}$ is $\gamma$-coherent, 
by the Reconstruction Theorem there is a distribution $f \in \cD'$ which is 
locally well approximated by $F$, see \eqref{eq:reco+1}.
In case $\gamma \ne 0$, this means the following:
\begin{equation} \label{eq:pre-coherent-nec}
	\begin{gathered}
	\text{$\forall$ compact set $K\subseteq \R^d$ $\exists\ r = r(K) \in \N$ such that} \\
	| ( f - F_x)(\psi^\lambda_x) | \lesssim \lambda^\gamma \\
	\text{uniformly for $x \in K$}, \ \lambda \in (0,1]
	\ \text{and} \ \psi \in \cB_r \,.
\end{gathered}
\end{equation}
Remarkably, 
coherence is also \emph{necessary} for \eqref{eq:pre-coherent-nec},
as we now show.

\begin{theorem}[Coherence is necessary]
If a germ $(F_x)_{x\in\R^d}$ 
satisfies \eqref{eq:pre-coherent-nec} for some $\gamma\in\R$,
then it is $\gamma$-coherent,
i.e.\ it satisfies the coherence condition \eqref{eq:coherent}, for any function $\varphi \in \cD$
and for a suitable family of exponents $\balpha = (\alpha_K)$.

If furthermore \eqref{eq:pre-coherent-nec} holds with $r(K) = r$ for every $K$,
for a fixed $r \in \N$,
then the germ $F$ is $(\alpha,\gamma)$-coherent for a suitable $\alpha \le 0$, i.e.\
we can take $\alpha_K = \alpha$ for all $K$.
\end{theorem}

This is a direct corollary of the next quantitative result.

\begin{proposition}\label{th:neces}
Let $(F_x)_{x\in\R^d}$ be a germ with the following property:
there exist a distribution $f \in \cD'$, numbers 
$\gamma \in \R$, $r\in\N$, $C  < \infty$ and a set $K \subseteq \R^d$ such that
\begin{equation} \label{eq:saty}
\begin{gathered}
	| ( f - F_x)(\psi^\lambda_x) |
	\le C \, \lambda^\gamma \\
	\text{for all \ $x \in K$, \ $\lambda \in (0,1]$ \ \text{and} \ $\psi \in \cB_r$}  \, ,
\end{gathered}
\end{equation}
with $\cB_r$ defined in \eqref{eq:cBr}.
Then for $\alpha := \min\{-r-d, \gamma\}$ we have
\begin{equation}\label{eq:pre-coh}
\begin{gathered}
	|(F_z - F_y)(\psi^\lambda_y)| \,\le\, 2C \, \lambda^\alpha \,
	(|z-y|+\lambda)^{\gamma-\alpha} \\
	\rule{0pt}{1.1em}\text{for all $y,z \in K$ with $|z-y| \le \tfrac{1}{2}$, 
	\ $\lambda \in (0,\tfrac{1}{2}]$ \ \text{and} \ $\psi \in \cB_r$}  \, .
\end{gathered}
\end{equation}
\end{proposition}

\begin{proof}
For $y,z\in K$, $\lambda \in (0,1]$ and $\psi \in \cB_r$. By \eqref{eq:saty} we can estimate
\begin{equation*}
\begin{split}
	|(F_z - F_y)(\psi^\lambda_y)|
	& = |(f - F_y)(\psi^\lambda_y) - (f - F_z)(\psi^\lambda_y)| \\
	& \le |(f - F_y)(\psi^\lambda_y)| + |(f - F_z)(\psi^\lambda_y)| \\
	& \le C \, \lambda^\gamma + |(f - F_z)(\psi^\lambda_y)| \,.
\end{split}
\end{equation*}
We claim that for $|z-y| \le \frac{1}{2}$ and $\lambda \in (0, \frac{1}{2}]$ we can bound
\begin{equation}\label{eq:eccoqua}
	|(f - F_z)(\psi^\lambda_y)| \le C \, \big( \tfrac{\lambda}{|z-y|+\lambda} \big)^{-r-d}
	\, (|z-y|+\lambda)^\gamma  \,.
\end{equation}
Note that for any $\alpha \le \gamma$ we can estimate
$\lambda^\gamma = \lambda^\alpha \, \lambda^{\gamma-\alpha}
	\le \lambda^\alpha \, (|z-y|+\lambda)^{\gamma-\alpha}$,
therefore if we set $\alpha := \min\{-r-d, \gamma\}$ we obtain \eqref{eq:pre-coh}.

It remains to prove \eqref{eq:eccoqua}.
Estimating $|(f - F_z)(\psi^\lambda_y) |$
is non obvious because $\psi^\lambda_y$ is centered at $y$ rather
than $z$. However, we claim that we can write
\begin{equation} \label{eq:claimphi}
	\psi^\lambda_y = \xi^{\lambda_1}_z \qquad
	\text{where} \qquad
	\xi := \psi^{\lambda_2}_w \,,
\end{equation}
where $\lambda_1, \lambda_2 \in (0,1]$ and $w \in B(0,1)$ are defined as follows:
\begin{equation*}
	\lambda_1 := |z-y|+\lambda \,, \qquad
	\lambda_2 := \tfrac{\lambda}{|z-y|+\lambda} \,, \qquad
	w := \tfrac{y-z}{|z-y|+\lambda} \,.
\end{equation*}
To prove \eqref{eq:claimphi}, recall that
$\xi^{\lambda_1}_z (x)
	= \lambda_1^{-d} \, \xi (
	\lambda_1^{-1}(x-z))$, hence for $\xi = \psi^{\lambda_2}_w$ we get
\begin{equation*}
\begin{split}
	\xi^{\lambda_1}_z (x)
	& = \lambda_1^{-d} \, \psi^{\lambda_2}_w (
	\lambda_1^{-1}(x-z)) = \lambda_1^{-d} \,
	\lambda_2^{-d} \, \psi ( \lambda_2^{-1} \{
	\lambda_1^{-1}(x-z) - w\}) \\
	& = (\lambda_1 \lambda_2)^{-d} \,
	\psi ( (\lambda_1 \lambda_2)^{-1} \{
	(x-z) - \lambda_1 w\}) = \lambda^{-d} \,
	\psi ( \lambda^{-1} \{
	x-y\}) = \psi^\lambda_y(x) \,.
\end{split}
\end{equation*}
Note that $\xi = \psi_w^{\lambda_2}$ is supported in $B(w, \lambda_2)
\subseteq B(0,1)$, because $|w| + \lambda_2 \le 1$
and $\psi$ is supported in $B(0,1)$.
Since $\xi$ is supported in $B(0,1)$, we have $\xi / \| \xi \|_{C^r} \in \cB_r$,
hence we can apply equation \eqref{eq:saty} 
with the replacements
\begin{equation*}
	x \ \rightsquigarrow \ z \,, \qquad
	\psi \ \rightsquigarrow \  \xi / \| \xi \|_{C^r} \,, \qquad
	\lambda \ \rightsquigarrow \ \lambda_1 
\end{equation*}
(note that $\lambda_1 \in (0,1]$ if $|z-y| \le \frac{1}{2}$ and $\lambda \in (0, \frac{1}{2}]$).
This yields
\begin{equation} \label{eq:ass-spec}
	| ( f - F_z)(\xi^{\lambda_1}_z) |
	\le \, C \, (\lambda_1)^\gamma \, \|\xi\|_{C^r}  \,.
\end{equation}
It remains to bound
\begin{equation*}
	\|\xi\|_{C^r} =
	\|\psi^{\lambda_2}_w\|_{C^r}
	= \max_{|k|\le r} \| \partial^k \psi^{\lambda_2}_w \|_\infty
	= \max_{|k|\le r} \| 
	\lambda_2^{-|k|-d} \, \partial^k \psi \|_\infty
	\le \lambda_2^{-r-d} \,,
\end{equation*}
because $\max_{|k|\le r} \|\partial^k \psi \|_\infty
= \|\psi\|_{C^r} \le 1$ for $\psi\in\cB_r$.
By \eqref{eq:claimphi} and \eqref{eq:ass-spec}, we get \eqref{eq:eccoqua}.
\end{proof}

\section{Convergence of test functions, convolutions\\ and mollifiers}
\label{sec:conv}

The space of test functions $\cD$ is equipped with a strong notion of convergence.

\begin{definition}[Convergence of test functions]
We say that
$\varphi_n \to \varphi$ in $\cD$ if and only if the following two conditions hold:
\begin{enumerate}
\item all $\varphi_n$'s are supported in some fixed compact set $K$,
i.e.\ $\varphi_n \in \cD(K) \ \forall n$;
\item $\varphi_n$ converges to $\varphi$ uniformly with all derivatives:
\begin{equation*}
	\forall r \in \N_0: \qquad
	\|\varphi_n - \varphi\|_{C^r} \to 0 \,.
\end{equation*}
\end{enumerate}
We typically consider sequences indexed by $n\in\N$, with convergence as $n\to\infty$,
or continuous families indexed by $n = \lambda \in (0,1]$, 
with convergence as $\lambda \downarrow 0$.
\end{definition}

\begin{remark} \label{rem:topology}
This notion of convergence is induced by a natural topology on $\cD$, 
called \emph{locally convex inductive limit topology}.
It is quite subtle -- non metrizable, not even first countable --
but we will not need to use it directly.
\end{remark}

We now show that the ``continuity property'' \eqref{eq:order}
in the definition of a distribution corresponds to ``sequential continuity''
with respect to convergence in $\cD$.\footnote{If a map $T: \cD \to \R$
is sequentially continuous, i.e.\ it satisfies \eqref{eq:sequentially},
this does \emph{not} imply that $T$ is a continuous map,
because the topology on $\cD$ is not first countable (recall \Cref{rem:topology}).
However, if $T$ is a \emph{linear map}, then sequential continuity implies continuity.}

\begin{lemma}\label{lem:sequential}
A linear functional $T:\cD(\R^d)\to\R$ is
a distribution if and only if
\begin{equation} \label{eq:sequentially}
	\varphi_n\to\varphi \ \text{ in }\ \cD \qquad \text{implies} \qquad
	T(\varphi_n) \to T(\varphi) \,.
\end{equation}
\end{lemma}

\begin{proof}
By the definition of convergence in $\cD$,
it is clear that \eqref{eq:order} implies \eqref{eq:sequentially}. Vice versa,
if \eqref{eq:order} fails for some compact $K$, then
for every $r = n \in \N$ and $C = n \in \N$ we can find $\varphi_n \in \cD(K)$
such that $|T(\varphi_n)| > n \|\varphi_n\|_{C^n}$;
if we define $\psi_n := n^{-1} \varphi_n / \|\varphi_n\|_{C^n}$,
we have $|T(\psi_n)| > 1$ for every $n\in\N$, which contradicts
\eqref{eq:sequentially} because $\psi_n \to 0$ in $\cD$ (indeed,
for any fixed $r\in\N$ we have $\|\psi_n\|_{C^r}
\le n^{-1}$ as soon as $n \ge r$).
\end{proof}

We recall that the convolution of two
measurable functions $f, g: \R^d \to \R$ is
the function $f*g = g*f: \R^d \to \R$ defined by
\begin{equation} \label{eq:f*g}
	(f * g)(x) := \int_{\R^d} f(x-y) \, g(y) \d y
	= \int_{\R^d} f(z) \, g(x-z) \d z \,,
\end{equation}
provided the integral makes sense for almost every $x\in\R^d$. 
This holds, in particular, when
$f = \varphi \in \cD$ is a test function and $g$
is \emph{locally integrable and compactly supported}: in this case the convolution
$\varphi * g \in \cD$ is a test function too, and we have
\begin{equation} \label{eq:der-conv}
	\partial^k(\varphi * g) = (\partial^k \varphi) * g \,.
\end{equation}
Given any distribution $T \in \cD'$, we can compute
\begin{equation*}
	T(\varphi * g) = \int_{\R^d} T(\varphi(\cdot-y)) \, g(y) \d y \,,
\end{equation*}
as one can deduce from \eqref{eq:f*g} (e.g.\ by linearity and Riemann sum approximations).
If we set $\varphi_y(x) : = \varphi(x-y) = \varphi_y^1(x)$, recall
\eqref{eq:scale}, we obtain the basic formula
\begin{equation}\label{eq:Tonconv}
	T(\varphi * g) = \int_{\R^d} T(\varphi_y) \, g(y) \d y \,,
\end{equation}
that will be used repeatedly in the sequel.

We next state a classical result that will be used frequently.

\begin{lemma}[Mollifiers]\label{th:molli}
Let $\rho: \R^d \to \R$, with $\int \rho = 1$ be compactly supported and integrable.
Then $\rho^\epsilon(z) := \rho_0^\epsilon(z) := \epsilon^{-d} \rho(\epsilon^{-1}z)$
are mollifiers as $\epsilon \downarrow 0$, i.e.\
\begin{equation*}
	\forall \varphi \in \cD: \qquad
	\varphi * \rho^\epsilon \to \varphi \quad \text{in $\cD$
	as $\epsilon \downarrow 0$} \,.
\end{equation*}
\end{lemma}

\begin{proof}
By \eqref{eq:der-conv} and $\int \rho^\epsilon = \int \rho = 1$
we can write, for any multi-index $k$,
\begin{equation*}
	\partial^k(\varphi * \rho^\epsilon)(x)-\partial^k\varphi(x)
	= \int_{\R^d} (\partial^k\varphi (x-y) - \partial^k\varphi (x)) \rho^\epsilon(y) \d y \,,
\end{equation*}
hence, by the change of variables $y = \epsilon z$,
\begin{equation} \label{eq:est-moll}
\begin{split}
	|\partial^k(\varphi * \rho^\epsilon)(x)-\partial^k\varphi(x)| 
	& \le \int_{\R^d} |\partial^k\varphi (x-y) - \partial^k\varphi(x)|
	\, |\rho^\epsilon(y)| \d y \\
	& = \int_{\R^d} |\partial^k\varphi (x- \epsilon z) - \partial^k\varphi(x)|
	\, |\rho(z)| \d z \,.
\end{split}
\end{equation}
Fix a compact set $K \subseteq \R^d$ and take $x \in K$.
Since $\rho$ is compactly supported, say on the ball $B(0,R)$,
for $\epsilon \in (0,1)$ the variable $x-\epsilon z$ belongs to the
compact set $K_R$, the $R$-neighborhood of $K$.
Then we can bound $|\partial^k\varphi (x- \epsilon z) - \partial^k\varphi(x)|
\lesssim \epsilon |z|$, because $\partial^k \varphi$ is of class $C^1$ (in fact $C^\infty$).
Since $\int |z| \, |\rho(z)| \d z < \infty$, it follows by \eqref{eq:est-moll} that
$\sup_{x\in K} |\partial^k(\varphi * \rho^\epsilon)(x)-\partial^k\varphi(x)| 
\lesssim \epsilon \to 0$.
This shows that $\varphi * \rho^\epsilon \to \varphi$ in $\cD$.
\end{proof}

We finally give the easy proof of \Cref{th:identi} (Uniqueness).

\begin{proof}[Proof of \Cref{th:identi}]
Let $\gamma>0$.
We fix a germ $(F_x)_{x\in\R^d}$, a test function $\varphi \in \cD$ with $\int \varphi \ne 0$,
a compact set $K\subseteq \R^d$ and two distributions
$f, g \in \cD'$ which satisfy, uniformly for $x \in K$,
\begin{equation}\label{eq:pre-coherent-fg}
	\lim_{\lambda\downarrow 0} \, |(f-F_x)(\varphi_x^\lambda)| =
	\lim_{\lambda\downarrow 0} \, |(g-F_x)(\varphi_x^\lambda)| = 0 \,.
\end{equation}
Our goal is to show that $f(\psi) = g(\psi)$ for every test function $\psi$
supported in $K$, i.e.\ $\psi\in\cD(K)$.
We may assume that $c := \int \varphi = 1$
(otherwise just replace $\varphi$ by $c^{-1} \, \varphi$).

We set $T := f-g$, we fix a test function $\psi \in \cD(K)$
and we show that $T(\psi) = 0$.
We have
$T(\psi) = \lim_{\lambda\downarrow 0} T(\psi * \varphi^\lambda)$
by \Cref{lem:sequential}, because
$\lim_{\lambda\downarrow 0}
\psi * \varphi^\lambda = \psi$ in $\cD$ by \Cref{th:molli}.
Recalling \eqref{eq:Tonconv}, we can write
\begin{equation*}
	|T(\psi * \varphi^\lambda)| =
	\bigg| \int_{\R^d} T(\varphi^\lambda_x) \, \psi(x) \d x \bigg|
	\le \|\psi\|_{L^1} \, \sup_{x\in K} |T(\varphi^\lambda_x)| \,,
\end{equation*}
where the last inequality holds for any $\lambda > 0$ since $\psi$ is supported in $K$.
It remains to show that $\lim_{\lambda\downarrow 0} T(\varphi^\lambda_x) = 0$ uniformly
for $x \in K$, for which it is enough to observe that
\[\begin{split}
	|T (\varphi_x^\lambda)| = |f(\varphi_x^\lambda) - g(\varphi_x^\lambda)|
	\le |(f- F_x) (\varphi_x^\lambda)| + |(g - F_x)  ( \varphi_x^\lambda)| 
	\end{split}
\]
and these terms vanish as $\lambda \downarrow 0$ uniformly
for $x \in K$, by \eqref{eq:pre-coherent-fg}.
\end{proof}

\section{Tweaking a test function}
\label{sec:tweaking}

Given an arbitrary test function $\varphi$
and an integer $r\in\N$,
we build a ``tweaked'' test function $\hat\varphi$
which annihilates monomials of degree from $1$ to $r-1$.
Recall that $\varphi^\lambda$ denotes the function
$\varphi^\lambda(x) := \lambda^{-d} \varphi(\lambda^{-1}x)$.

\begin{lemma}[Tweaking]\label{th:moments}
Fix $r \in \N = \{1,2,\ldots\}$
and distinct $\lambda_0, \lambda_1, \ldots, \lambda_{r-1} \in (0,\infty)$.
Define the constants $c_0, c_1, \ldots, c_{r-1} \in \R$ as follows:
\begin{equation}\label{eq:ci}
	c_i = 
	\prod_{k \in \{0,\ldots, {r-1}\}: \, k \ne i} \
	\frac{\lambda_k}{\lambda_k-\lambda_i} 
\end{equation}
(when $r=1$ we agree that $c_0 := 1$).
Then, for any
measurable and compactly supported  $\varphi: \R^d \to \R$
and any $a \in \R$, the ``tweaked'' function
$\hat\varphi$ defined by
\begin{equation} \label{eq:hatphi}
	\hat\varphi := a \sum_{i=0}^{r-1} c_i \, \varphi^{\lambda_i}  
\end{equation}
has integral equal to $a \int \varphi$ and annihilates monomials
of degree from $1$ to ${r-1}$:
\begin{equation} \label{eq:moments}
	\int \hat\varphi = a \int \varphi \quad \text{and} \quad
	\int_{\R^d} y^k \, \hat\varphi(y) \d y = 0\,,
	\quad \ \forall \, k \in \N_0^d : \ 1 \le |k| \le {r-1} \,.
\end{equation}
\end{lemma}

\begin{remark}
For fixed $a\in\R$,
equation \eqref{eq:moments} is a set of conditions,
one for each $k\in(\N_0)^d$ with $|k| \le {r-1}$
(where $k=0$ corresponds to $\int \hat\varphi = a \int \varphi$).
The number of such conditions
equals $r$ for $d=1$, while \emph{it is strictly larger
than $r$ for $d \ge 2$}. Nevertheless, 
we can fulfill these conditions by choosing only $r$ variables $c_0, c_1, \ldots, c_{r-1}$
as in \eqref{eq:ci}. This is due to the scaling properties
of monomials.
\end{remark}

We now show that in the coherence condition
\eqref{eq:coherent} we can replace $\varphi\in\cD$ by a suitable $\hat\varphi$
as in \Cref{th:moments}. Assume that for some $R_\varphi < \infty$ we have that
\begin{equation*}
	\int\varphi\neq 0, \qquad \varphi \text{ is supported in } B(0, R_\varphi) \,.
\end{equation*}
Then, given $r\in\N$, we define $\hat\varphi = \hat\varphi^{[r]}$ 
by \eqref{eq:hatphi} for $a = 1/ \int \varphi$
and for suitable $\lambda_i$'s:
\begin{equation}\label{eq:hatvarphi}
	\hat\varphi := \frac{1}{\int \varphi} \, \sum_{i=0}^{r-1} c_i \, 
	\varphi^{\lambda_i} \qquad
	\text{where}  \ \ \ \lambda_i := \frac{2^{-i-1}}{1+R_\varphi}
	\ \ \ \text{and $c_i$ as in \eqref{eq:ci}} \,.
\end{equation}

\begin{lemma}\label{th:coherence-r}
Let $F = (F_x)_{x\in\R^d}$ be a $(\alpha,\gamma)$-coherent germ
as in \Cref{def:coherent-germ}. For any  $r \in \N$, the coherence condition
\eqref{eq:coherent} still holds if $\varphi$ is replaced by
$\hat\varphi = \hat\varphi^{[r]}$ 
defined in \eqref{eq:hatvarphi}.
Such a test function $\hat\varphi$ has the following properties:
\begin{gather}\label{eq:supphatphi}
	\hat\varphi \text{ is supported in } B(0,\tfrac{1}{2}) \,, \\
	\label{eq:monomials}
	\int_{\R^d} \hat\varphi(y) \d y = 1 \,, \qquad 
	\int_{\R^d} y^k \, \hat\varphi(y) \d y = 0
	\quad \text{for} \ 1 \le |k| \le r-1 \,, \\
	\label{eq:boundhatvarphi}
	\| \hat\varphi \|_{L^1} \le \frac{e^2 \, r}{|\int \varphi|} \, \|\varphi\|_{L^1} \, .
\end{gather}
\end{lemma}

\begin{proof}
The function $\hat\varphi$ is supported in $B(0,1/2)$ because
$\lambda_i \le \frac{1}{2 \,R_\varphi}$.
Relation \eqref{eq:monomials} holds by \eqref{eq:moments}.
To prove \eqref{eq:boundhatvarphi}, note that by \eqref{eq:ci} we can bound
\begin{equation}\label{eq:boundonci}
	|c_i| =
	\prod_{k \in \{0,\ldots, r-1\}: \, k \ne i} \
	\frac{1}{|1-2^{k-i}|}
	\le \prod_{m=1}^\infty \frac{1}{1-2^{-m}}
	\le \prod_{m=1}^\infty (1+2^{-m})
	\le e^2 \,,
\end{equation}
because $|1-2^{k-i}| \ge 1$ for $k > i$
and $(1-x)^{-1} \le 1+2x \le e^{2x}$ for $0 \le x \le \frac{1}{2}$.
This bound proves \eqref{eq:boundhatvarphi},
by \eqref{eq:hatvarphi} and the fact that $\|\varphi^{\lambda_i}\|_{L^1}
= \|\varphi\|_{L^1}$.
\end{proof}

\begin{proof}[Proof of \Cref{th:moments}]
If $r=1$ equation \eqref{eq:moments} reduces to $\int \hat\varphi = \int \varphi$,
which holds because $\hat\varphi = \varphi^{\lambda_0}$
(recall that $c_0 = 1$ when $r=1$). Henceforth we
fix $r \in \N$ with $r \ge 2$.

Fix distinct 
$\lambda_0, \lambda_1, \ldots, \lambda_{r-1} \in (0,\infty)$
and define $c_0, c_1, \ldots, c_{r-1}$ by \eqref{eq:ci}.
Define $\hat\varphi$ by \eqref{eq:hatphi}.
For any multi-index $k\in\N_0^d$, 
since $y^k := y_1^{k_1} y_2^{k_2} \cdots y_d^{k_d}$, we can compute
\begin{equation*}
	\int_{\R^d} y^k \, \hat\varphi(y) \d y = 
	\sum_{i=0}^{r-1} c_i \, \int_{\R^d} y^k \,
	\lambda_i^{-d} \, \varphi(\lambda_i^{-1} y) \d y
	= \bigg( \sum_{i=0}^{r-1} c_i \, \lambda_i^{|k|} \bigg) 
	\int_{\R^d} x^k \, \varphi(x) \d x \,,
\end{equation*}
using the change of variables $y =\lambda_i x$.
Therefore $\hat\varphi$ fulfills the conditions in \eqref{eq:moments} if 
\begin{equation*}
	\sum_{i=0}^{r-1} c_i = 1 \qquad \text{and} \qquad
	\sum_{i=0}^{r-1} c_i \, \lambda_i^{|k|} = 0 \quad \text{for } 1 \le |k| \le {r-1} \,.
\end{equation*}
This is a linear system of $r$ equations, namely
\begin{equation*}
	A \, \left(\begin{matrix}
	c_0 \\ c_1 \\ c_2 \\ \vdots \\ c_{r-1}
	\end{matrix}\right) 
	=
	\left(\begin{matrix}
	1 \\ 0 \\ 0 \\ \vdots \\ 0
	\end{matrix}\right)
	\qquad \text{where} \qquad
	A := \left(\begin{matrix}
	1 & 1 & \ldots & 1 \\
	\lambda_0 \, & \lambda_1 \, & \ldots \, & \lambda_d \\
	\lambda_0^2 \, & \lambda_1^2 \, & \ldots \, & \lambda_d^2 \\
	\vdots \, & \vdots \, & \vdots \, & \vdots \\
	\lambda_0^{r-1} \, & \lambda_1^{r-1} \, & \ldots \, & \lambda_d^{r-1}
	\end{matrix}\right) \,.
\end{equation*}
Note that $A$ is a Vandermonde matrix with
$\det(A) = \prod_{0 \le i < j \le d} (\lambda_j - \lambda_i) \ne 0$, because
$\lambda_0, \lambda_1, \ldots, \lambda_{r-1}$ are all distinct. 
The inverse matrix $A^{-1}$ is explicit, see
equation~(7) (where a transpose is missing) in \cite{cf:Kli}:\footnote{See also
\url{https://proofwiki.org/wiki/Inverse_of_Vandermonde_Matrix}}
\begin{equation*}
	(A^{-1})_{ij} \ = \ (-1)^{j} \,
	\frac{\displaystyle \,
	\sum_{\substack{K \subseteq \{0,\ldots, r-1\} \setminus \{i\}\\
	|K| = r-1-j}} \ \prod_{k\in K} \lambda_k \, }
	{\displaystyle\prod_{k \in \{0,\ldots, r-1\} \setminus \{i\}} (\lambda_k - \lambda_i)}
	\qquad  \forall 0 \le i,j \le r-1 \,.
\end{equation*}
In particular, if we set $j=0$, we see that $c_i = (A^{-1})_{i0}$ is given by
\begin{equation*}
	c_i = \frac{\displaystyle \prod_{k \in \{0,\ldots,r-1\}
	\setminus \{i\}} \lambda_k}
	{\displaystyle
	\prod_{k \in \{0,\ldots, r-1\} \setminus \{i\}} (\lambda_k - \lambda_i)}
	= \prod_{k \in \{0,\ldots, {r-1}\}: \, k \ne i} \
	\frac{\lambda_k}{\lambda_k-\lambda_i} \,,
\end{equation*}
which matches \eqref{eq:ci}.
\end{proof}

\section{Basic estimates on convolutions}
\label{sec:basic-estimates-convolutions}

In this section we give two elementary but important Lemmas on convolutions.
We fix $r\in\N = \{1,2,\ldots\}$ and a test function $\hat\varphi = \hat\varphi^{[r]} \in \cD$ 
with the following
properties:
\begin{gather}\label{eq:supphatphi0}
	\hat\varphi \text{ is supported in } B(0,\tfrac{1}{2}) \,, \\
	\label{eq:monomials0}
	\int_{\R^d} y^k \, \hat\varphi(y) \d y = 0
	\quad \text{for} \ 1 \le |k| \le r-1 \,.
\end{gather}
We stress that \eqref{eq:monomials0} is \emph{not} required for $k=0$
(indeed, we typically want $\int \hat\varphi = 1$).

\begin{remark}\label{rem:fromphitohatphi}
Starting from an \emph{arbitrary} test function $\varphi\in\cD$, we can
define $\hat\varphi$ as in \Cref{th:moments},
for \emph{any} choice of distinct $(\lambda_i)_{i=0,\ldots, r-1}$ and $a\in\R$.
Then \eqref{eq:monomials0} holds by \eqref{eq:moments},
while \eqref{eq:supphatphi0} holds provided we choose
the $\lambda_i$'s small enough.
\end{remark}

Next we define 
\begin{equation}\label{eq:checkvarphi}
	\check{\varphi} := \hat\varphi^{\frac{1}{2}} 
	- \hat\varphi^{2},
\end{equation}
where  by $ \hat\varphi^{\frac{1}{2}},\hat\varphi^2$ we mean 
$\hat\varphi^\lambda(z) = \lambda^{-d} \hat\varphi(\lambda^{-1}z)$ for $\lambda = \frac12,2$, respectively. 
The function $\check{\varphi}$ will play an important role in the sequel.
It follows by \eqref{eq:supphatphi0} and \eqref{eq:monomials0} that
\begin{gather}\label{eq:suppcheckphi}
	\check\varphi \text{ is supported in } B(0,1) \,, \\
	\label{eq:checkmonomials}
	\int_{\R^d} y^k \, \check\varphi(y) \d y = 0
	\quad \text{for} \ 0 \le |k| \le r-1 \,.
\end{gather}
We stress that \eqref{eq:checkmonomials} holds also for $k=0$, because
$\int \hat\varphi^{\frac{1}{2}} = \int \hat\varphi^{2} = \int \hat\varphi^\lambda$
for any $\lambda$.

Our first Lemma concerns the convolution of a test function $\eta$ with $\check\varphi$.

\begin{lemma}\label{th:keyintest}
Fix a test function $\eta\in\cD(H)$ supported in a compact set $H \subseteq \R^d$.
Let $\check\varphi \in \cD$ satisfy \eqref{eq:suppcheckphi}
and \eqref{eq:checkmonomials}.
For any $\epsilon>0$,  the function $\check\varphi^{\epsilon} * \eta$ 
 is supported in
the $\epsilon$-enlargement $\bar H_{\epsilon}$ of $H$, see \eqref{eq:enlargement}, and 
\begin{equation}\label{neq:conclusio}
	\|\check\varphi^{\epsilon} * \eta\|_{L^1}  \le \mathrm{Vol}(\bar H_{\epsilon})\, \|\eta\|_{C^r} \,
	\|\check\varphi\|_{L^1} \, \epsilon^r \, .
\end{equation}
\end{lemma}
\begin{proof}
Since $\eta$ is supported in $H$ and $\check{\varphi}$ is supported in $B(0,1)$, 
then $\check{\varphi}^{\epsilon} * \eta$ is supported in $\bar H_{\epsilon}$.
Fix $y \in \bar H_{\epsilon}$ and
denote by $p_{y}(\cdot) := \sum_{|k| \le r-1} \frac{\partial^k \eta (y)}{k!} \, 
(\cdot -y)^k$ the Taylor polynomial of $\eta$ of order $r-1$
based at $y$, which satisfies for all $z\in\R^d$
\begin{equation} \label{neq:boundTaylor}
	|\eta(z) - p_{y}(z)| \le \|\eta\|_{C^r} \, |z-y|^r \,.
\end{equation}
It follows by \eqref{eq:checkmonomials}
that $\int_{\R^d} \check{\varphi}^{\epsilon}(y-z)
\, p_y(z) \d z = 0$, hence we can write
\begin{equation*}
\begin{split}
	(\check{\varphi}^{\epsilon} * \eta)(y) = \int_{\R^d} \check{\varphi}^{\epsilon}(y-z)
	\, \big\{ \eta(z) -
	p_{y}(z) \big\}  \d z \,.
\end{split}
\end{equation*}
Since $\check{\varphi}^{\epsilon}$ is supported in $B(0,\epsilon)$, by \eqref{neq:boundTaylor}
\begin{equation*}
\begin{split}
	|(\check{\varphi}^{\epsilon} * \eta)(y)| & \le \|\eta\|_{C^r} \,
	\int_{\R^d} |\check{\varphi}^{\epsilon}(y-z)| \,
	|z-y|^r \d z  \le \|\eta\|_{C^r} \,
	\|\check{\varphi}\|_{L^1} \, \epsilon^r \,.
\end{split}
\end{equation*}
This completes the proof of \eqref{neq:conclusio}.
\end{proof}

Our second Lemmas concerns convolutions of (scaled versions of) a test function $\psi$
with either $\hat\varphi$ or $\check\varphi$,
integrated against an arbitrary function $G$.

\begin{lemma}\label{th:keyintest2}
Let $\lambda, \epsilon>0$, $K\subset\R^d$ a compact set and $G:\R^d\to\R$ 
a measurable function. 
{Let $\hat\varphi, \check\varphi \in \cD$ satisfy \eqref{eq:supphatphi0}, \eqref{eq:monomials0} 
and \eqref{eq:suppcheckphi}, \eqref{eq:checkmonomials}, respectively.}
Then
for all $x\in K$ 
and $\psi\in\cB_r$, see \eqref{eq:cBr},
\begin{gather}
	\label{eq:Ge1}
	\left|\int_{\R^d} G(y) \, (\hat\varphi^{2\epsilon} * \psi_x^\lambda)(y) \, 
	\d y\right|\leq 2^d \,
	\|\hat\varphi\|_{L^1}\, \sup_{B(x,\lambda+{\epsilon})} |G|\, , \\ 
	\label{eq:Ge2}
	\left|\int_{\R^d} G(y) \, (\check\varphi^{\epsilon} * \psi_x^\lambda)(y) \d y\right|
	\leq 4^d\, 
	\|\check\varphi\|_{L^1}\, 
	\min\big\{\epsilon/\lambda,1\big\}^{r} \,
	\sup_{B(x,\lambda+{\epsilon})} |G|\, .
\end{gather}
\end{lemma}

\begin{proof}
Since $\hat\varphi$ and $\psi$ are supported in $B(0,1/2)$ and $B(0,1)$ respectively,
the function $\hat\varphi^{2\epsilon} * \psi_x^\lambda$ is 
supported in $B(x,\lambda+{\epsilon})$.
Then we can bound 
\[
	\left|\int_{\R^d} G(y) \, (\hat\varphi^{2\epsilon} * \psi_x^\lambda)(y) \d y\right|\leq 
	\|\hat\varphi^{2\epsilon} * \psi_x^\lambda\|_{L^1}\, 
	\sup_{B(x,\lambda+\epsilon)} |G|\, .
\]
Now
\begin{equation*}
	\|\hat\varphi^{2\epsilon} * \psi_x^\lambda\|_{L^1}
	\le \|\hat\varphi^{2\epsilon} \|_{L^1} \| \psi_x^\lambda\|_{L^1}
	\le 2^d \, \|\hat\varphi \|_{L^1} \,,
\end{equation*}
because $\|\hat\varphi^{2\epsilon} \|_{L^1} = \|\hat\varphi \|_{L^1}$ and
\eqref{eq:Ge1} is proved, because
\begin{equation}\label{eq:psiBr}
\sup_{\psi \in \cB_r} \|\psi_x^\lambda\|_{L^1} = \sup_{\psi \in \cB_r} \|\psi\|_{L^1} 
\le 2^d \, \sup_{\psi \in \cB_r} \|\psi\|_{\infty} \le 2^d\, ,
\end{equation}
since the volume of the unit ball in $\R^d$ is bounded above by $2^d$.
Analogously
\[
	\left|\int_{\R^d} G(y) \, (\check\varphi^{2\epsilon} * \psi_x^\lambda)(y) \d y\right|
	\leq 
	\|\check\varphi^{2\epsilon} * \psi_x^\lambda\|_{L^1}\, 
	\sup_{B(x,\lambda+\epsilon)} |G|\, .
\]
As in \eqref{eq:psiBr} we can bound
\begin{equation*}
	\|\check\varphi^{2\epsilon} * \psi_x^\lambda\|_{L^1} \le
	\|\check\varphi^{2\epsilon}\|_{L^1} \, \| \psi_x^\lambda\|_{L^1}
	= \|\check\varphi\|_{L^1} \, \| \psi\|_{L^1}
	\le 2^d \, \|\check\varphi\|_{L^1} \,,
\end{equation*}
which proves \eqref{eq:Ge2} for $\lambda \le \epsilon$.
When $\lambda > \epsilon$, we apply \eqref{neq:conclusio} to get
\[
	\|\check\varphi^{2\epsilon} * \psi_x^\lambda\|_{L^1}\leq 
	\mathrm{Vol}(B(x,\lambda+\epsilon)) \, 
	\|\psi_x^\lambda\|_{C^r}\,\epsilon^r\, \|\check\varphi\|_{L^1}.
\]
Note that $\mathrm{Vol}(B(x,\lambda+\epsilon))\leq (2(\lambda+\epsilon))^d
\leq 4^d \, \lambda^d $ for $\lambda > \epsilon$. Since $\psi\in\cB_r$, 
see \eqref{eq:cBr}, we can easily bound $\|\psi_x^\lambda\|_{C^r}$ by \eqref{eq:scale}:
\begin{equation*}
	\|\psi_x^\lambda\|_{C^r} = \max_{|k| \le r} 
	\|\partial^k (\psi_x^\lambda) \|_\infty
	= \max_{|k| \le r} 
	\| \lambda^{-|k|-d} (\partial^k \psi) \|_\infty
	\le \lambda^{-r-d}\,.
\end{equation*}
The proof of \eqref{eq:Ge2} is complete.
\end{proof}

\section{Proof of the Reconstruction Theorem for $\gamma > 0$}
\label{sec:reco+}

In this section  we prove \Cref{th:reco+} when $\gamma > 0$.
Given any $\gamma$-coherent germ $F = (F_x)_{x\in\R^d}$,
we show the existence of a distribution $f\in\cD'$ which satisfies \eqref{eq:reco+1}.
Uniqueness of $f$ follows by \Cref{th:identi}, because the right hand side of \eqref{eq:reco+1}
vanishes for $\gamma > 0$. Then linearity of the map $F\mapsto\cR F$ 
is a consequence of  uniqueness.

We now turn to existence. A large part of the proof actually holds for any $\gamma \in \R$,
only in the last steps we specialize to $\gamma > 0$.

\subsubsection*{Step 0. Setup}
We fix a $(\balpha,\gamma)$-coherent germ $(F_x)_{x\in\R^d}$ as in \Cref{def:coherent-germ},
for some $\balpha = (\alpha_K)$, with local
homogeneity bounds $\bbeta = (\beta_K)$ as in \Cref{th:boundor}. Without loss of generality,
we  suppose that
with $K\mapsto\alpha_K$ and $K\mapsto\beta_K$ are monotone as in \eqref{eq:monotone-alpha} and 
\eqref {eq:monotone-beta}. We will specify when we need to assume $\gamma > 0$.

We fix a compact set $K\subset\R^d$ and
define its $3/2$-fattening $\bar K_{3/2}$ as in \eqref{eq:enlargement}.
Throughout the proof we set
\begin{equation}\label{eq:alphabeta}
	\alpha:=\alpha_{\bar K_{3/2}} \,, \qquad \beta:=\beta_{\bar K_{3/2}} \,,
\end{equation}
so that \eqref{eq:coherent} and \eqref{eq:bounded-order-germ} hold on the compact set $\bar K_{3/2}$.
More explicitly, there are finite constants $C_1,C_2$ such that
for all $y, z \in \bar K_{3/2}$ with $|z-y| \le 2$ and $\epsilon \in (0,1]$ we have
\begin{gather} \label{eq:ineq1}
	|(F_z-F_y)(\varphi_y^\epsilon)| \le C_1 \, \epsilon^\alpha
	\, (|z-y|+\epsilon)^{\gamma-\alpha} \,, \qquad
	|F_y(\varphi_y^\epsilon)| \le C_2 \, \epsilon^{\beta} \,,
\end{gather}
and in fact we can choose $C_1 := \vertiii{F}^\co_{\bar K_{3/2}, \varphi, \alpha, \gamma}$.
We also fix an integer $r\in\N$ such that 
\begin{equation}\label{eq:erre}
	r=r_{\bar K_{3/2}} > \max\{-\alpha,-\beta\} \,.
\end{equation}

By \Cref{th:coherence-r}, we can build a ``tweaked'' test function $\hat\varphi = \hat\varphi^{[r]}$ 
which fulfills properties \eqref{eq:supphatphi} and \eqref{eq:monomials}, namely the support of 
$\hat\varphi$ is included in $B(0,1/2)$ and
\[
	\int_{\R^d} \hat\varphi(y) \d y = 1 \,, \qquad 
	\int_{\R^d} y^k \, \hat\varphi(y) \d y = 0
	\quad \text{for} \ 1 \le |k| \le r-1 \,.
\]
We claim that we can replace $\varphi$ by $\hat\varphi$ in \eqref{eq:ineq1} and obtain,
for all $y, z \in \bar K_{3/2}$ with $|z-y| \le 2$ and $\epsilon \in (0,1]$,
\begin{align} \label{eq:hyp1}
	|(F_z-F_y)(\hat\varphi_y^\epsilon)| & \le \hat C_1 \, \epsilon^\alpha
	\, (|z-y|+\epsilon)^{\gamma-\alpha} \,, 
 \\	\label{eq:hyp2} |F_y(\hat\varphi_y^\epsilon)| & \le \hat C_2 \, \epsilon^{\beta} \,,
\end{align}
where the constants $\hat C_1, \hat C_2$ are given by
\begin{equation} \label{eq:controlC12}
	\hat C_1 := \tfrac{e^2}{|\int \varphi|} \,
	r \, \big(\tfrac{2^{-r-1}}{1+R_\varphi}\big)^{\alpha} \,
	\vertiii{F}^\co_{\bar K_{3/2}, \varphi, \alpha, \gamma} \,, \qquad \hat C_2 :=
	\tfrac{e^2}{|\int \varphi|} \,
	r \, \big(\tfrac{2^{-r-1}}{1+R_\varphi}\big)^{\beta\wedge 0} \, C_2 \,,
\end{equation}
and $R_\varphi$ is such that $\varphi$ is supported in $B(0,R_\varphi)$.

Indeed, for every $\epsilon \in (0,1]$ and $i=0,\ldots, r-1$ we can estimate by \eqref{eq:hatvarphi}
\begin{equation*}
	(\epsilon\lambda_i)^\alpha
	\, (|z-y|+\epsilon\lambda_i)^{\gamma-\alpha}
	\le \big(\tfrac{2^{-r-1}}{1+R_\varphi}\big)^{\alpha} \
	\epsilon^\alpha
	\, (|z-y|+\epsilon)^{\gamma-\alpha} \,,
\end{equation*}
because $\frac{2^{-r-1}}{1+R_\varphi} < \lambda_i \le 1$
(recall that $\alpha\le 0$ and $\gamma \ge \alpha$, see \Cref{def:coherent-germ}).
Similarly
\begin{equation*}
	(\epsilon\lambda_i)^\beta \le 
	\big(\tfrac{2^{-r-1}}{1+R_\varphi}\big)^{\beta\wedge 0} \ \epsilon^\beta \,.
\end{equation*}
Plugging these bounds into \eqref{eq:ineq1},
by \eqref{eq:hatvarphi} and \eqref{eq:boundonci} we obtain
\eqref{eq:hyp1}-\eqref{eq:hyp2}-\eqref{eq:controlC12}.

\subsubsection*{Step 1. Strategy.}

We can now outline our strategy.
We use the mollifiers
\begin{equation*}
	\rho^\epsilon(z) = \epsilon^{-d} \rho(\epsilon^{-1} z)
\end{equation*}
where $\rho$ is defined as follows (recall that $\hat\varphi^2$ means 
$\hat\varphi^\lambda(z) = \lambda^{-d} \hat\varphi(\lambda^{-1}z)$
for $\lambda = 2$):
\begin{equation} \label{eq:rhoepsilon}
	\rho := \hat\varphi^2 * \hat\varphi \qquad \text{and} \qquad
	\epsilon = \epsilon_n := 2^{-n}, \quad n\in\N_0  \,.
\end{equation}
Note that $\int\rho=\int  \hat\varphi^2\int \hat\varphi=1$.

This peculiar choice of $\rho$ 
ensures that \emph{the difference $\rho^{\frac{1}{2}} - \rho$ is a convolution}:
\begin{equation} \label{eq:diffconv}
	\rho^{\frac{1}{2}} - \rho = \hat\varphi * \check{\varphi}
	\qquad \text{where we define} \qquad \check{\varphi} := \hat\varphi^{\frac{1}{2}} 
	- \hat\varphi^{2} \,,
\end{equation}
because $(f^\lambda)^{\lambda'} = f^{\lambda\lambda'}$
and $(f*g)^\lambda = f^\lambda * g^\lambda$,
see \eqref{eq:scale} and \eqref{eq:checkvarphi}. It follows that
\begin{equation} \label{eq:checkvarphi2}
	\rho^{\epsilon_{n+1}} - \rho^{\epsilon_n}
	= (\rho^{\frac{1}{2}} - \rho)^{\epsilon_n} 
	= \hat\varphi^{\epsilon_n} * \check{\varphi}^{\epsilon_n} \,.
\end{equation}
This will allow us to compare efficiently convolutions with $\rho^{\epsilon_{n+1}}$ and $\rho^{\epsilon_n}$.

We are ready to define a sequence of distributions
that will be shown to converge to a limiting
distribution $f\in\cD'$ which fulfills \eqref{eq:reco+1}.
To motivate the definition, note that
for any distribution $\xi \in \cD'$ and test function $\psi\in\cD$,
by \Cref{lem:sequential}, we have
\begin{equation*}
	\xi(\psi) = \lim_{n\to\infty} \ \xi(\rho^{\epsilon_n} * \psi)
	= \lim_{n\to\infty} \ \int_{\R^d} \xi(\rho^{\epsilon_n}_z) \, 
	\psi\, (z) \d z \,,
\end{equation*}
where we applied \eqref{eq:Tonconv}.
When we have a germ $F = (F_x)_{x\in\R^d}$ instead of a fixed distribution $\xi$,
a natural idea is to  replace
$\xi(\rho^{\epsilon_n}_z)$ by $F_z(\rho^{\epsilon_n}_z)$.
This leads to:

\begin{definition}[Approximating distributions]\label{def:fn}
Given a germ $F = (F_x)_{x\in\R^d}$,
for $n\in\N$ we define $f_n \in \cD'$ as follows:
\begin{equation} \label{eq:fn}
	f_n(\psi) := \int_{\R^d} F_z(\rho^{\epsilon_n}_z) \, 
	\psi\, (z) \d z \,, \qquad \psi \in \cD \,.
\end{equation}
\end{definition}

\begin{remark}
We recall that, by \Cref{def:germs}, the map $z\mapsto F_z(\rho^\epsilon)$ is measurable. 
Since the map
$z\mapsto\rho^{\epsilon}_z\in\cD$ is continuous, it follows that
the map $(z,y) \mapsto F_z(\rho_y^\epsilon)$ is jointly measurable
as pointwise limit of measurable maps:
$F_z(\rho_y^\epsilon) = \lim_{n\to\infty} F_z(\rho_{\lfloor ny \rfloor/n}^\epsilon)$,
where $\lfloor x \rfloor := (\lfloor x_1\rfloor,\ldots,\lfloor x_d\rfloor)$ and
$\lfloor a \rfloor:=\max\{n\in\Z: \ z \le a\}$ is the integer part of $a\in\R$.
In particular, $z\mapsto F_z(\rho^{\epsilon_n}_z)$ is measurable.
\end{remark}

\subsubsection*{Step 2. Decomposition}

Let us look closer at $f_n(\psi)$ in \eqref{eq:fn}.
We start with a telescopic sum:
\begin{equation}\label{eq:telesco}
	f_n(\psi) = f_1(\psi) + \sum_{k=1}^{n-1}
	g_k(\psi) \qquad \text{where} \qquad
	g_k(\psi) := f_{k+1}(\psi) - f_k(\psi) \,.
\end{equation}
We can write $g_k(\psi)
= \int_{\R^d} F_z(\rho^{\epsilon_{k+1}}_z - \rho^{\epsilon_k}_z) \, \psi(z)  \d z$
by \eqref{eq:fn} and then $F_z(\rho^{\epsilon_{k+1}}_z - \rho^{\epsilon_k}_z)
= \int_{\R^d} F_z(\hat\varphi^{\epsilon_k}_y) \, \check\varphi^{\epsilon_k}_z(y) \d y$,
by \eqref{eq:checkvarphi2}  and \eqref{eq:Tonconv},
which leads to the fundamental expression
\begin{equation*} 
\begin{split}
	g_k(\psi) = \int_{\R^d} \int_{\R^d} 
	F_z(\hat\varphi^{\epsilon_k}_y) \, \check\varphi^{\epsilon_k}(y-z) \,  \psi (z)
	\d y \d z \,.
\end{split}
\end{equation*}
If we write $F_z = F_y + (F_z-F_y)$ inside the last integral, we can decompose
\begin{equation}\label{eq:gkdeco}
\begin{split}
	g_k(\psi) & = \underbrace{\int_{\R^d} \int_{\R^d} 
	F_y(\hat\varphi^{\epsilon_k}_y) \, \check\varphi^{\epsilon_k}(y-z) \,  \psi (z)
	\d y \d z}_{g_k'(\psi)} \\
	& \qquad + \underbrace{\int_{\R^d} \int_{\R^d} 
	(F_z-F_y)(\hat\varphi^{\epsilon_k}_y) \, \check\varphi^{\epsilon_k}(y-z) \,  \psi (z)
	\d y \d z}_{g_k''(\psi)} \,.
\end{split}
\end{equation}
When we plug this into \eqref{eq:telesco}, we can write
\begin{gather}\label{eq:splitfn}
	f_n(\psi) = f_1(\psi) + f_n'(\psi) + f_n''(\psi) \,, \\
	\label{eq:fn'''}
	\text{where} \qquad
	f_n'(\psi) := \sum_{k=1}^{n-1} g_k'(\psi) \,, \qquad
	f_n''(\psi) := \sum_{k=1}^{n-1} g_k''(\psi) \,.
\end{gather}
In the next steps we proceed as follows. Recall that we  fixed
a compact set $K\subseteq\R^d$.
\begin{itemize}
\item In Step~3 we show that
\begin{equation} \label{eq:flim1}
	\forall \gamma \in \R: \quad
	f'(\psi) := \lim_{n\to\infty} f'_n(\psi) \quad \text{exists} \ \ \forall \, \psi \in \cD(\bar K_1)\, .
	\end{equation}
\item In Step~4 we show that
\begin{equation} \label{eq:flim2}
	\forall \gamma > 0: \quad
	f''(\psi) := \lim_{n\to\infty} f''_n(\psi) 
	\quad \text{exists} \ \ \forall \, \psi \in \cD(\bar K_1)\,.
\end{equation}
Then if $\gamma > 0$ the limit 
$f^K(\psi) := \lim_{n\to\infty} f_n(\psi)$ exists 
for $\psi \in \cD(\bar K_1)$ and equals
\begin{equation} \label{eq:flast}
	f^K(\psi)\ = \ f_1(\psi) + f'(\psi) + f''(\psi) \,, \qquad \psi \in \cD(\bar K_1) \,.
\end{equation}

\item 
In Step~5 we show that $f^K$ is a distribution on $\bar K_1$
which satisfies
\begin{equation}
	\label{eq:reco+1proof}
\begin{gathered}
	\forall \gamma > 0: \qquad
	| ( f^K - F_x)(\psi^\lambda_x) |
	\le \mathfrak{c} \,
	\vertiii{F}^\co_{\bar K_{3/2}, \varphi, \alpha,\gamma} \ \lambda^\gamma \\
	\phantom{\forall \gamma > 0: \qquad}
	\text{uniformly } \text{for } \psi \in \cB_r, \ x \in K, \ \lambda \in (0,1] \,,
\end{gathered}
\end{equation}
where the constant $\mathfrak{c}
= \mathfrak{c}_{\alpha, \gamma, r, d, \varphi}$ is given in \eqref{eq:frakc} below.

We stress that \emph{in principle $f^K(\psi)$ depends on the chosen compact set $K$},
because $f_n(\psi)$ depends on $\hat\varphi = \hat\varphi^{[r]}$,
see \eqref{eq:fn} and \eqref{eq:rhoepsilon}, and \emph{the value of $r$ depends on $K$
through $\alpha = \alpha_{\bar K_{3/2}}$, $\beta = \beta_{\bar K_{3/2}}$}, 
see \eqref{eq:erre} and \eqref{eq:alphabeta}.
In the special case when $\alpha_K = \alpha$ and $\beta_K = \beta$ for every $K$
(i.e.\ the germ $F$ is $(\alpha,\gamma)$-coherent with 
global homogeneity bound $\beta$),
then $f^K(\psi) = f(\psi)$ does not depend on $K$ and the proof is completed, because $f$
satisfies \eqref{eq:reco+1} in virtue of \eqref{eq:reco+1proof}.
In the general case, a small extra step is needed to complete the proof.

\item In Step~6 we show that for $\gamma > 0$ the distributions $f^K$ are consistent, i.e.
\begin{equation}\label{eq:consistent}
	\text{for $K \subseteq K'$:} \qquad
	f^K(\psi) = f^{K'}(\psi) \qquad \forall \psi \in \cD(\bar K_1) \,.
\end{equation}
This property lets us define a \emph{global} distribution
$f\in\cD'$ which satisfies \eqref{eq:reco+1}, thanks to  \eqref{eq:reco+1proof}. 
This concludes the proof for $\gamma > 0$.
\end{itemize}

\subsubsection*{Step 3. Proof of \eqref{eq:flim1} for $\gamma\in\R$}

By \eqref{eq:fn'''}, to prove \eqref{eq:flim1} it suffices to show that
\begin{equation} \label{eq:seriesgoalnew}
	\text{for all } \gamma \in \R: \qquad
	\sum_{k=1}^\infty |g'_{k}(\psi)| < \infty \,,
	\qquad \forall \, \psi \in \cD(\bar{K}_1) \,.
\end{equation}
Recall that
\begin{equation*}
\begin{split}
	g_k'(\psi) & = \int_{\R^d} \int_{\R^d} 
	F_y(\hat\varphi^{\epsilon_k}_y) \, \check\varphi^{\epsilon_k}(y-z) \,  \psi (z)
	\d y \d z 
	 = \int_{\R^d} 
	F_y(\hat\varphi^{\epsilon_k}_y) \, \check\varphi^{\epsilon_k}*  \psi (y)
	\d y \, .
	\end{split}
\end{equation*}
Note that $\check{\varphi} = \hat\varphi^{\frac{1}{2}} - \hat\varphi^2$ is supported in $B(0,1)$,
because $\hat\varphi$ is supported in $B(0,\frac{1}{2})$. 
Since $\psi$ is supported by $\bar{K}_1$ and $\check\varphi^{\epsilon_k}$ by $B(0,\epsilon_k)$ with $\epsilon_k\leq 1/2$, then $\check\varphi^{\epsilon_k}*  \psi$
is supported by $\bar K_{3/2}$. Then
\[
	|g_k'(\psi)| \leq \|\check\varphi^{\epsilon_k}*  \psi\|_{L^1}\, \sup_{y\in \bar K_{3/2}}|
	F_y(\hat\varphi^{\epsilon_k}_y) |\,.
\]
By \eqref{neq:conclusio} we have the bound 
\[
\|\check\varphi^{\epsilon_k}*  \psi\|_{L^1} \le \mathrm{Vol}(\bar K_{3/2})\, \|\psi\|_{C^r} \,\epsilon_k^r \,
	\|\check\varphi\|_{L^1}\, .
\]
By \eqref{eq:hyp2}, for all $y\in\bar K_{3/2}$ 
we have the bound $| F_y(\hat\varphi^{\epsilon}_y) | \le \hat C_2 \, \epsilon^{\beta}$. 
Then we obtain
\begin{equation}\label{eq:differenza}
	|g_k'(\psi)| 
	\le \big\{ \hat C_2 \, \mathrm{Vol}(\bar K_{3/2}) \,
	\|\check{\varphi}\|_{L^1} \,
	\|\psi\|_{C^r} \big\} \,  \epsilon_k^{{\beta} + r} \,.
\end{equation}
Since $\epsilon_k = 2^{-k}$ and ${\beta}+r >0$ by assumption, see \eqref{eq:erre},
we have $\sum_{k=1}^\infty |g_k'(\psi)| < \infty$ which completes
the proof of \eqref{eq:seriesgoalnew}.

\subsubsection*{Step 4. Proof of \eqref{eq:flim2} for $\gamma>0$}
By \eqref{eq:fn'''}, to prove \eqref{eq:flim2} it suffices to show that
\begin{equation} \label{eq:seriesgoalnew2}
	\text{if } \gamma > 0: \qquad	\sum_{k=1}^\infty |g''_{k}(\psi)| < \infty \,,
	\qquad \forall \, \psi \in \cD(\bar{K}_1) \,.
\end{equation}
Recall that
\begin{equation}\label{eq:gk''}
	g_k''(\psi) = \int_{\R^d} \int_{\R^d} 
	(F_z-F_y)(\hat\varphi^{\epsilon_k}_y) \, \check\varphi^{\epsilon_k}(y-z) \,  \psi (z)
	\d y \d z \,.
\end{equation}
We recall that $\check\varphi^{\epsilon_k}$
is supported in $B(0,\epsilon_k)$, so that
\[
	|g_k''(\psi)| \leq \|\check\varphi^{\epsilon_k}\|_{L^1}\, \|\psi \|_{L^1}\, 
	\sup_{z\in \bar K_1, |y-z|\leq \epsilon_k}
	|(F_z-F_y)(\hat\varphi^{\epsilon_k}_y)|\, ,
\]
with $\epsilon_k \le 1/2$ since $k\ge 1$.
Then \eqref{eq:hyp1} gives
\begin{equation*}
\sup_{z\in \bar K_1, |y-z|\leq \epsilon_k}	|(F_z-F_y)(\hat\varphi^{\epsilon_k}_y)|
	\le \hat C_1 \, \epsilon_k^\alpha \, (2 \epsilon_k)^{\gamma-\alpha} \,,
\end{equation*}
hence from \eqref{eq:gk''} we obtain
$|g_k''(\psi)| \le 2^{\gamma-\alpha} \, \hat C_1 \, \epsilon_k^\gamma \,\,
\| \check\varphi^{\epsilon_k} \|_{L^1} \, \| \psi\|_{L^1}$.
We finally observe that $\| \check\varphi^{\epsilon_k} \|_{L^1}
= \| \check\varphi \|_{L^1}$
by \eqref{eq:scale}. This gives the bound
\begin{equation} \label{eq:gk''est}
	|g_k''(\psi)| \le \big\{ 2^{\gamma-\alpha} \, \hat C_1\,
	\| \check\varphi \|_{L^1} \, \| \psi\|_{L^1} \big\}  \, \epsilon_k^\gamma \,.
\end{equation}
Since $\gamma > 0$ and $\epsilon_k = 2^{-k}$,
we obtain $\sum_{k=1}^\infty |g_k''(\psi)| < \infty$,
proving \eqref{eq:seriesgoalnew2}.

\subsubsection*{Step 5. Proof of \eqref{eq:reco+1proof}}

We showed in the previous
steps that both $f_n'(\psi)$ and 
$f_n''(\psi)$ converge for $\gamma > 0$.
Recalling \eqref{eq:splitfn}, we have that $f_n(\psi)$
converges to $f^K(\psi)$ given by \eqref{eq:flast}, i.e.\
\begin{equation*}
	f^K(\psi) = f_1(\psi) + \sum_{k=1}^\infty g_k'(\psi) + \sum_{k=1}^\infty  g_k''(\psi) \,.
\end{equation*}

\begin{remark}
By \eqref{eq:differenza} and \eqref{eq:gk''est}
there is $C = C_{K,\gamma,{\beta},r,\hat\varphi} < \infty$ such that
\begin{equation*}
	|f^K(\psi)| \le C \big\{ \|\psi\|_{L^1} + \|\psi\|_{C^r}\big\}
	\le C \{ \mathrm{Vol}(\bar K_{3/2}) + 1 \} \, \|\psi\|_{C^r}
	\qquad \text{for } \psi \in \cD(\bar{K}_1) \,.
\end{equation*}
This shows that $f^K \in \cD'(\bar{K}_1)$ is indeed a distribution on $\bar{K}_1$,
see \eqref{eq:order}.
\end{remark}

We now prove that $f^K(\cdot)$ satisfies 
\eqref{eq:reco+1proof}.
We fix a point $x\in K$ and define
\begin{equation*}
\begin{split}
	\tilde f(\psi) & := f^K(\psi)-F_x(\psi) \,,
	\qquad \psi\in\cD(\bar{K}_1).
\end{split}
\end{equation*}
We also define $\tilde f_n(\psi)$ similarly to $f_n(\psi)$ in \eqref{eq:fn},
just replacing $F_z$ by $F_z - F_x$:
\begin{equation}\label{neq:tildefn}
\tilde f_n(\psi):=	
	\int_{\R^d} (F_z - F_x)(\rho^{\epsilon_n}_z) \, 
	\psi\, (z) \d z
	\ = \ f_n(\psi) - F_x(\rho^{\epsilon_n} * \psi) \,,
\end{equation}
having used \eqref{eq:Tonconv}.
Since $F_x(\rho^{\epsilon_n} * \psi) \to F_x(\psi)$
by \Cref{th:molli} and \Cref{lem:sequential}, we have
\begin{equation} \label{eq:tildeflim}
	\tilde f(\psi)= \lim_{n\to\infty} \tilde f_n(\psi) \,.
\end{equation}

We now fix $\lambda \in (0,1]$ and  define
\begin{equation}\label{eq:Nlambda}
	N=N_{\lambda} := \min\{k\in\N: \ \epsilon_k \leq \lambda\}, 
\end{equation}
so that $N\geq 1$ and in particular
\begin{equation}\label{eq:epsla}
	\epsilon_{N} \leq \lambda < \epsilon_{{N}-1} = 2\epsilon_{N} \,.
\end{equation}
Let us now fix $\psi\in\cB_r$, see \eqref{eq:cBr}. 
By the triangle inequality we can bound
\begin{equation}\label{neq:ffN}
\begin{split}
	|\tilde f(\psi_x^\lambda)| 
	& \le |\tilde f_{N}(\psi_x^\lambda) |
	+ |(\tilde f-\tilde f_{N})(\psi_x^\lambda)| \,.
\end{split}
\end{equation}
We will estimate separately the two terms in the right-hand side.

\medskip
\noindent
\emph{First term in \eqref{neq:ffN}.}
By \eqref{neq:tildefn}, recalling \eqref{eq:rhoepsilon} and \eqref{eq:Tonconv}, 
we can write
\begin{equation} \label{eq:tildefN}
	\tilde f_N(\psi_x^\lambda) = 
	\int_{\R^d} \int_{\R^d} (F_z - F_x)(\hat\varphi^{\epsilon_N}_y) \, 
	\hat{\varphi}^{2\epsilon_N}(y-z) \,
	\psi_x^\lambda(z) \d y \d z \,.
\end{equation}
This integral is similar to \eqref{eq:gk''} and
we argue as in the proof of \eqref{eq:gk''est}. Recall that $\hat\varphi$ 
has support in $B(0,\frac{1}{2})$. Then
$\hat\varphi^{2\epsilon_N}$
has support in $B(0,\epsilon_N)$ and we may assume that $|y-z| \le \epsilon_N \le \frac{1}{2}$
in the right-hand side of \eqref{eq:tildefN}. Since $\psi_x^\lambda$ 
is supported in $B(x,\lambda)\subset\bar{K}_1$, we can assume that
$|z-x|\leq\lambda$, hence $z\in\bar K_1$ and $y\in\bar K_{3/2}$. Then
\[
	|\tilde f_N(\psi_x^\lambda)| \leq \|\hat{\varphi}^{2\epsilon_N}\|_{L^1}\, 
	\|\psi_x^\lambda\|_{L^1}\, 
	\sup_{z\in B(x,\lambda), |y-z| \le \epsilon_N} |(F_z - F_x)(\hat\varphi^{\epsilon_N}_y)|.
\]
By the triangle inequality $|(F_z-F_x)(\hat\varphi^{\epsilon_N}_y)|\leq 
|(F_z-F_y)(\hat\varphi^{\epsilon_N}_y)| +|(F_y-F_x)(\hat\varphi^{\epsilon_N}_y)|$,
and since \eqref{eq:hyp1} and \eqref{eq:epsla} give
\begin{gather*}
	\sup_{z\in B(x,\lambda), |y-z| \le \epsilon_N}	|(F_z-F_y)(\hat\varphi^{\epsilon_N}_y)|
	\le \hat C_1 \, \epsilon_N^\alpha \, (2 \epsilon_N)^{\gamma-\alpha}\leq 
	\hat C_1 \, 2^{\gamma-\alpha} \,
	\lambda^\gamma \,, \\
	\sup_{z\in B(x,\lambda), |y-z| \le \epsilon_N}	|(F_y-F_x)(\hat\varphi^{\epsilon_N}_y)|
	\le \hat C_1 \, \epsilon_N^\alpha \, (\lambda+ 2\epsilon_N)^{\gamma-\alpha} 
	\le \hat C_1 \, 4^{\gamma-\alpha} \,
	\lambda^\gamma \,,
\end{gather*}
we obtain 
\[
	|\tilde f_N(\psi_x^\lambda)| \le 2\cdot 4^{\gamma-\alpha} \, \hat C_1 \, \lambda^\gamma \,
	\| \hat\varphi^{2\epsilon_N} \|_{L^1} \, \| \psi_x^\lambda\|_{L^1}\, .
\]
We can easily bound $\|\psi_x^\lambda\|_{L^1} \le2^d$ for $\psi \in \cB_r$, see  \eqref{eq:psiBr},
and $\| \hat\varphi^{2\epsilon_N} \|_{L^1} = \| \hat\varphi \|_{L^1}$. All this yields 
the following estimate for the first term $|\tilde f_N(\psi_x^\lambda)|$
in \eqref{neq:ffN}
\begin{equation} \label{eq:st1ep}
\begin{split}
	|\tilde f_N(\psi_x^\lambda)|
	& \le \{4^{\gamma-\alpha} \, 2^{d+1} \} \,
	\|\hat{\varphi}\|_{L^1} \, \hat C_1 \,\lambda^{\gamma} \,.
\end{split}
\end{equation}

\medskip
\noindent
\emph{Second term in \eqref{neq:ffN}.}
Next we bound, by \eqref{eq:tildeflim},
\begin{equation}\label{eq:drasticbound}
	|(\tilde f-\tilde f_{N})(\psi_x^\lambda)| \le
	\sum_{k\ge N} |(\tilde f_{k+1} - \tilde f_k)(\psi_x^\lambda)| \,.
\end{equation}
Recalling \eqref{neq:tildefn} and \eqref{eq:splitfn}-\eqref{eq:fn'''}, we 
can write
\begin{equation} \label{eq:nede}
\begin{split}
	(\tilde f_{k+1} - \tilde f_{k})(\psi_x^\lambda)
	&= (f_{k+1}- f_{k})(\psi_x^\lambda)
	- F_x\big( (\rho^{\epsilon_{k+1}} - \rho^{\epsilon_k})*\psi_x^\lambda \big) \\
	& = \underbrace{\,g'_k(\psi_x^\lambda) - 
	F_x\big( (\rho^{\epsilon_{k+1}} - \rho^{\epsilon_k})*\psi_x^\lambda \big)\,}_{A_k^\lambda}
	\,+\, \underbrace{\,g''_k(\psi_x^\lambda)\,}_{B_k^\lambda} \,.
\end{split}
\end{equation}
We now look at $A_k^\lambda$ and $B_k^\lambda$.
The estimates for $A_k^\lambda$ hold for any $\gamma \in \R$ and will be useful
in \Cref{sec:reco+neg} for the case $\gamma \le 0$, hence we state
them as a separate result.

\begin{lemma}\label{th:Ak}
Define $A_k^\lambda$ as in \eqref{eq:nede}. For any $\gamma \in \R$ we have
\begin{equation} \label{eq:Alambdak}
	|A^\lambda_k| \le 4^{d+\gamma-\alpha} \, \hat C_1 \,
	\|\check{\varphi}\|_{L^1} \cdot
	\begin{cases}
	\lambda^{\gamma - \alpha - r}\, 
	\epsilon_k^{\alpha+r}  & \text{if } \epsilon_k < \lambda \\
	\epsilon_k^{\gamma} & \text{if } \epsilon_k \ge \lambda
	\end{cases} \,,
\end{equation}
and for $N = N_\lambda$ in \eqref{eq:Nlambda} we have
\begin{equation}\label{eq:tobeusedlater}
	\sum_{k\ge N} |A^\lambda_k|
	\le \frac{4^{d+\gamma-\alpha}}{1-2^{-\alpha-r}} \, \hat C_1 \,
	\|\check{\varphi}\|_{L^1} \, \lambda^{\gamma}\,.
\end{equation}
\end{lemma}

\begin{proof}
By \eqref{eq:Tonconv},
together with the crucial property \eqref{eq:checkvarphi2} of 
$\rho^{\epsilon_{k+1}} - \rho^{\epsilon_k}$, we can write
\begin{equation*}
\begin{split}
	F_x\big( (\rho^{\epsilon_{k+1}} - \rho^{\epsilon_k})*\psi_x^\lambda \big)
	& = \int_{\R^d} F_x( \rho^{\epsilon_{k+1}}_z - \rho^{\epsilon_k}_z ) \, \psi_x^\lambda(z) \d z  \\
	& = \int_{\R^d} \int_{\R^d} 
	F_x(\hat\varphi^{\epsilon_k}_y) \, \check\varphi^{\epsilon_k}(y-z) \,  \psi_x^\lambda (z)
	\d y \d z \,.
\end{split}
\end{equation*}
Recalling the definition \eqref{eq:gkdeco} of $g_k'$,
we obtain
\begin{equation*}
\begin{split}
	A^\lambda_k & :=\int_{\R^d} \int_{\R^d} (F_y - F_x)(\hat\varphi^{\epsilon_k}_y) \, 
	\check{\varphi}^{\epsilon_k}(y-z) \,
	\psi_x^\lambda(z) \d y \d z 
	\\ & = \int_{\R^d} (F_y - F_x)(\hat\varphi^{\epsilon_k}_y) \, 
	(\check{\varphi}^{\epsilon_k}* 
	\psi_x^\lambda)(y) \d y\, .
	\end{split}
\end{equation*}
If we define $G(y):=(F_y - F_x)(\hat\varphi^{\epsilon}_y)$, 
we see that $|A^\lambda_k|$ can be estimated as in \eqref{eq:Ge2}, which yields
\[
	|A^\lambda_k|\leq 4^d\, 
	\|\check\varphi\|_{L^1}\, 
	\min\big\{\epsilon_k/\lambda ,1\big\}^r \,
	\sup_{y \in B(x,\lambda+\epsilon_k)} 
	|(F_y - F_x)(\hat\varphi^{\epsilon_k}_y)|\, .
\]
For $y\in B(x,\lambda+\epsilon_k)$, by \eqref{eq:hyp1} we have
\begin{equation*}
\begin{split}
	| (F_x-F_y)(\hat\varphi^{\epsilon_k}_y) | \le
	\hat C_1 \, \epsilon_k^\alpha(|x-y|+\epsilon_k)^{\gamma-\alpha}
	\leq 
	\hat C_1
	\left(\lambda+2\epsilon_k\right)^{\gamma-\alpha} 
	\epsilon_k^{\alpha} \, ,
\end{split}\end{equation*}
which proves \eqref{eq:Alambdak} because $\left(\lambda+2\epsilon_k\right)^{\gamma-\alpha}
\le 3^{\gamma-\alpha} \,  
\max\{\epsilon_k, \lambda\}^{\gamma-\alpha}$.

We next turn to \eqref{eq:tobeusedlater}.
For $k\geq N$ we have $\epsilon_k\leq\lambda$, see \eqref{eq:epsla}, hence we can apply
the first line in \eqref{eq:Alambdak}.
Since $\alpha+r > 0$ by assumption, we have
\begin{equation*}
	\sum_{k\ge N}\epsilon_k^{\alpha+r}
	= \frac{\epsilon_N^{\alpha+r}}{1-2^{-\alpha-r}}
	\le \frac{\lambda^{\alpha+r}}{1-2^{-\alpha-r}} \,,
\end{equation*}
therefore from \eqref{eq:Alambdak} we obtain \eqref{eq:tobeusedlater}.
\end{proof}

We next focus on $B_k^\lambda = g''_k(\psi_x^\lambda)$ in \eqref{eq:nede}.
Recalling the definition \eqref{eq:gkdeco} of $g''_k$, we can write
\[
	B^\lambda_k:=\int_{\R^d} \int_{\R^d} (F_z - F_y)(\hat\varphi^{\epsilon_k}_y) \, 
	\check{\varphi}^{\epsilon_k}(y-z) \,
	\psi_x^\lambda(z) \d y \d z\,.
\]
Since $\psi$ and $\check\varphi$ are both supported in $B(0,1)$,
we can suppose that $|z-x|\le\lambda$ 
and $|y-z|\leq \epsilon_k \leq 1/2$ and therefore $z\in\bar K_1$, $y\in\bar K_{3/2}$. Then
\[
	|B^\lambda_k| \leq \|\check\varphi^{\epsilon_k}\|_{L^1}\, \|\psi_x^\lambda \|_{L^1}\, 
	\sup_{z\in \bar K_1, |y-z|\leq \epsilon_k}
	|(F_z-F_y)(\hat\varphi^{\epsilon_k}_y)|\, .
\]
By \eqref{eq:hyp1}  we have the bound
 \[
 	\sup_{z\in \bar K_1, |y-z|\leq \epsilon_k}|(F_z-F_y)(\hat\varphi^{\epsilon_k}_y)|
	\le \hat C_1 \, \epsilon_k^\alpha \, (2\epsilon_k)^{\gamma-\alpha} \,,
\]
and therefore, since $\|\psi_x^\lambda\|_{L^1} \le2^d$ for $\psi \in \cB_r$ by \eqref{eq:psiBr}, 
\begin{equation*}
	|B^\lambda_k| \le 2^{\gamma-\alpha} \, 2^{d}\, \hat C_1 \,
	\| \check\varphi \|_{L^1} \, \epsilon_k^\gamma\,.
\end{equation*}
Note now that $\gamma>0$ here, so that 
\[
	\sum_{k\ge N}\epsilon_k^{\gamma}
	= \frac{\epsilon_N^{\gamma}}{1-2^{-\gamma}} \le
	\frac{\lambda^{\gamma}}{1-2^{-\gamma}} \,,
\]
which yields
\begin{equation} \label{eq:sumb}
	\sum_{k\ge N} |B_k^\lambda| \le 
	\frac{2^{\gamma-\alpha} \, 2^{d}}{1-2^{-\gamma}} \, \hat C_1 \,
	\| \check\varphi \|_{L^1} \, 
	\lambda^{\gamma} \,.
\end{equation}

Recalling \eqref{eq:drasticbound} and \eqref{eq:nede}, we obtain from
\eqref{eq:tobeusedlater} and \eqref{eq:sumb} the desired estimate for
the second term in \eqref{neq:ffN}:
\begin{equation} \label{eq:st2ep}
	|\big(\tilde f-\tilde f_{N}\big)(\psi_x^\lambda)|
	\le \frac{ 2  \cdot 4^{d+\gamma-\alpha}}{1-2^{-\gamma\wedge(\alpha+r)}} 
	\, \hat C_1 \, \|\check{\varphi}\|_{L^1} \, \lambda^{\gamma} \,.
\end{equation}

\medskip
\noindent
\emph{Conclusion}.
At last, we can gather \eqref{neq:ffN}, \eqref{eq:st1ep} and \eqref{eq:st2ep}.
We estimate $\|\check{\varphi}\|_{L^1} \le 2 \, \|\hat{\varphi}\|_{L^1}$ by \eqref{eq:diffconv}, to get
\begin{equation*}
	|(f^K-F_x)(\psi_x^\lambda)| \le 
	\frac{2\cdot 4^{d+1+\gamma-\alpha}}{1-2^{-\gamma\wedge(\alpha+r)}} \,
	\hat C_1 \, \|\hat{\varphi}\|_{L^1} \, \lambda^{\gamma} \,.
\end{equation*}
If we estimate $\|\hat{\varphi}\|_{L^1}$ 
using \eqref{eq:boundhatvarphi}
and $\hat C_1$ 
using \eqref{eq:controlC12}, we obtain
\begin{equation}\label{eq:phiC} 
	\hat C_1\, \| \hat\varphi \|_{L^1} \le \big(\tfrac{e^2}{|\int \varphi|} \, r\big)^2
	\big(\tfrac{2^{-r-1}}{1+R_\varphi}\big)^{\alpha} \, \|\varphi\|_{L^1} \, 
	\vertiii{F}^\co_{\bar K_{3/2}, \varphi, \alpha, \gamma} \,.
\end{equation}
If we bound $e \le 4$ for simplicity, we obtain finally \eqref{eq:reco+1proof} with 
\begin{equation}\label{eq:frakc}
	{\mathfrak c}= \mathfrak{c}_{\alpha, \gamma, r, d, \varphi} =
	\frac{r^2\,2^{(r+1)\alpha} \, 4^{d+\gamma-\alpha+6}}
	{1-2^{-\gamma\wedge(\alpha+r)}} \,
	\frac{\|\varphi\|_{L^1} \, (1+R_\varphi)^{-\alpha}}{|\int \varphi|^2}
	\qquad \text{(for $\gamma > 0$)} 
\end{equation}
where $R_\varphi$ is the radius of a ball $B(0,R_\varphi)$ which contains the support of $\varphi$.

\subsubsection*{Step 6. Proof of \eqref{eq:consistent}}
We finally show that the distributions $f^K$ built in the previous steps are consistent, namely
for $K \subseteq K'$ and
for all test functions $\psi \in \cD(\bar{K}_1)$ that are supported in $\bar{K}_1$
 we have $f^{K'}(\psi) = f^K(\psi)$.
This is an immediate consequence of \Cref{th:identi},
because if we fix any $\xi \in \cD(\bar{K}_1)$ with $\int\xi \ne 0$
it follows by \eqref{eq:reco+1proof} with $\psi = \xi$ that both $f^K$ and $f^{K'}$
satisfy \eqref{eq:pre-coherent} with $\varphi = \xi$ on the
compact set $\bar{K}_1$.

We can finally define a \emph{global} distribution $f\in\cD'$:
given any test function $\psi \in \cD$, we pick a compact set $K$ large enough
so that $\psi \in \cD(\bar{K}_1)$ and we define $f(\psi) := f^K(\psi)$
(this is well-posed thanks
to the consistency relation \eqref{eq:consistent} that we have just proved).
Then, for any compact set $K$, we can replace $f^K$ by $f$ in \eqref{eq:reco+1proof},
which shows that $f$ satisfies \eqref{eq:reco+1}.
This completes the proof of \Cref{th:reco+} for $\gamma > 0$.\qed

\section{Proof of the Reconstruction Theorem for $\gamma \le 0$}
\label{sec:reco+neg}

In this section we prove \Cref{th:reco+} when $\gamma \le 0$.
We stress that we do not have a unique choice for the reconstruction $\cR F$,
because relation \eqref{eq:reco+1} for $\gamma \le 0$ does not characterize $f$ uniquely,
see \Cref{th:identi} above and \Cref{rem:well-def} below. 

Henceforth we fix a germ $F = (F_x)_{x\in\R^d}$ which is $\gamma$-coherent with $\gamma \le 0$.
In order to find a correct choice of $\cR F$,
we start following the proof of the case $\gamma > 0$, see \Cref{sec:reco+}.
We fix a compact set $K \subseteq \R^d$ and we fix $\alpha,\beta,r$ as in \eqref{eq:alphabeta}-\eqref{eq:erre}. 
The key problem when $\gamma \le 0$ is that the sequence of approximating distributions $f_n$
that we defined in \eqref{eq:fn} will typically \emph{not} converge,
hence we can no longer define $f^K := \lim_{n\to\infty} f_n$.
More precisely, if we recall the decomposition
\begin{equation*}
	f_n(\psi) = f_1(\psi) + f_n'(\psi) + f_n''(\psi) \,, \qquad \psi\in\cD(\bar K_1)\, ,
\end{equation*}
see \eqref{eq:telesco}-\eqref{eq:fn'''}, then it is the term $f''_n(\psi)$ which can fail to converge 
for $\gamma \le 0$, since
the proof in Step 4 was based on \eqref{eq:gk''est} 
and exploited $\gamma>0$.
On the other hand, we showed that $f'(\psi) := \lim_{n\to\infty} f'_n(\psi)$ exists for every
$\gamma \in \R$, see \eqref{eq:flim1}. 

Therefore, for $\gamma \le 0$ the idea is to suppress $f''_n(\psi)$. Recalling \eqref{eq:fn'''},
we thus set
\begin{equation} \label{eq:fle0}
\begin{split}
	f^K(\psi) & = f_1(\psi) + f'_1(\psi) =
	f_1(\psi) + \sum_{k=1}^\infty g_k'(\psi) \,, 
	\qquad \psi\in\cD(\bar K_1) \,.
\end{split}
\end{equation}
We complete the proof in two steps, that we now describe.
\begin{itemize}
\item In Step I we show that $f^K \in \cD'(\bar{K}_1)$
is a distribution on $\bar{K}_1$ which satisfies 
\begin{equation}
	\label{eq:reco+1proof2<}
	\begin{gathered}
	| ( f^K - F_x)(\psi^\lambda_x) |
	\le {\mathfrak{C}} \,
	\vertiii{F}^\co_{\bar K_{3/2}, \varphi, \alpha,\gamma} \cdot 
	\begin{cases}
	\lambda^\gamma & \text{if } \gamma < 0 \\
	\big(1+|\log \lambda|\big) & \text{if } \gamma = 0
	\end{cases} \\
	\text{uniformly } \text{for } \psi \in \cB_r, \ x \in K, \ \lambda \in (0,1] \,,
	\end{gathered}
\end{equation}
for a suitable $\mathfrak{C}$ given below, see \eqref{eq:unfr} for $\gamma < 0$
and \eqref{eq:unfr2} for $\gamma = 0$.
\end{itemize}

\begin{remark}
\emph{We stress that in general $f^K(\psi)$ depends on the compact set $K$}.
Indeed, $f_1(\psi)$ and $g_k'(\psi)$ depend on $\hat\varphi = \hat\varphi^{[r]}$
and the value of $r > \max\{-\alpha,-\beta\} := \max\{-\alpha_{\bar{K}_{3/2}},-\beta_{\bar{K}_{3/2}}\}$ 
is a function of $K$, see \eqref{eq:erre} and \eqref{eq:alphabeta}.
\end{remark}

We first consider the special case when 
\emph{the germ $F$ is $(\alpha,\gamma)$-coherent with 
global homogeneity bound $\beta$},
that is when $\alpha_K=\alpha$ and $\beta_K=\beta$ for every compact set $K$.
Then we can choose a fixed $r > \max\{-\alpha,-\beta\}$ and
$f^K(\psi) = f(\psi)$ does not depend on $K$, hence
replacing $f^K$ by $f$ in \eqref{eq:reco+1proof2<} we obtain precisely \eqref{eq:reco+1}.

It remains to show that the map $F \mapsto f =: \cR F$ is linear
(we recall that the family of $(\alpha,\gamma)$-coherent germs
with global homogeneity bound $\beta$ is a vector space, see \Cref{rem:possfin}).
This follows easily by the definition \eqref{eq:fle0} of $f^K = f$,
because both $f_1$ and $g'_k$ are linear functions of $F$, see \eqref{eq:fn}
and \eqref{eq:gkdeco}.
We have thus completed the proof of \Cref{th:reco+} for $\gamma \le 0$ in this special case.

We finally go back to the general case when $\alpha_K$ and $\beta_K$ may depend
on $K$, hence $f^K$ also depends on $K$.
We complete the proof of \Cref{th:reco+} for $\gamma \le 0$ as follows.
\begin{itemize}
\item In Step II we build a global distribution $f\in\cD'$ out of the $f^K$'s, by a localisation 
argument based on a partition of unity,
and we show that $f$ satisfies \eqref{eq:reco+1}.
\end{itemize}
It only remains to prove Steps~I and~II.

\subsubsection*{Step I. Proof of \eqref{eq:reco+1proof2<}}
Let us outline the strategy we are going to follow.

We have fixed a compact set $K \subseteq \R^d$. We now fix a point $x \in K$.
By \Cref{lem:sequential} and \Cref{th:molli}
we have $F_x(\psi) = \lim_{n\to\infty} F_x(\rho^{\epsilon_n} * \psi)$.
In view of \eqref{eq:fle0}, we define
\begin{equation}\label{eq:tildefn}
\begin{split}
	\bar f_{n}(\psi) 
	& := \Bigg( f_1(\psi) + \sum_{k=1}^{n-1} g_k'(\psi) \Bigg)  \,-\, F_x(\rho^{\epsilon_n} * \psi) 
\end{split}
\end{equation}
so that we can write
\[
	\bar f(\psi) := (f^K-F_x)(\psi) = \lim_{n\to\infty} \bar f_{n}(\psi)\, .
\]
We now fix $\lambda \in (0,1]$ and replace $\psi$ by $\psi_x^\lambda$.
By the triangle inequality, we get
\begin{equation}\label{neq:ffN2}
	|\bar f(\psi_x^\lambda)| \le 
	|(\bar f-\bar f_{N})(\psi_x^\lambda)| + |\bar f_{N}(\psi_x^\lambda) | \,,
\end{equation}
where $N\geq 1$, defined in \eqref{eq:Nlambda}, is such that
(we recall that $\epsilon_k=2^{-k}$)
\begin{equation*}
	\epsilon_{N} \leq \lambda < \epsilon_{{N}-1} = 2\,\epsilon_{N} \,.
\end{equation*}
We estimate the two terms in the right hand side of \eqref{neq:ffN2} separately.

\medskip
\noindent
\emph{First term in \eqref{neq:ffN2}.}
We bound
\begin{equation*}
	|(\bar f-\bar f_{N})(\psi_x^\lambda)|
	\le \sum_{k\ge N} |(\bar f_{k+1} - \bar f_k)(\psi_x^\lambda)| \,.
\end{equation*}
By \eqref{eq:tildefn} we can write
\begin{equation*}
	(\bar f_{k+1} - \bar f_k)(\psi_x^\lambda) \,=\, 
	g'_k(\psi_x^\lambda) - 
	F_x\big( (\rho^{\epsilon_{k+1}} - \rho^{\epsilon_k})*\psi_x^\lambda \big)
	\,=\, A_k^\lambda \,,
\end{equation*}
where the term $A_k^\lambda$ was defined in \eqref{eq:nede}.
We can then apply \Cref{th:Ak}, which holds also for $\gamma \le 0$: in particular,
by relation \eqref{eq:tobeusedlater} we obtain
\begin{equation}\label{eq:st2epp}
	|(\bar f-\bar f_{N})(\psi_x^\lambda)|
	\le \frac{4^{d+\gamma-\alpha}}{1-2^{-\alpha-r}} \, \hat C_1 \,
	\|\check{\varphi}\|_{L^1} \, \lambda^{\gamma}\,.
\end{equation}

\medskip
\noindent
\emph{Second term in \eqref{neq:ffN2}.} Since $N \ge 1$, we can bound
\begin{equation} \label{eq:babo}
	|\bar f_{N}(\psi_x^\lambda) | \le
|\bar f_{1}(\psi_x^\lambda) | + \sum_{k=1}^{N-1}
	|(\bar f_{k+1} - \bar f_k)(\psi_x^\lambda) | \,.
\end{equation}
For $k\le N-1$ we have $\epsilon_k \ge \epsilon_{N-1}
\ge \lambda$,  therefore
by the second line in \eqref{eq:Alambdak}
\begin{equation} \label{eq:usto}
	|(\bar f_{k+1} - \bar f_k)(\psi_x^\lambda) |  \le
	4^{d+\gamma-\alpha} \, \hat C_1 \,
	\|\check{\varphi}\|_{L^1} \, \epsilon_k^\gamma \, .
\end{equation}

Next we estimate $|\bar f_{1}(\psi_x^\lambda) | $.
By \eqref{eq:Tonconv} we have $F_x( \rho^{\epsilon_{1}} * \psi )
= \int_{\R^d} F_x( \rho^{\epsilon_{1}}_z ) \, \psi(z) \d z$.
Recalling \eqref{eq:tildefn} and the definitions \eqref{eq:fn}, \eqref{eq:rhoepsilon} of $f_1$
and $\rho$, we obtain
\begin{equation*}
\begin{split}
	\bar f_1(\psi_x^\lambda) 
	& = f_1(\psi)  \,-\, F_x(\rho^{\epsilon_n} * \psi)
	= \int_{\R^d} (F_y-F_x)(\rho^{\epsilon_1}_y) 
	\,\psi_x^\lambda(y) \d y \\
	& = \int_{\R^d} \int_{\R^d} (F_z - F_x)(\hat\varphi_y^{\epsilon_1}) \, 
	\hat{\varphi}^{2 \epsilon_1}(y-z) \,
	\psi_x^\lambda(z) \d y \d z \,.
\end{split}
\end{equation*}
Since $\epsilon_k=2^{-k}$ and $\hat\varphi$ has support in $B(0,\frac{1}{2})$, 
then $\hat\varphi^{2 \epsilon_1}$
has support in $B(0,\epsilon_1) = B(0,\frac{1}{2})$ and we may assume that 
$|y-z| \le \frac{1}{2}$. 
Since $\psi_x^\lambda$ is supported in $B(x,\lambda)\subset\bar{K}_1$, we can assume that
$|z-x|\leq\lambda$ and $z\in\bar K_1$, $y\in\bar K_{3/2}$. Then
\[
	|\bar f_1(\psi_x^\lambda) | 
	\leq \|\hat{\varphi}^{2 \epsilon_1}\|_{L^1}\, \|\psi_x^\lambda\|_{L^1}\, 
	\sup_{z\in B(x,\lambda), \, |y-z| \le \frac{1}{2}} |(F_z - F_x)(\hat\varphi_y)|.
\]
Moreover \eqref{eq:hyp1} for $\epsilon=1$ gives
\begin{gather*}
	\sup_{z\in B(x,\lambda), \, |y-z| \le \frac{1}{2}}	|(F_z-F_y)(\hat\varphi_y)|
	\le \hat C_1 \, (|y-z|+1)^{\gamma-\alpha}
	\le \hat C_1 \, 2^{\gamma-\alpha} \, ,\\
	\sup_{z\in B(x,\lambda), \, |y-z| \le \frac{1}{2}}	|(F_y-F_x)(\hat\varphi_y)|
	\le \hat C_1 \, (|y-x|+1)^{\gamma-\alpha}
	\le \hat C_1 \, 3^{\gamma-\alpha} \,,
\end{gather*}
therefore by the triangular inequality
\begin{equation} \label{eq:Clambdak1}
	|\bar f_1(\psi_x^\lambda) | 
	\le  2 \, \hat C_1\, 3^{\gamma-\alpha} \, \|\psi_x^\lambda\|_{L^1} \, 
	\|\hat{\varphi}^{2 \epsilon_1}\|_{L^1}
	\le \hat C_1\, 3^{\gamma-\alpha}
	\, 2^{d+1}\, \|\hat{\varphi}\|_{L^1} \, ,
\end{equation}
since $\|\psi_x^\lambda\|_{L^1} \le2^d$ for $\psi \in \cB_r$, by \eqref{eq:psiBr},
and $\|\hat{\varphi}^{2 \epsilon_1}\|_{L^1} = \|\hat{\varphi}\|_{L^1}$.
We can finally estimate $|\bar f_{N}(\psi_x^\lambda) |$ by \eqref{eq:babo}.
We get by \eqref{eq:usto} and \eqref{eq:Clambdak1}
\begin{equation} \label{eq:quaco1}
	|\bar f_{N}(\psi_x^\lambda) | \le 
	4^{d+\gamma-\alpha} \, \hat C_1 \,
	\|\check{\varphi}\|_{L^1} \, 
	\sum_{k=0}^{N-1}  \epsilon_k^\gamma \,.
\end{equation}
Recalling that $\epsilon_k = 2^{-k}$ 
and $\epsilon_N=2^{-N} \ge \lambda/2$, we obtain for $\gamma\leq 0$
\begin{equation}\label{eq:quaco2}
	\sum_{k=0}^{N-1}  \epsilon_k^\gamma
	= \sum_{k=0}^{N-1} 2^{-\gamma k} \le 
	\begin{cases}
	\displaystyle \frac{(\lambda/2)^\gamma -1}{2^{-\gamma}-1}
	\le \frac{\lambda^\gamma}{1-2^{\gamma}}
	& \text{if } \gamma < 0 \\
	\rule{0pt}{2.2em}
	\displaystyle 	\frac{\log \frac{2}{\lambda}}{\log 2}
	& \text{if } \gamma = 0 \,.
	\end{cases}
\end{equation}

\medskip
\noindent
\emph{Conclusion}.
At last, we can gather \eqref{neq:ffN2}, \eqref{eq:st2epp} 
and \eqref{eq:quaco1}-\eqref{eq:quaco2}.
For $\gamma<0$, since $\|\check{\varphi}\|_{L^1} \le 2 \, \|\hat{\varphi}\|_{L^1}$ 
by \eqref{eq:diffconv}, we obtain
\begin{equation*}
	|(f^K-F_x)(\psi_x^\lambda)| \le 
	\frac{4^{d+\gamma-\alpha +1}}{1-2^{-(\alpha+r)\wedge(-\gamma)}} \,
	\|\hat{\varphi}\|_{L^1} \, \hat C_1 \,\lambda^{\gamma} \,.
\end{equation*}
By \eqref{eq:phiC}, if we bound $e \le 4$ for simplicity, we obtain for all $\lambda \in (0,1]$
\begin{equation} \label{eq:unfr}
	|(f^K-F_x)(\psi_x^\lambda)| \le 
	\underbrace{\frac{r^2 \, 2^{-(r+1)\alpha }
	\, 4^{d+\gamma-\alpha+6}}
	{1-2^{-(\alpha+r)\wedge(-\gamma)}} \,
	\frac{\|\varphi\|_{L^1} \, (1+R_\varphi)^{-\alpha}}{|\int \varphi|^2}}_{\mathfrak C} \; 
	\vertiii{F}^\co_{\bar K_{3/2}, \varphi, \alpha, \gamma} \; \lambda^{\gamma} \,.
\end{equation}
For $\gamma=0$, since $\log(2/\lambda) / \log 2 \le 2(1+|\log \lambda|)$, we obtain
by \eqref{eq:quaco1}
\begin{equation} \label{eq:unfr2}
	|(f^K-F_x)(\psi_x^\lambda)| \le 
	\underbrace{\frac{r^2 \, 2^{-(r+1)\alpha}
	\, 4^{d-\alpha+6}}{1-2^{-\alpha-r}} \, 
	\frac{\|\varphi\|_{L^1} \, (1+R_\varphi)^{-\alpha}}{|\int \varphi|^2}}_{\mathfrak C} \; 
	\vertiii{F}^\co_{\bar K_{3/2}, \varphi, \alpha, \gamma}\,
	\left( 1+ |\log\lambda|\right) \, .
\end{equation}
This completes the proof of \eqref{eq:reco+1proof2<}.

\subsubsection*{Step II. Localization} 

In Step I we constructed for every
compact set $K\subset\R^d$ a distribution $f^K\in\cD'(\bar K_1)$ which satisfies 
\eqref{eq:reco+1proof2<}. 
We now exploit this construction only when $K$ is a ball. Indeed,
we use a partition of unity subordinated to a cover made by balls, to
construct a global distribution $f\in\cD'$ which satisfies \eqref{eq:reco+1}.

Fix $\eta\in\cD(B(0,\frac{1}{4}))$ such that $\eta\geq 0$ 
on $B(0,\frac{1}{4})$ and $\eta\geq 1$ on $B(0,\frac{1}{8})$ and set
\[
	E := \frac{1}{4\sqrt{d}} \, \Z^d\, , \qquad 
	\xi(x):=\frac{\eta(x)}{\sum_{z\in E} \eta(x-z)}, \qquad x\in\R^d.
\]
For every $y\in E$ we set  $\xi_y(x):=\xi(x-y)$, $x\in\R^d$. Then 
$\xi_y$ is supported in $B(y,\frac{1}{4})$
and note that $\sum_{y\in E} \xi_y(x) = 1$, for all $x\in\R^d$,
that is \emph{$(\xi_y)_{y\in E}$ is a
partition of unity subordinated to the cover $(B(y,\frac{1}{4}))_{y\in E}$}.
We finally define a distribution $f \in \cD'(\R^d)$ by
\begin{equation*}
	f \,:=\, \sum_{y\in E} f^{B_y} \cdot \xi_y \,, \qquad 
	\text{where} \quad B_y:=B(y,\tfrac{1}{4})\, ,
\end{equation*}
or more explicitly
\begin{equation*}
	f(\psi) = \sum_{y\in E} f^{B_y} (\xi_y \, \psi)\, , \qquad \forall \psi \in \cD(\R^d) \,.
\end{equation*}
We fix an arbitrary compact set $K\subset\R^d$ and
we redefine $\alpha,\beta$ and $r$ as follows:
\begin{equation}\label{eq:newr}
	\alpha := \alpha_{\bar K_2} \,, \qquad
	\beta := \beta_{\bar K_2} \,, \qquad
	r >\max\{-\alpha_{\bar K_2},-\beta_{\bar K_2}\} \,,
\end{equation}
i.e.\ we replace $\bar K_{3/2}$ by $\bar K_2$, as in the statement of \Cref{th:reco+}.
It remains to show that $f$ satisfies \eqref{eq:reco+1} on $K$ with 
these values of $r$ and~$\alpha$. 

We select the finite family of points $y_1,\ldots,y_n\in E$ for which
the balls $B_{y_i}$ have non-empty intersection with $K$. Since
each ball $B_{y_i}$ has diameter $\frac{1}{2}$, we have
\begin{equation*}
	K\subseteq\bigcup_{i=1,\ldots,n}B_{y_i} \subseteq \bar K_{1/2} \,.
\end{equation*}
Note that 
the $3/2$-enlargement
of each $B_{y_i}$ is contained in $\bar K_2$, the $2$-enlargement
of $K$. Then, by Step~I and
by the monotonicity properties \eqref{eq:monotone-alpha}-\eqref{eq:monotone-beta} 
of $K\mapsto\alpha_K$ and $K\mapsto\beta_K$, \emph{each distribution $f^{B_{y_i}}$
satisfies \eqref{eq:reco+1proof2<} for $K = B_{y_i}$
and for $r$ and $\alpha$ chosen as in \eqref{eq:newr}}.
For any test function
$\psi$ supported in $B(0,1)$ we can write
\begin{equation*}
	\xi_y(z) \, \psi^\lambda_x(z) = \zeta_x^\lambda(z) \qquad \text{where} \qquad
	\zeta(z)
	= \zeta^{[x,y,\lambda]}(z) := \xi_y(x+\lambda z) \, \psi(z) \,.
\end{equation*}
If we apply \eqref{eq:reco+1proof2<} for $K=B_y$ to
$\zeta / \|\zeta\|_{C^r} \in \cB_r$, we obtain for $\gamma < 0$
\begin{equation*}
	| ( f^{B_y} - F_x)(\xi_y \, \psi^\lambda_x) |
	= | ( f^{B_y} - F_x)(\zeta^\lambda_x) |
	\le \mathfrak{C} \, \|\zeta\|_{C^r} \, 
	\vertiii{F}^\co_{\bar K_2, \varphi, \alpha, \gamma} \, \lambda^\gamma \,,
\end{equation*}
where ${\mathfrak C}$ is as in \eqref{eq:unfr}. Note that, by Leibniz's rule,
\begin{equation*}
	\|\zeta\|_{C^r} \le 2^r \, \|\xi\|_{C^r} \, \|\psi\|_{C^r} \,.
\end{equation*}
Then, by definition of $f$ and by $\sum_{y\in E} \xi_y \equiv 1$,
\begin{equation*}
\begin{split}
	|(f-F_x)(\psi^\lambda_x)| &= \bigg| \sum_{y\in E} (f^{B_y} - F_x)
	(\xi_y \, \psi^\lambda_x) \bigg| \le 
	\sum_{y\in E} |(f^{B_y} - F_x)
	(\xi_y \, \psi^\lambda_x)| \\
	& \le (11\sqrt{d})^d \, \mathfrak{C} \, 
	2^r \, \|\xi\|_{C^r} \, \|\psi\|_{C^r} \,
	\vertiii{F}^\co_{\bar K_2, \varphi, \alpha, \gamma} \, \lambda^\gamma \,.
\end{split}
\end{equation*}
The last inequality holds because $\xi_y \, \psi^\lambda_x \equiv 0$
unless $|x-y| \le \frac{1}{4} + \lambda \le \frac{5}{4}$ 
and this can be satisfied by at most 
$(2 \cdot \frac{5}{4} \cdot 4\sqrt{d} + 1)^d \le (11\sqrt{d})^d$ many $y \in E$.
Therefore $f\in\cD'$  satisfies \eqref{eq:reco+1} for $\gamma < 0$, with
$\mathfrak{c} = \mathfrak{c}_{\alpha, \gamma, r, d, \varphi}$ given as follows
(recall \eqref{eq:unfr}):
\begin{gather}\label{eq:frakc2}
	\mathfrak{c}_{\alpha, \gamma, r, d, \varphi} = 2^r\,\|\xi\|_{C^r} \, (11\sqrt{d})^d \,
	\frac{r^2 \, 2^{-(r+1)\alpha }
	\, 4^{d+\gamma-\alpha+6}}
	{1-2^{-(\alpha+r)\wedge(-\gamma)}} \,
	\frac{\|\varphi\|_{L^1} \, (1+R_\varphi)^{-\alpha}}{|\int \varphi|^2} 
	\qquad \text{(for $\gamma < 0$)}
	\end{gather}
	where $R_\varphi$ is the radius of a ball $B(0,R_\varphi)$ which contains the support of $\varphi$.
With similar arguments, using \eqref{eq:unfr2},
for $\gamma = 0$ we obtain that $f\in\cD'$ satisfies \eqref{eq:reco+1}, with
\begin{gather}\label{eq:frakc3}
	\mathfrak{c}_{\alpha, \gamma, r, d, \varphi} = 2^r\,\|\xi\|_{C^r} \, (11\sqrt{d})^d \,
	\frac{r^2 \, 2^{-(r+1)\alpha}
	\, 4^{d-\alpha+6}}{1-2^{-\alpha-r}} \, 
	\frac{\|\varphi\|_{L^1} \, (1+R_\varphi)^{-\alpha}}{|\int \varphi|^2} 
	\qquad \text{(for $\gamma = 0$)}
\end{gather}
The proof is complete.
\qed

\bigskip

In the next sections we 
introduce the spaces of distributions with negative H\"older regularity
and we discuss some consequences of the Reconstruction Theorem.

\section{Negative H\"older spaces}
\label{sec:negHolder}

We generalize the classical H\"older spaces $\cC^\alpha$,
by allowing the index $\alpha$ to be negative. 
We recall that the family $\cB_r$ of test functions
was defined in \eqref{eq:cBr}.

\begin{definition}[Negative H\"older spaces]\label{def:Holderneg}
Given $\alpha \in (-\infty,0]$, we define $\cC^\alpha = \cC^\alpha(\R^d)$
as the space of distributions $T \in \cD'$ such that
\begin{equation} \label{eq:Holderneg}
\begin{gathered}
	|T(\psi_x^\epsilon)| \lesssim \epsilon^\alpha \\
	\text{uniformly for $x$ in compact sets, for $\epsilon \in (0,1]$
	and for $\psi \in \cB_{r_\alpha}$} \,,
\end{gathered}
\end{equation}
where we define $r_\alpha$
as the smallest integer $r\in\N$ such that $r > -\alpha$.
\end{definition}

\begin{remark}
Other definitions of the space $\cC^0$ are possible, see e.g.\ \cite{HL17}.
The one that we give here is convenient for our goals.
\end{remark}

For any distribution $T \in \cD'$ and $\alpha \le 0$, we define $\|T\|_{\cC^\alpha(K)}$ 
as the best constant in \eqref{eq:Holderneg}:
\begin{equation} \label{eq:normCalphaneg}
	\|T\|_{\cC^\alpha(K)} := \sup_{x\in K, \ \lambda \in (0,1], \ \psi \in \cB_{r_\alpha}}
	\frac{|T(\psi_x^\lambda)|}{\lambda^\alpha} \,.
\end{equation}
Then $T \in \cC^\alpha$ if and only if
$\|T\|_{\cC^\alpha(K)} < \infty$, for all compact sets
$K \subseteq \R^d$.

\begin{remark}
The quantity $\|\cdot\|_{\cC^\alpha(K)}$ is a semi-norm on $\cC^\alpha$.
It is actually a true norm for distributions $T$ which are \emph{supported in $K$},
i.e.\ such that $T(\xi) = 0$ for all test functions $\xi\in\cD$ which are supported in $K^c$.
\end{remark}

Remarkably, in order for a distribution $T \in \cD'$ to belong to $\cC^\alpha$,
it is enough that \eqref{eq:Holderneg} holds \emph{for a single, arbitrary 
test function $\psi = \varphi$ with $\int \varphi \ne 0$},
rather than uniformly for $\psi \in \cB_{r_\alpha}$.
This is ensured by our next results, that we prove below
using the same ideas as in the proof of the Reconstruction Theorem.

\begin{theorem}[Characterization of negative H\"older spaces]\label{th:charHolder}
Given a distribution $T \in \cD'$
and   $\alpha \in (-\infty, 0]$, the following conditions are equivalent.
\begin{enumerate}
\item\label{it:H1} $T$ is in $\cC^\alpha$

\item\label{it:H2} There is an integer $r > -\alpha$ such that
\eqref{eq:Holderneg} holds with $\cB_{r_\alpha}$ replaced by $\cB_r$.

\item\label{it:H3} There is a test function $\varphi \in \cD$ with $\int \varphi \ne 0$ such that
\begin{equation*}
\begin{gathered}
	|T(\varphi_x^\epsilon)| \lesssim \epsilon^\alpha \\
	\text{uniformly for $x$ in compact sets and for $\epsilon \in
	\{2^{-k}\}_{k\in\N} \subseteq (0,1]$} \,.
\end{gathered}
\end{equation*}
\end{enumerate}

Moreover, the semi-norm $\|T\|_{\cC^\alpha(K)}$ defined in \eqref{eq:normCalphaneg}
can be estimated explicitly
using an arbitrary test function $\varphi \in \cD$ with $\int \varphi \ne 0$:
\begin{equation} \label{eq:boundnormneg}
	\|T\|_{\cC^\alpha(K)} \le 
	\mathfrak{b}_{\varphi,\alpha,r_\alpha,d} \, \sup_{x\in \bar{K}_2, \ \epsilon \in (0,1]}
	\frac{|T(\varphi_x^\epsilon)|}{\epsilon^\alpha}
\end{equation}
where $\mathfrak{b}_{\varphi,\alpha,r,d}$ is an explicit constant, defined in \eqref{eq:constb} below.
\end{theorem}

We deduce a simple \emph{countable} criterion for a distribution $T\in\cD'$ to belong to $\cC^\alpha$.

\begin{theorem}[Countable criterion for negative H\"older spaces]
Let $\alpha \le 0$ and $T\in\cD'$. Then $T\in\cC^\alpha$ if (and only if) there 
is a test function $\varphi \in \cD$ with $\int \varphi \ne 0$ such that,
for every fixed $n\in\N$, we have
\begin{equation} \label{eq:Holdernegpsi3}
\begin{gathered}
	|T(\varphi_x^{\epsilon})| \lesssim \epsilon^{\alpha} \\
	\text{uniformly for $x\in{\mathbb Q}^d\cap B(0,n)$ and $\epsilon
	\in \{2^{-k}\}_{k\in\N}$} \,.
\end{gathered}
\end{equation}
\end{theorem}

\begin{proof}
The map $x\mapsto \varphi_x^\epsilon\in\cD$ is continuous,
hence $x \mapsto T(\varphi_x^\epsilon)$ is a continuous function. 
It follows that \eqref{eq:Holdernegpsi3} holds for all $x \in B(0,n)$,
so \Cref{th:charHolder} applies.
\end{proof}

We finally turn to the proof of \Cref{th:charHolder},
that we obtain as a corollary of the following more general result,
{proved at the end of this section.}

\begin{proposition}\label{0th:appetizer}
Let $T \in \cD'(\R^d)$ be a distribution with the following property:
there are a subset $K \subseteq \R^d$
and a test function $\varphi \in \cD$ with $\int \varphi \ne 0$ such that
\begin{equation}\label{eq:hypT}
	\forall x\in \bar{K}_2, \ \ \forall \, \epsilon \in \{2^{-k}\}_{k\in\N}: \qquad
	|T(\varphi_x^{\epsilon})| \le \epsilon^{\alpha} \, f(\epsilon, x) \,,
\end{equation}
for some exponent $\alpha \leq 0$ and some
arbitrary function $f: (0,1] \times \bar{K}_2 \to [0,\infty)$.

Then we can upgrade relation \eqref{eq:hypT} as follows:
for any integer $r > -\alpha$,
\begin{equation}\label{0eq:conclT}
	\forall x\in K, \ \ \forall \lambda \in (0,1], \ \ \forall \psi \in \cB_r: \qquad
	|T(\psi_x^\lambda)| \le \mathfrak{b}_{\varphi,\alpha,r,d} \, \lambda^\alpha \, 
	\bar{f}(\lambda,x) \,, 
\end{equation}
where $\mathfrak{b}_{\varphi,\alpha,r,d}$ is the constant in \eqref{eq:constb} below,
and $\bar f: (0,1] \times K \to [0,\infty)$ equals
\begin{equation}\label{eq:fbar}
	\bar{f}(\lambda, x) :=
	\sup_{\lambda' \in (0, \lambda], \ 
	x' \in B(x,2\lambda)} f(\lambda', x') \,.
\end{equation}
\end{proposition}

\begin{proof}[Proof of \Cref{th:charHolder}]
Clearly \ref{it:H1}. implies \ref{it:H2}., because $\cB_r \subseteq \cB_{r_\alpha}$
for $r \ge r_\alpha$, and \ref{it:H2}. implies \ref{it:H3}., because
we can choose any $\varphi = \psi \in \cB_r$ with $\int \psi \ne 0$.

To prove that \ref{it:H3}. implies \ref{it:H1}., it suffices to apply \Cref{0th:appetizer} 
on every compact set with a constant function $f(\lambda,x)\equiv C$.
Equation \eqref{eq:boundnormneg} then follows by \eqref{0eq:conclT}.
\end{proof}

We next show that the reconstruction $f = \cR F$ of a coherent germ $F$
provided by the Reconstruction Theorem belongs to a negative H\"older space
and it is a continuous function of the germ, in a suitable sense.

We recall that the \emph{coherence} of a germ is quantified by the semi-norm
$\vertiii{F}^\co_{\bar K_1, \varphi, \alpha, \gamma}$ defined in \eqref{eq:triple1}.
We introduce a second semi-norm which quantifies the \emph{homogeneity} of a
coherent germ:
for any compact set $K\subset\R^d$ we define, recalling \Cref{th:boundor},
\begin{equation}\label{eq:normFhom}
	\vertiii{F}^\ho_{K,\varphi, \beta}
	:= \sup_{x\in K,\ \epsilon\in(0,1]} \frac{|F_x(\varphi^\epsilon_x)|}{\epsilon^\beta} \,,
\end{equation}
where $\varphi$ is as in \Cref{def:coherent-germ}.
We can now state the following result.

\begin{theorem}[Reconstruction Theorem and H\"older spaces]\label{pr:image}
Let $(F_x)_{x\in\R^d}$ be a $(\alpha,\gamma)$-coherent germ with local homogeneity bound
$\beta<\gamma$. 
If $\beta > 0$, then $\cR F = 0$. 
If $\beta \le 0$,
then $\cR F$ belongs to $\cC^{\beta}$ and for every compact set $K\subseteq \R^d$
\begin{equation}\label{eq:cont2}
	\|\cR F\|_{\cC^\beta(K)}
	\le \mathfrak{C} \,
	\Big( \vertiii{F}^\co_{\bar K_4, \varphi, \alpha, \gamma}
	+ \vertiii{F}^\ho_{\bar K_2,\varphi, \beta}\Big) \,,
\end{equation}
where $\varphi$ be the test function in the coherence condition \eqref{eq:coherent}
and $\mathfrak{C} = \mathfrak{C}_{\alpha, \gamma, \beta, d, \varphi} < \infty$
is a constant
which depends neither on $F$ nor on $K$.
\end{theorem}

\begin{remark}
The bound \eqref{eq:cont2} holds \emph{for any test function $\varphi \in \cD$
with $\int\varphi \ne 0$}, as for
the coherence condition \eqref{eq:coherent}. This will be shown in \Cref{th:coh}.
\end{remark}

\begin{proof}
When $\beta > 0$ we already observed in  \Cref{rem:poshom} that $\cR F = 0$.
Henceforth we fix $\beta \le 0$. 
Let $\varphi$ be the test function in the coherence condition \eqref{eq:coherent}.
Let $f=\cR F$ by a reconstruction of $F$. Fix a compact set $K$:
if we show that
\begin{equation} \label{eq:uiclaim}
	\sup_{x \in \bar{K}_2, \ \lambda \in (0,1]}\frac{|f(\varphi^\lambda_x)|}{\lambda^\beta}
	\le \mathfrak{C}' \, \Big( \vertiii{F}^\co_{\bar K_4, \varphi, \alpha, \gamma}
	+ \vertiii{F}^\ho_{\bar K_2,\varphi, \beta}\Big) 
\end{equation}
for some $\mathfrak{C}' = \mathfrak{C}'_{\alpha, \gamma, \beta, d, \varphi} < \infty$,
then we obtain \eqref{eq:cont2} by \eqref{eq:boundnormneg} with
$\mathfrak{C} = \mathfrak{b}_{\varphi,\beta,r_\beta,d} \, \mathfrak{C}'$.

It remains to prove \eqref{eq:uiclaim}.
Let us set $\bar r :=
\min\{r \in \N: \ r > \max \{-\alpha, -\beta\}\}$.
We observed in \Cref{rem:RTsingle} that
$\xi := c \, \varphi^\eta \in \cB_{\bar{r}}$ for suitable $c, \eta > 0$
(which depend on $\varphi$ and $\bar r$). Then by \eqref{eq:reco+1} for $r=\bar{r}$ we have,
uniformly for $x \in \bar{K}_2$ and $\lambda\in(0,1]$,
\begin{equation*}
	|(f-F_x)(\varphi^\lambda_x)|
	= c^{-1} \, |(f-F_x)(\psi^{\eta^{-1}\lambda}_x)|  \le \mathfrak{c}' \,
	\vertiii{F}^\co_{\bar K_4, \varphi, \alpha, \gamma} \cdot  \begin{cases}
	\lambda^{\gamma} & \text{if } \gamma \ne 0 \\
	(1 + |\log\lambda|) & \text{if } \gamma = 0 
	\end{cases} 
\end{equation*}
for a suitable $\mathfrak{c}' = \mathfrak{c}'_{\alpha,\gamma,\beta,d,\varphi}$.
Since $\beta < \gamma$, we bound
$\lambda^\gamma \le \lambda^\beta$ for $\gamma \ne 0$ 
and $1+|\log \lambda| \le c_\beta \, \lambda^\beta$ for $\gamma = 0$,
 for all $\lambda \in (0,1]$
(by direct computation $c_\beta = -\beta^{-1} \, e^{-1-\beta}$).
Recalling \eqref{eq:normFhom}, by the triangle inequality we obtain
\begin{equation*}
\begin{split}
	\sup_{x \in \bar{K}_2, \ \lambda \in (0,1]}\frac{|f(\varphi^\lambda_x)|}{\lambda^\beta}
	& \le \sup_{x \in \bar{K}_2, \ \lambda \in (0,1]}
	\frac{|(f-F_x)(\varphi^\lambda_x)| + |F_x(\varphi^\lambda_x)|}{\lambda^\beta} \\
	& \le (1+c_\beta) \, \mathfrak{c}' \, \vertiii{F}^\co_{\bar K_4, \varphi, \alpha, \gamma} 
	\, + \, \vertiii{F}^\ho_{\bar K_2, \varphi, \beta} \,,
\end{split}
\end{equation*}
which completes the proof of \eqref{eq:uiclaim}.
\end{proof}

\begin{remark}[Non uniqueness]\label{rem:well-def}
Let $(F_x)_{x\in\R^d}$ be a $(\alpha,\gamma)$-coherent germ with $\gamma < 0$ 
and let $f_1$ and $f_2$ be two distributions 
which both satisfy \eqref{eq:reco+1}. Then
\[
|(f_1-f_2)(\psi^\lambda_x)| \leq |(f_1-F_x)(\psi^\lambda_x)|+|(f_2-F_x)(\psi^\lambda_x)|
\lesssim \lambda^\gamma
\]
uniformly for $x$ in compact sets and $\lambda\in(0,1]$ and therefore $f_1-f_2\in\cC^{\gamma}$,
by \Cref{th:charHolder}. Viceversa, if 
$f\in\cD'$ satisfies \eqref{eq:reco+1} and $D\in\cC^{\gamma}$, 
then $f+D$ also satisfies \eqref{eq:reco+1}.
Therefore, \emph{the reconstruction $f = \cR F$ of a $(\alpha,\gamma)$-coherent germ $F$
with $\gamma < 0$ is not unique, but it
is well-defined up to an element of $\cC^{\gamma}$}.
\end{remark}

{We conclude this section with the proof of \Cref{0th:appetizer}.}

\begin{proof}[Proof of \Cref{0th:appetizer}]
Fix $\varphi \in \cD$ with $\int \varphi \ne 0$
which satisfies \eqref{eq:hypT} and
$r \in \N$ with $r > -\alpha$. We define the test function $\hat\varphi = \hat\varphi^{[r]}$ by
\eqref{eq:hatvarphi} and we claim that
\begin{equation}\label{0eq:hypT}
	\forall x\in \bar{K}_2, \ \ \forall \, \epsilon \in \{2^{-k}\}_{k\in\N}: \qquad
	|T(\hat\varphi_x^{\epsilon})| \le C \, \epsilon^{\alpha} \, \tilde f(\epsilon, x) \,,
\end{equation}
where
\begin{equation}\label{eq:Cproof}
	\tilde f(\epsilon, x) := \sup_{\epsilon' \in (0,\epsilon]} f(\epsilon', x) \,, \qquad
	C := \tfrac{e^2 \, r}{|\int \varphi|} \, \big(\tfrac{2^{-r-1}}{1+R_\varphi}\big)^\alpha \,.
\end{equation}
To prove this claim, it suffices to write
$T(\hat\varphi_x^\epsilon) = \frac{1}{\int \varphi} \, \sum_{i=0}^{r-1} c_i \, 
T(\varphi_x^{\epsilon \lambda_i})$
and to apply \eqref{eq:hypT} to $T(\varphi_x^{\epsilon \lambda_i})$,
noting that $\frac{2^{-r-1}}{1+R_\varphi} < \lambda_i \le 1$ by \eqref{eq:hatvarphi} 
and $|c_i| \le e^2$ by \eqref{eq:boundonci}.

We recall that $\hat\varphi$ satisfies \eqref{eq:supphatphi}-\eqref{eq:monomials}
as well as \eqref{eq:boundhatvarphi}.
Next we define
\begin{equation*}
	\rho := \hat\varphi^2 * \hat\varphi \,, \qquad
	\epsilon_k = 2^{-k}\, ,
\end{equation*}
as in \eqref{eq:rhoepsilon} above.
Then, see \eqref{eq:checkvarphi2},
\begin{equation} \label{0eq:diffrho}
	\rho^{\epsilon_{k+1}} - \rho^{\epsilon_k} =
	\hat\varphi^{\epsilon_k} * \check\varphi^{\epsilon_k} \qquad
	\text{where} \qquad
	\check\varphi := \hat\varphi^{\frac{1}{2}} - \hat\varphi^{2} \,.
\end{equation}
Note that $(\rho^{\epsilon_n})_{n\in\N}$ are mollifiers, because $\int \rho = \int \hat\varphi \cdot
\int \hat\varphi^2 = 1$ (recall that $\int \hat\varphi = 1$), therefore
for any test function $\psi$ we have
\begin{equation} \label{eq:recalli}
	T(\psi_x^\lambda) = \lim_{n\to\infty}
	T(\rho^{\epsilon_n} * \psi_x^\lambda)
\end{equation}
hence for every $N\in\N$ we can write
\begin{equation} \label{0eq:Tdeco}
	T(\psi_x^\lambda) 
	\ = \ \underbrace{T(\rho^{\epsilon_N} * \psi_x^\lambda)}_A
	+
	\underbrace{\big( T(\psi_x^\lambda) -T(\rho^{\epsilon_N} * \psi_x^\lambda)\big)}_B\,.
\end{equation}
Henceforth we fix $\psi \in \cB_r$ and we set
$N := \min\{k\in\N: \ \epsilon_k \le {\lambda}\}$ so that $N \ge 1$ and
\begin{equation} \label{0eq:epsilonNlambda}
	\tfrac{1}{2} \lambda < \epsilon_N \le  \lambda \,.
\end{equation}
We estimate separately the two terms $A$ and $B$ in \eqref{0eq:Tdeco}.

\bigskip
\noindent
\emph{Estimate of $A$.} We can write
\begin{equation*}
\begin{split}
	A &= T(\rho^{\epsilon_N} * \psi_x^\lambda) 
	= \int_{\R^d} T(\rho_z^{\epsilon_N}) \, \psi_x^\lambda(z) \d z\\
	&= \int_{\R^d} \int_{\R^d} T(\hat\varphi_y^{\epsilon_N})
	\, \hat\varphi^{2\epsilon_N}(y-z) \, \psi_x^\lambda(z) \d y \d z \\
	&= \int_{\R^d} T(\hat\varphi_y^{\epsilon_N})
	\, (\hat\varphi^{2\epsilon_N} * \psi_x^\lambda)(y) \d y  \,.
\end{split}
\end{equation*}
We now apply \Cref{th:keyintest2} with
$G(y):=T(\hat\varphi_y^{\epsilon})$: by \eqref{eq:Ge1} we obtain
\[
	|A| \le 2^{d} \, \|\hat\varphi \|_{L^1} \, \sup_{y \in B(x,\lambda+\epsilon_N)}
	|T(\hat\varphi_y^{\epsilon_N})|\, .
\]
By \eqref{0eq:epsilonNlambda} we have $\lambda + \epsilon_N \le 2\lambda$
and $\epsilon_N\geq\lambda/2$. Since $\alpha\leq 0$, we obtain by \eqref{0eq:hypT} 
\[
	\sup_{y \in B(x,\lambda+\epsilon_N)} |T(\hat\varphi_y^{\epsilon_N})|
	\leq C \, \epsilon_N^{\alpha} \, \sup_{y\in B(x,2\lambda)} \tilde f(\epsilon_N,y)
	\leq C \, (\lambda/2)^\alpha\sup_{\lambda'\in(0,\lambda], \ x'\in B(x,2\lambda)} f(\lambda',x')\, ,
\]
and finally, recalling \eqref{eq:fbar},
\begin{equation} \label{0eq:aest}
	|A| \le \big\{ 2^{d-\alpha}
	\, C \, \|\hat\varphi \|_{L^1} \big\} \, \lambda^\alpha \,  \bar f(\lambda,x)\, .
\end{equation}

\bigskip
\noindent
\emph{Estimate of $B$.} Let us fix $k \in \N$ with $k \ge N$. We can write, by \eqref{0eq:diffrho},
\begin{equation*}
\begin{split}
	b_k & :=  T(\rho^{\epsilon_{k+1}} * \psi_x^\lambda)
	- T(\rho^{\epsilon_k} * \psi_x^\lambda)
	= \int_{\R^d} T(\rho_z^{\epsilon_{k+1}} - \rho_z^{\epsilon_k}) \, \psi_x^\lambda(z) \d z\\
	&= \int_{\R^d} \int_{\R^d} T(\hat\varphi_y^{\epsilon_k})
	\, \check\varphi^{\epsilon_k}(y-z) \, \psi_x^\lambda(z) \d y \d z \\
	&= \int_{\R^d} T(\hat\varphi_y^{\epsilon_k})
	\, (\check\varphi^{\epsilon_k} * \psi_x^\lambda)(y) \d y  \,.
\end{split}
\end{equation*}
Note that $\check\varphi$ is supported in $B(0,1)$ 
(because $\hat\varphi$ is supported in $B(0,\frac{1}{2})$, recall \eqref{eq:supphatphi})
and
$\epsilon_k \le \epsilon_N \le \lambda$ for $k \ge N$.
Then $\check\varphi^{\epsilon_k} * \psi_x^\lambda$ is supported in $B(w,\lambda+\epsilon_k)
\subseteq B(w, 2\lambda) $. 
We apply again \Cref{th:keyintest2}
with $G(y):=T(\hat\varphi_y^{\epsilon_k})$:
by \eqref{0eq:hypT} and \eqref{eq:fbar} we can bound 
$\sup_{y \in B(w, \lambda + \epsilon_k)} |G(y)| 
\le C \, \epsilon_k^\alpha \, \bar f(\epsilon_k, w)
\le C \, \epsilon_k^\alpha \, \bar f(\lambda, w)$ which yields, by \eqref{eq:Ge2},
\begin{equation*}
	|b_k| \le C \, 4^d \, 
	\|\check\varphi\|_{L^1} \, \lambda^{-r} \,  \epsilon_k^{\alpha+r} \, \bar{f}(\lambda,w) \,.
\end{equation*}
Since $\alpha+r>0$ by assumption,
we obtain $\sum_{k\geq N}|b_k|<+\infty$ and,
recalling \eqref{eq:recalli}, we can write
$B = T(\psi_x^\lambda) -T(\rho^{\epsilon_N} * \psi_x^\lambda)$
as the converging sequence $B=\sum_{k=N}^\infty b_k$.
Since $\sum_{k=N}^\infty  \epsilon_k^{\alpha+r} = (1-2^{-\alpha-r})^{-1} \, \epsilon_N^{\alpha+r}$,
this yields
\begin{equation}\label{0eq:best}
	|B| \le \sum_{k=N}^\infty \
	|b_k| \le \frac{C \, 4^d \, 
	\|\check\varphi\|_{L^1}}{1-2^{-\alpha-r}}
	\, \lambda^{-r} \,  \epsilon_N^{\alpha+r} \, \bar{f}(\lambda,w) \,.
\end{equation}

\bigskip
\noindent
\emph{Conclusion.} 
By \eqref{0eq:Tdeco}, \eqref{0eq:aest} and \eqref{0eq:best}, 
since $\|\check\varphi\|_{L^1}\leq 2\|\hat\varphi \|_{L^1}$
and $\epsilon_N \le \lambda$, we get
\begin{equation*}
	|T(\psi_x^\lambda) | \le 
	\frac{4^{d-\alpha+1} }{1-2^{-\alpha-r}}
	\, \|\hat\varphi \|_{L^1} \,
	C \, \lambda^\alpha \, \bar{f}(\lambda,w) \,.
\end{equation*}
If we plug the bound \eqref{eq:boundhatvarphi}
and the definition \eqref{eq:Cproof} of $C$, we get
\begin{equation*}
	|T(\psi_x^\lambda) | \le 
	\bigg\{ \frac{4^{d-\alpha+1}}{1-2^{-\alpha-r}} \,
	\bigg(\frac{e^2 \, r}{|\int \varphi|}\bigg)^2 \, \bigg(\frac{2^{-r-1}}{1+R_\varphi}\bigg)^\alpha
	\, \|\varphi \|_{L^1}
	\bigg\} \,\lambda^\alpha \, \bar{f}(\lambda,w) \,.
\end{equation*}
Therefore we have proved \eqref{0eq:conclT}, with the explicit constant
\begin{equation}\label{eq:constb}
	\mathfrak{b}_{\varphi,\alpha,r,d} := 
	\frac{4^{d-\alpha+1} \, e^4 \, 2^{-\alpha(r+1)} \, r^2}{1-2^{-\alpha-r}} \,
	\frac{(1+R_\varphi)^{-\alpha} \, \|\varphi \|_{L^1}}{|\int \varphi|} 
\end{equation}
The proof is complete.
\end{proof}

\section{More on coherent germs}
\label{sec:coherent+}

As an application of \Cref{0th:appetizer}, we show that
the coherence condition \eqref{eq:coherent} can be strengthened,
replacing the test function $\varphi$ by an arbitrary test function,
provided we slightly adjust the exponent $\alpha_K$.

\begin{proposition}[Enhanced coherence]
\label{th:coh}
Let $F = (F_x)_{x\in\R^d}$ be a $\gamma$-coherent germ,
i.e.\ \eqref{eq:coherent} holds for some $\varphi \in \cD$
and some family $\balpha = (\alpha_K)$.
If we define
\begin{equation} \label{eq:balpha'}
	\balpha' = (\alpha'_K) \qquad \text{where} \qquad
	\alpha'_K := \alpha_{\bar K_2} \,,
\end{equation}
then we can replace $\varphi$ in \eqref{eq:coherent} by an arbitrary test function,
provided we replace $\alpha_K$ by $\alpha'_K$. More precisely,
for any compact set $K\subseteq\R^d$ and any $r > -\alpha'_K$ we have
\begin{equation}\label{eq:coherent-unif}
\begin{gathered}
	| (F_z - F_y)(\psi^\epsilon_y) | \lesssim
	\epsilon^{\alpha'_K} \, (|z-y| + \epsilon)^{\gamma - {\alpha'_K}} \\
	\text{uniformly for {$z,y \in K$}}, 
	\ \epsilon \in (0,1] \text{ and } \psi \in \cB_r \,.
\end{gathered}
\end{equation}
It follows that the family of $\gamma$-coherent
germs is a vector space.
\end{proposition}

\begin{proof}
Assume that \eqref{eq:coherent-unif} has been proved.
Given an arbitrary test function
$\xi \in \cD$, we can write $\xi = c \, \psi^{\lambda}$
for suitable $c \in \R$, $\lambda \in (0,1]$ and
$\psi \in \cB_r$ (exercise), hence
$\xi_y^\epsilon = c \, \psi^{\lambda \epsilon}_{y}$.
Then it  follows by \eqref{eq:coherent-unif} that we
can replace $\varphi$ by $\xi$ in \eqref{eq:coherent}.

It remains to prove \eqref{eq:coherent-unif}.
It is convenient to center the test function at a third point~$x$,
i.e.\ to replace $\psi_y^\epsilon$ by $\psi_x^\epsilon$. By the triangle 
inequality we can bound
\begin{equation} \label{eq:tria}
	|(F_z - F_y)(\varphi_x^\epsilon)| \le
	|(F_z - F_x)(\varphi_x^\epsilon)| + |(F_y - F_x)(\varphi_x^\epsilon)| \,.
\end{equation}
Let us fix a compact set $K \subseteq \R^d$.
Both terms in the right hand side of \eqref{eq:tria} can be estimated by
the coherence condition \eqref{eq:coherent} for the enlarged set $\bar{K}_2$.
Recalling \eqref{eq:balpha'}, we see that there is $c_K < \infty$ such that
\begin{equation*}
\begin{gathered}
	\forall z,y \in K, \
	\forall x\in \bar{K}_2, \ \ \forall \epsilon \in (0,1]: \\
	|(F_z - F_y)(\varphi_x^\epsilon)| \le c_K \, \epsilon^{\alpha'_K} \,
	(|z-x| + |y-x| + \epsilon)^{\gamma-{\alpha'_K}} \,.
\end{gathered}
\end{equation*}
For fixed $y,z \in K$ we can apply \Cref{0th:appetizer},
with $T = F_z - F_y$ and $f(\epsilon, x) = (|z-x| + |y-x| + \epsilon)^{\gamma-{\alpha'_K}}$.
Given any $r\in\N$ with $r > -{\alpha'_K}$,
relation \eqref{0eq:conclT} yields
\begin{align*}
	\forall z,y,&\, x \in K, \ \  \forall \lambda \in (0,1]\,,  \ \ \forall \psi \in \cB_r \,: \\
	|(F_z - F_y)(\psi_x^\lambda)| &\le 
	\mathfrak{b}_{\varphi,{\alpha'_K},r,d} \, \lambda^{\alpha'_K} \,
	(|z-x| + |y-x| + 5\lambda)^{\gamma-{\alpha'_K}} \\
	& \lesssim  \lambda^{\alpha'_K} \,
	(|z-x| + |y-x| + \lambda)^{\gamma-{\alpha'_K}} \,.
\end{align*}
If we plug $x = y$ we obtain \eqref{eq:coherent-unif}.
\end{proof}

We now show that also the local homogeneity relation \eqref{eq:bounded-order-germ}
can be strengthened, replacing $\varphi$ by an arbitrary test function,
provided we slightly adjust $\beta_K$.

\begin{proposition}[Enhanced local homogeneity]
\label{pr:enhom}
Let $F = (F_x)_{x\in\R^d}$ be a $\gamma$-coherent germ
with local homogeneity bounds $\bbeta = (\beta_K)$, see \eqref{eq:bounded-order-germ}. 
If we set
\begin{equation*}
	\bbeta' = (\beta'_K) \qquad \text{where} \qquad \beta'_K := \beta_{\bar{K}_2} \,,
\end{equation*}
then we can replace $\varphi$ in \eqref{eq:bounded-order-germ} by an
arbitrary test function,
provided we replace $\beta_K$ by $\beta'_K$. More precisely,
for any compact set $K\subseteq\R^d$ and any $r > \max\{-\alpha'_K, - \beta'_K\}$,
with $\alpha'_K$ defined in \eqref{eq:balpha'}, we have
\begin{equation} \label{eq:bounded-order-germ-unif2}
	\begin{gathered}
	|F_x(\psi^\epsilon_x)| \lesssim \, \epsilon^{\beta'_K} \\
	\text{uniformly for $x \in K$, $\epsilon \in (0,1]$ and $\psi \in \cB_{ r}$} \,.
\end{gathered}
\end{equation}
\end{proposition}

\begin{proof}
We apply the Reconstruction Theorem:
let $f = \cR F$ is a reconstruction of $F$.
Fix a compact set $K \subseteq \R^d$ and
$r > \max\{- \alpha_{\bar{K}_2}, - \beta_{\bar{K}_2}\}$.
Then, by \eqref{eq:reco+1}, 
\begin{equation*}
	|(f-F_x)(\psi_x^\epsilon)| \lesssim \begin{cases}
	\epsilon^\gamma & \text{if } \gamma \ne 0 \\
	\big(1+|\log \epsilon|\big) & \text{if } \gamma = 0
	\end{cases}
\end{equation*}
uniformly for $x \in K$, $\epsilon \in (0,1]$ and $\psi \in \cB_{r}$.
Since $f \in \cC^{\beta}$ by \Cref{pr:image}, we have
\begin{equation*}
	|f(\psi_x^\epsilon)| \lesssim \epsilon^{\beta}
\end{equation*}
uniformly for $x \in K$, $\epsilon \in (0,1]$ and
$\psi \in \cB_{r}$. Since $\beta < \gamma$, we finally get
\begin{equation*}
	|F_x(\psi_x^\epsilon)| \le |(F_x-f)(\psi_x^\epsilon)| + |f(\psi_x^\epsilon)|
	\lesssim \epsilon^{\beta},
\end{equation*}
uniformly for $x \in K$, $\epsilon \in (0,1]$ and $\psi \in \cB_{r}$.
This proves \eqref{eq:bounded-order-germ-unif2}.
\end{proof}

\section{Young product of functions and distributions}
\label{sec:Young}

As an application of the Reconstruction Theorem,
we prove that there is a canonical definition of product between
a H\"older function $f \in \cC^\alpha$, with $\alpha > 0$,
and a H\"older distribution $g \in \cC^\beta$, with $\beta \le 0$,
provided $\alpha + \beta > 0$.  This classical result has been obtained with 
wavelets analysis or Bony's paraproducts, see e.g. \cite[Theorem~1 in Section~4.4.3]{rs96}, 
\cite[Theorem~2.52]{BCD11} and \cite[Proposition~4.14]{Hairer2014d}. 
Our proof of the Reconstruction Theorem provides a new approach to this result,
which bypasses Fourier analysis and
applies to general (non tempered) distributions.
In the case $\alpha+\beta\leq 0$, 
a non-unique and non-canonical ``product'' can still be constructed.

We start with some general considerations. 
Given any distribution $g \in \cD'$ and any smooth function 
$f \in C^\infty$, their product $P = g \cdot f$ is canonically defined by
\begin{equation*}
	P(\varphi) = 	(g \cdot f)(\varphi) := g(\varphi \, f) \,, \qquad \forall \varphi \in \cD \,.
\end{equation*}
If $f \in \cC^\alpha$ with $\alpha > 0$ this no longer makes sense,
as $\varphi \, f $ might not be a test function. However we can still
give a \emph{local definition of $g \cdot f$ close to a point $x\in\R^d$},
replacing $f$ by its Taylor polynomial $F_x$ of order 
$\underline{r}(\alpha):=\max\{n\in\N_0:n<\alpha\}$ based at $x$:
\begin{equation}\label{eq:polyeffe}
	F_x(\cdot):=\sum_{0 \le |k|< \alpha} \partial^k f(x) \, \frac{(\cdot-x)^k}{k!} \,.
\end{equation}
This leads us to define the germ $P = (P_x:=g\cdot F_x)_{x\in\R^d}$, that is
\begin{equation}\label{eq:Pprod}
	P_x(\varphi) = (g \cdot F_x)(\varphi) :=g(\varphi\, F_x), \qquad \varphi\in\cD\, .
\end{equation}
We can now state the following result.

\begin{theorem}[Young product]\label{th:Young}
Fix $\alpha > 0$ and $\beta \le 0$. 
\begin{itemize}
\item If $\alpha+\beta > 0$, there exists 
a bilinear continuous map $\cM: \cC^{\alpha}\times\cC^\beta\to \cC^\beta$
which extends the usual product $\cM(f,g) = f \cdot g$ when $f \in C^\infty$.
This map is characterized by the following property: for any $r \in \N$
with $r > -\beta$
\begin{equation}\label{eq:propM}
\begin{gathered}
	\big| (\cM(f,g) - g \cdot F_x )(\psi_x^\lambda) \big|
	\lesssim 
	\begin{cases}
	\lambda^{\alpha+\beta} & \text{if } \alpha+\beta \ne 0 \\
	\big(1+|\log \lambda|\big) & \text{if } \alpha+\beta = 0
	\end{cases}  \\
	\text{uniformly for $x$ in compact sets, $\lambda \in (0,1]$ 
	and $\psi \in \cB_r$} \,,
\end{gathered}
\end{equation}
where $F_x$ is the Taylor polynomial of $f$ based at $x$, see
\eqref{eq:polyeffe}.

\item If $\alpha+\beta \le 0$, 
there exists a bilinear continuous map $\cM: \cC^{\alpha}\times\cC^\beta\to \cC^\beta$
which satisfies property \eqref{eq:propM}. This map is neither unique nor canonical.
However, 
for $\alpha+\beta < 0$ any two maps $\cM, \cM'$ which satisfy property \eqref{eq:propM} 
must differ by a map in $\cC^{\alpha+\beta}$,
i.e.\ we must have $\cM - \cM' : \cC^{\alpha}\times\cC^\beta\to  \cC^{\alpha+\beta}$.
\end{itemize}
\end{theorem}

\begin{remark}
For \emph{fixed} $\alpha > 0$ and $\beta \le 0$ with $\alpha+\beta > 0$,
we cannot claim that $\cM: \cC^{\alpha}\times\cC^\beta\to \cC^\beta$
is the unique continuous map
which extends the usual product $\cM(f,g) = f \cdot g$ when $f \in C^\infty$, simply
because $C^\infty$ is not dense in $\cC^\alpha$. On the other hand, 
given any $\beta \le 0$, we can state that 
\emph{$\cM: \bigcup_{\alpha > -\beta}\cC^{\alpha}\times\cC^\beta\to \cC^\beta$
is indeed the unique continuous map which extends the usual product}, 
because $C^\infty$
is dense in $\cC^\alpha$ with respect to the topology of $\cC^{\alpha'}$, for any $\alpha' < \alpha$.
\end{remark}

\begin{remark}
For $\alpha+\beta \le 0$ the ``product'' $\cM$
that we construct is 
\emph{non-local}, as can be inferred from the proof of the Reconstruction Theorem.
This is reminiscent of the para-products studied by Gubinelli-Imkeller-Perkowski \cite{GIP15}.
\end{remark}

Before proving \Cref{th:Young} we need some preparation.
We recall that the negative H\"older space $\cC^\beta$ with $\beta \le 0$ is equipped
with the family of semi-norms $\|\cdot\|_{\cC^\beta(K)}$ defined in \eqref{eq:normCalphaneg},
for compact sets $K \subseteq \R^d$.
We now introduce a corresponding family of semi-norms  $\|\varphi\|_{\cC^\alpha(K)} $
for \emph{positive} H\"older spaces $\cC^\alpha$ with $\alpha > 0$. Recall that
\begin{equation*}
	\underline{r}(\alpha):=\max\{n\in\N_0:n<\alpha\} \,.
\end{equation*}
Then, given a compact set $K\subseteq \R^d$, we define $\|\cdot\|_{\cC^\alpha(K)}$
by taking the maximum between between $\|f\|_{C^{\underline{r}}(K)}$ and 
the best implicit constant in \eqref{eq:holdergamma} when $x,y \in K$:
\begin{equation}\label{eq:Holderalphanorm}
	\|f\|_{\cC^\alpha(K)} 
	:= \max\bigg\{ \|f\|_{C^{\underline{r}}(K)}, \, 
	\sup_{x,y \in K} \frac{|f(y)-F_x(y)|}{|y-x|^\alpha} \bigg\} \,.
\end{equation}
We can now formulate more precisely the continuity of $\cM$ stated in \Cref{th:Young}:
we are going to prove that for every compact set $K \subseteq \R^d$
\begin{equation}\label{eq:contcM}
	\|\cM(f,g)\|_{\cC^\beta(K)} \lesssim
	\|f\|_{\cC^\alpha(\bar{K}_4)} \, \|g\|_{\cC^\beta(\bar{K}_4)} \,.
\end{equation}

To prove \Cref{th:Young}, we first quantify the coherence of the germ $P$ in \eqref{eq:Pprod}.

\begin{proposition}\label{pr:mult}
If $f \in \cC^\alpha$ and $g \in \cC^\beta$, with
$\alpha>0$ and $\beta \le 0$, then the germ 
$P = (P_x)_{x\in\R^d}$ is $(\beta,\alpha+\beta)$-coherent 
and has homogeneity
bounded below by $\beta$.
\end{proposition}

\begin{proof}
We are going to show that
there is a test function $\varphi\in\cD(B(0,1))$ with $\int\varphi \ne 0$ such that,
for every compact set $K\subset\R^d$,
the following relations hold:
\begin{align}\label{eq:coherent-P}
	| (P_z - P_y)(\varphi^\epsilon_y) | 
	& \lesssim \| f \|_{\cC^\alpha(K)} \, 
	\|g\|_{\cC^\beta(K)} \,
	\epsilon^{\beta} \, (|z-y| + \epsilon)^{\alpha} \,, \\
	\label{eq:hom-P}
	| P_x(\varphi^\epsilon_x)| 
	& \lesssim \| f \|_{\cC^\alpha(K)} \, 
	\|g\|_{\cC^\beta(K)} \, \epsilon^\beta \,,
\end{align}
uniformly for $x, y,z \in K$ and $\epsilon \in (0,1]$.
Throughout this proof, all implicit constants hidden in the notation $\lesssim$ may depend
on the parameters $\alpha, \beta$, but but not on $K, f, g$.

We first prove \eqref{eq:coherent-P}.
Let us fix a compact set $K\subset\R^d$
and we set $r = r_\beta := \min\{r \in \N: \ r > -\beta\}$.
By \eqref{eq:Holderneg} {applied to $\psi / \|\psi\|_{C^r}$} we can bound,
recalling \eqref{eq:normCalphaneg},
\begin{equation}\label{eq:Holbo}
	\left| g(\psi^\epsilon_y) \right| \leq \|g\|_{\cC^\beta(K)}\, \|\psi\|_{C^r} \,  \epsilon^\beta
	\qquad \text{for all } \epsilon \in (0,1]\, , \ \psi\in\cD(B(0,1))\, , \ y\in K\, .
\end{equation}
Fix now any $\varphi\in\cD(B(0,1))$
with $\int\varphi \ne 0$ and $\|\varphi\|_{C^r} \le 1$. By \eqref{eq:F-F}, for any $y,z \in K$
\[
\begin{split}
(P_z - P_y)(\varphi^\epsilon_y)=& -\sum_{0 \le |k|< \alpha} g\big({(\cdot-y)^k}\, \varphi^\epsilon_y\big)
\,\frac{R^k(y,z)}{k!}
\end{split}
\]
where $|R^k(y,z)| \lesssim \|f\|_{\cC^\alpha(K)} \, |z-y|^{\alpha-|k|}$. 
We have for fixed $y\in\R^d$, $k\in\N_0^d$ and $\epsilon>0$
\begin{equation*}
	(w-y)^k\, \varphi^\epsilon_y(w) =\epsilon^{|k|} \, \psi^\epsilon_y(w)\,, 
	\quad \text{where} \quad
	\psi(w):=w^k\, \varphi(w) \, .
\end{equation*}
Then $\psi\in\cD(B(0,1))$ and $\|\psi\|_{C^r} \lesssim \|\varphi\|_{C^r} \le 1$, hence
it follows by \eqref{eq:Holbo} that
\begin{equation}\label{eq:gpolyf}
	|g\big({(\cdot-y)^k}\, \varphi^\epsilon_y\big)|
	= \epsilon^{|k|} \, g\big(\psi_y^\epsilon\big) \lesssim
	\|g\|_{\cC^\beta(K)} \,  \epsilon^{\beta+|k|} \, .
\end{equation}
We thus obtain, uniformly for $z,y \in K$ and $\epsilon \in (0,1]$,
\[
\begin{split}
	| (P_z - P_y)(\varphi^\epsilon_y) | 
	& \lesssim \|f\|_{\cC^\alpha(K)} \, 
	\|g\|_{\cC^\beta(K)} \,  \sum_{0 \le |k|<\alpha} \epsilon^{\beta+|k|}\,  |z-y|^{\alpha-|k|}\\
	& \lesssim \|f\|_{\cC^\alpha(K)} \, \|g\|_{\cC^\beta(K)} \,  \epsilon^\beta(|z-y|+\epsilon)^\alpha,
\end{split}
\]
which completes the proof of \eqref{eq:coherent-P}.

We next prove \eqref{eq:hom-P}.
By \eqref{eq:polyeffe} and \eqref{eq:Pprod},
recalling \eqref{eq:Holderalphanorm} and \eqref{eq:gpolyf}, we obtain
\[
\begin{split}
	\left|P_x(\varphi^\epsilon_x)\right| & \leq 
	\sum_{0\leq |k|<\gamma} \left|g\left({(\cdot-x)^k}\, \varphi^\epsilon_x\right)\right|\left|
	\,\frac{\partial^k f(x)}{k!}\right|
	\lesssim \|g\|_{\cC^\beta(K)} \,  \sum_{0\leq |k|<\gamma}  \epsilon^{\beta+|k|}
	\left|\,\frac{\partial^k f(x)}{k!}\right| \\
	& \lesssim \|f\|_{\cC^\alpha(K)} \, \|g\|_{\cC^\beta(K)} \,
	\sum_{0\leq |k|<\gamma}  \epsilon^{\beta+|k|}
	\lesssim \|f\|_{\cC^\alpha(K)} \, \|g\|_{\cC^\beta(K)} \,\epsilon^\beta\, ,
\end{split}
\]
uniformly for $x$ in compact sets and $\epsilon \in (0,1]$.
This completes the proof.
\end{proof}

We can finally give the proof of \Cref{th:Young}.

\begin{proof}[Proof of \Cref{th:Young}]
We know that the germ $P$ in \eqref{eq:Pprod}
is $(\alpha,\alpha+\beta)$-coherent and has local homogeneity bound
$\beta$, by \Cref{pr:mult}.
We also know by \Cref{pr:image} that $\cR P$ belongs to $\cC^\beta$ 
(note that $\beta<\alpha+\beta$).
Since the map $P \mapsto \cR P$ is linear, and
since $P$ is a bilinear function of $(f,g)$, it follows that
we can define a bilinear map 
\[
	{\mathcal M}:\cC^{\alpha}\times\cC^\beta\to \cC^\beta, \qquad {\mathcal M}(f,g):=\cR P\, .
\]
Property \eqref{eq:propM} is a translation of
\eqref{eq:reco+1}, which characterizes $\cM$ if and only if $\alpha+\beta>0$.

Note that by \eqref{eq:cont2}
\[
	\|\cM(f,g)\|_{\cC^\beta(K)}
	\lesssim \Big( \vertiii{P}^\co_{{\bar K_4}, \varphi, \alpha, \gamma}
	+
	\vertiii{P}^\ho_{{\bar{K}_2},\varphi, \beta}\Big) \,.
\]
It follows by the estimates \eqref{eq:coherent-P}-\eqref{eq:hom-P} 
in the proof of \Cref{pr:mult} that
\[
	\vertiii{P}^\co_{{\bar K_4}, \varphi, \alpha, \gamma}
	+
	\vertiii{P}^\ho_{{\bar{K}_2},\varphi, \beta}
	\lesssim \|g\|_{\cC^\beta({\bar K_4})} \, 
	\|f\|_{\cC^\alpha({\bar K_4})} \,,
\]
which proves \eqref{eq:contcM}, hence $\cM$ is a \emph{continuous map}.
\end{proof}


\frenchspacing
\begin{thebibliography}{FdLP06}
\expandafter\ifx\csname url\endcsname\relax
  \def\url#1{\texttt{#1}}\fi
\expandafter\ifx\csname urlprefix\endcsname\relax\def\urlprefix{URL }\fi
\expandafter\ifx\csname href\endcsname\relax
  \def\href#1#2{#2}\fi
\expandafter\ifx\csname burlalt\endcsname\relax
  \def\burlalt#1#2{\href{#2}{\texttt{#1}}}\fi

\bibitem[BCD11]{BCD11}
\textsc{H.~Bahouri}, \textsc{J.-Y. Chemin}, and \textsc{R.~Danchin}.
\newblock \emph{Fourier analysis and nonlinear partial differential equations},
  vol. 343 of \emph{Grundlehren der Mathematischen Wissenschaften [Fundamental
  Principles of Mathematical Sciences]}.
\newblock Springer, Heidelberg, 2011.

\bibitem[FdLP06]{FeDe06}
\textsc{D.~Feyel} and \textsc{A.~de~La~Pradelle}.
\newblock Curvilinear integrals along enriched paths.
\newblock \emph{Electron. J. Probab.} \textbf{11}, (2006), no. 34, 860--892.
\newblock
  \burlalt{doi:10.1214/EJP.v11-356}{http://dx.doi.org/10.1214/EJP.v11-356}.

\bibitem[FH20]{Friz2020}
\textsc{P.~K. Friz} and \textsc{M.~Hairer}.
\newblock \emph{{A Course on Rough Paths: With an Introduction to Regularity
  Structures}}.
\newblock Universitext. Springer, 2 ed., 2020.

\bibitem[GIP15]{GIP15}
\textsc{M.~Gubinelli}, \textsc{P.~Imkeller}, and \textsc{N.~Perkowski}.
\newblock Paracontrolled distributions and singular {PDE}s.
\newblock \emph{Forum Math. Pi} \textbf{3}, (2015), e6, 75.
\newblock \burlalt{arXiv:1210.2684}{http://arxiv.org/abs/1210.2684}.
\newblock
  \burlalt{doi:10.1017/fmp.2015.2}{http://dx.doi.org/10.1017/fmp.2015.2}.

\bibitem[Gub04]{Gubinelli2004}
\textsc{M.~Gubinelli}.
\newblock {Controlling Rough Paths}.
\newblock \emph{J. Funct. Anal.} \textbf{216}, no.~1, (2004), 86--140.
\newblock \burlalt{arXiv:math/0306433}{http://arxiv.org/abs/math/0306433}.
\newblock
  \burlalt{doi:10.1016/j.jfa.2004.01.002}{http://dx.doi.org/10.1016/j.jfa.2004.01.002}.

\bibitem[Hai14]{Hairer2014d}
\textsc{M.~Hairer}.
\newblock {A theory of regularity structures}.
\newblock \emph{Invent. Math.} \textbf{198}, no.~2, (2014), 269--504.
\newblock \burlalt{arXiv:1303.5113}{http://arxiv.org/abs/1303.5113}.
\newblock
  \burlalt{doi:10.1007/s00222-014-0505-4}{http://dx.doi.org/10.1007/s00222-014-0505-4}.

\bibitem[HL17]{HL17}
\textsc{M.~Hairer} and \textsc{C.~Labb\'{e}}.
\newblock The reconstruction theorem in {B}esov spaces.
\newblock \emph{J. Funct. Anal.} \textbf{273}, no.~8, (2017), 2578--2618.
\newblock
  \burlalt{doi:10.1016/j.jfa.2017.07.002}{http://dx.doi.org/10.1016/j.jfa.2017.07.002}.

\bibitem[Kli67]{cf:Kli}
\textsc{A.~Klinger}.
\newblock The {V}andermonde matrix.
\newblock \emph{Amer. Math. Monthly} \textbf{74}, (1967), 571--574.
\newblock \burlalt{doi:10.2307/2314898}{http://dx.doi.org/10.2307/2314898}.

\bibitem[Lyo98]{Lyons1998}
\textsc{T.~Lyons}.
\newblock {Differential equations driven by rough signals}.
\newblock \emph{Rev. Mat. Iberoam.} \textbf{14}, (1998), 215--310.
\newblock \burlalt{doi:10.4171/RMI/240}{http://dx.doi.org/10.4171/RMI/240}.

\bibitem[MW18]{mw18}
\textsc{A.~Moinat} and \textsc{H.~Weber}.
\newblock Space-time localisation for the dynamic $\phi^4_3$ model.
\newblock \emph{{\rm to appear in} Communications on Pure and Applied
  Mathematics} (2018).
\newblock \burlalt{arXiv:1811.05764}{http://arxiv.org/abs/1811.05764}.

\bibitem[OW19]{ow19}
\textsc{F.~Otto} and \textsc{H.~Weber}.
\newblock Quasilinear {SPDE}s via rough paths.
\newblock \emph{Arch. Ration. Mech. Anal.} \textbf{232}, no.~2, (2019),
  873--950.
\newblock
  \burlalt{doi:10.1007/s00205-018-01335-8}{http://dx.doi.org/10.1007/s00205-018-01335-8}.

\bibitem[RS96]{rs96}
\textsc{T.~Runst} and \textsc{W.~Sickel}.
\newblock \emph{Sobolev spaces of fractional order, {N}emytskij operators, and
  nonlinear partial differential equations}, vol.~3 of \emph{De Gruyter Series
  in Nonlinear Analysis and Applications}.
\newblock Walter de Gruyter \& Co., Berlin, 1996.

\bibitem[ST18]{ST18}
\textsc{H.~Singh} and \textsc{J.~Teichmann}.
\newblock An elementary proof of the reconstruction theorem (2018).
\newblock \burlalt{arXiv:1812.03082}{http://arxiv.org/abs/1812.03082}.

\bibitem[Whi34]{Whi34}
\textsc{H.~Whitney}.
\newblock Analytic extensions of differentiable functions defined in closed
  sets.
\newblock \emph{Trans. Amer. Math. Soc.} \textbf{36}, no.~1, (1934), 63--89.
\newblock \burlalt{doi:10.2307/1989708}{http://dx.doi.org/10.2307/1989708}.

\bibitem[Zam21]{Zambotti20}
\textsc{L.~Zambotti}.
\newblock A brief and personal history of stochastic partial differential
  equations.
\newblock \emph{Discrete and Continuous Dynamical Systems - Series A (DCDS-A)}
  \textbf{41}, no.~1, (2021).
\newblock
  \burlalt{doi:10.3934/dcds.2020264}{http://dx.doi.org/10.3934/dcds.2020264}.

\end{thebibliography}
\end{document}